\documentclass[11pt]
{amsart}
\usepackage{amssymb,amsmath,amsthm,amsfonts,amsopn,url,color,
enumerate,mathtools,microtype,MnSymbol,comment}
\usepackage[hyperindex,breaklinks,colorlinks=true,linkcolor=red,urlcolor=green,citecolor=teal,pagebackref=true]%
{hyperref}
\usepackage[normalem]{ulem}
\usepackage{xcolor}
\usepackage[all]{xy}
\input xy
\xyoption{all}
\usepackage{amscd}
\usepackage{soul}

\theoremstyle{plain}
\newtheorem{thm}{Theorem}[section]

\newtheorem{prop}[thm]{Proposition}
\newtheorem{lemma}[thm]{Lemma}
\newtheorem{cor}[thm]{Corollary}

\theoremstyle{definition}
\newtheorem{defn}[thm]{Definition}
\newtheorem*{defn*}{Definition}
\newtheorem*{question*}{Question}

\newtheorem{example}[thm]{Example}
\newtheorem*{example*}{Example}
\newtheorem{rem}[thm]{Remark}
\newtheorem*{rem*}{Remark}
\newtheorem{notation}[thm]{Notation}

\newcommand{\field}[1]{\mathbb{#1}}
\newcommand{\N}{\field{N}}

\newcommand{\Q}{\field{Q}}

\newcommand{\ideal}[1]{\mathfrak{#1}}
\newcommand{\m}{\ideal{m}}

\newcommand{\p}{\ideal{p}}

\newcommand{\func}[1]{\mathrm{#1} \,}

\newcommand{\depth}{\func{depth}}

\newcommand{\im}{\func{im}}

\newcommand{\ra}{\rightarrow}

\DeclareMathOperator{\ann}{ann}

\DeclareMathOperator{\Hom}{Hom}

\newcommand{\RR}{\rm RR}
\newcommand{\be}{\begin{enumerate}}
\newcommand{\ee}{\end{enumerate}}

\newcommand{\li}
 {\leftfootline}
\newcommand{\lic}[2]{{#1}^{\li}_{#2}}

\newcommand{\Ri}{{-\rm{Rs}}}

\newcommand{\EHUi}{{-\rm{EHU}}}
\newcommand{\EHUic}[2]{{#1}^\EHUi_{#2}}
\newcommand{\onto}{\twoheadrightarrow}

\newcommand{\into}{\hookrightarrow}

\newcommand{\cC}{\mathcal{C}}
\newcommand{\cI}{\mathcal{I}}
\newcommand{\cD}{\mathcal{D}}

\newcommand{\cM}{\mathcal{M}}
\newcommand{\cP}{\mathcal{P}}
\newcommand{\cQ}{\mathcal{Q}}
\newcommand{\cR}{\mathcal{R}}
\renewcommand{\phi}{\varphi}

\DeclareMathOperator{\Sym}{Sym}

\newcommand{\opsub}{order-preserving on submodules}

\newcommand{\dual}{\smallsmile}
\newcommand{\cl}{{\mathrm{cl}}}
\let\int\relax
\DeclareMathOperator{\int}{i}

\DeclareMathOperator{\ab}{a}
\DeclareMathOperator{\re}{r}
\DeclareMathOperator{\ch}{c}
\DeclareMathOperator{\h}{h}
\newcommand{\po}[2]{{#1}_{#2}}

\newcommand{\rcl}[2]{{#1}^{\po {\cl} {\re}}_{#2}}

\DeclareMathOperator{\core}{-core}

\newcommand{\Jcolsym}[1]{{#1} \rm{bf}}

\newcommand{\Jcol}[3]{{#2}^{\Jcolsym #1}_{#3}}

\newcommand{\Jintrelsym}[1]{{#1}{\rm be}}

\newcommand{\Jintrel}[3]{{#3}_{\Jintrelsym{#1}}^{#2}}

\newcommand{\fg}{finitely generated}
\newcommand{\charp}{characteristic $p>0$}

\newcommand{\cvl}{\color{violet}}

\author{Neil Epstein}
\address{Department of Mathematical Sciences \\ George Mason University \\ Fairfax, VA  22030}
\email{nepstei2@gmu.edu}

\author{Rebecca R.G.}
\address{Department of Mathematical Sciences \\ George Mason University \\ Fairfax, VA  22030}
\email{rrebhuhn@gmu.edu}

\author{Janet Vassilev}
\address{Department of Mathematics and Statistics \\ University of New Mexico \\ Albuquerque, NM 87131}
\email{jvassil@math.unm.edu}

\title[How to extend closure and interior operations]{How to extend closure and interior operations to more modules}

\date{November 3, 2023}
\begin{document}
\begin{abstract}
There are several ways to convert a closure or interior operation to a different operation that has particular desirable properties.  In this paper, we axiomatize 3 ways to do so, drawing on disparate examples from the literature, including tight closure, basically full closure, and various versions of integral closure.  In doing so, we explore several such desirable properties, including \emph{hereditary}, \emph{residual}, and \emph{cofunctorial}, and see how they interact with other properties such as the \emph{finitistic} property.
\end{abstract}

\maketitle
\setcounter{tocdepth}{2} 
\tableofcontents

\section{Introduction}

In extending a closure operation on ideals to a closure operation on modules, there are three main techniques that appear in the literature:
\begin{enumerate}
    \item (Hereditary version) Start by defining the closure operation on submodules of injective modules. Given $L \subseteq M$, embed $M$ in an injective module $E$. To compute the closure of $L$ in $M$, compute the closure of $\im(L) \subseteq E$, then intersect with $M$. (See Section \ref{sec:hereditaryversions}.)
    \item (Residual version) Start by defining the closure operation on submodules of free modules. Given $L \subseteq M$, find a free module $F$ that surjects onto $M$. To compute the closure of $L$ in $M$, take the closure of $\text{pre-im}(M)$ in $F$ and project it back to $M$. (See Section \ref{sec:residualversion}.)
    \item (Free module pre-enveloping class version) 
    Start by defining the closure operation on submodules of free modules. Given $L \subseteq M$, consider all maps from $M$ to free modules $F$. To compute the closure of $L$ in $M$, compute the closure of $\im(L)$ in all such $F$, then take the intersection of the pre-images in $M$. (See Section \ref{sec:preenveloping}.)
\end{enumerate}

For example, Rees's extension of integral closure to finitely generated modules over a domain in \cite{reesintegralclosure} takes the hereditary approach modelling the version of integral closure for ideals defined by extending to valuation rings in the fraction field and contracting back to the ring. Whereas the integral closure defined by Eisenbud, Huneke and Ulrich \cite{EHU-Ralg} uses the free module pre-enveloping class version modelling the version of integral closure for ideals defined through the Rees algebra of the ideal. Motivated by the study of tight closure, Epstein and Ulrich \cite{nmeUlr-lint} have shown there is a unique residual version of integral closure, which they call the liftable integral closure. These versions of integral closure are known to disagree on examples \cite{EHU-Ralg,nmeUlr-lint}.  
We do have a few results comparing these closures, such as that Rees integral closure is contained in EHU-integral closure, with equality if and only if the module is a submodule of a free module or the ring is a domain \cite{nmeUlr-lint}. But this doesn't provide a framework for comparing the closures in general.

In contrast to the work done for integral closure, extensions of tight closure and Frobenius closure to submodules of finitely generated modules have for the most part been defined through the residual approach as this has been fruitful in the study of singularities.  However, in \cite[Section 8]{HHmain} Hochster and Huneke did introduce a hereditary version of tight closure that they called the absolute tight closure.

Very few closure operations are both hereditary and residual, though module closures coming from flat modules are \cite{dietzclosureprops}. Research using residual operations (tight closure, Frobenius closure, plus closure, module closures, see for example \cite{epstein})
tends to focus on the study of singularities of Noetherian rings. Research using hereditary operations (star operations, semi-prime operations, see for example \cite{semiprimeoperations,staroperations,elliottclosureops}) tends to focus on the study of fractional ideals over integral domains. Recent work of Elliott \cite[Chapter 5]{elliottclosureops} extends results on these classes of hereditary closure operations to more general classes of modules, enabling a more direct comparison between residual and hereditary operations. The various versions of integral closure straddle the divide, coming up in both contexts.

Our prior work built a foundation for exploring these three methods in greater generality. In \cite{nmeRG-cidual}, the first two authors described a duality between residual closure operations $\cl$ and interior operations $\int$ over complete local rings. This duality assigns to each residual closure operation a dual interior operation, such that taking the dual twice returns the original residual closure operation. We additionally showed that a number of structures in commutative algebra not usually viewed as closure operations or interior operations actually satisfy the conditions needed to use this duality, including variations on trace modules, torsion submodules, divisible submodules, and 0th local cohomology modules.

In \cite{ERGV-nonres} we defined a new version of duality that works on any \emph{pair operation}, a notion that encompasses both closure and interior operations.  In particular, we used this duality to find the duals of non-residual closures, which are \emph{relative} interior operations. A relative interior operation takes a pair of modules $N \subseteq M$ to a submodule $N_{\int}^M$ of $N$, and in contrast to an (absolute) interior operation, $N_{\int}^M$ may depend on both $N$ and $M$. When the closure operation is residual, this new duality agrees with the duality of \cite{nmeRG-cidual}, and the resulting interior operation is absolute, meaning that the interior of a module $N$ does not depend on an ambient module $M$ containing $N$.
This enabled us to compute integral hulls of submodules of the injective hull of the residue field over a local ring, using the duality to cores of ideals in the ring. 

In this paper we generalize the three methods of extending a closure operation to the context of pair operations, as well as discussing how to convert a closure operation to a hereditary or residual one. In Section \ref{sec:background}, we recall the relevant properties of pair operations that we introduced in \cite{ERGV-nonres} along with a number of versions of integral closure and tight closure. In Section  \ref{sec:pairopssmselector},  we discuss some new properties for pair operations.  In particular, we define  \emph{cofunctoriality} for pair operations, 
and show that it is 
the dual of functoriality for pair operations in Proposition \ref{pr:functcofunct}.  However, when the pair operation is order preserving on submodules as with closure or interior operations, cofunctorial and functorial pair operations are equivalent (See Proposition \ref{pr:cofunctorial}). 
In Section \ref{sec:pairopssmselector}, we also define the dual notions of hereditary and cohereditary pair operations, pair operations arising from submodule selectors, as well as finitistic versions of pair operations.

In Section \ref{sec:residualversion}, from any pair operation $p$, we 
construct a cohereditary version $\po{p}{\ch}$. When the pair operation is a closure operation, this becomes a residual version.  In Section \ref{sec:hereditaryversions}, from a pair operation $p$ we 
construct a hereditary version $\po{p}{\h}$. In Section \ref{sec:duality}, we show that when the ring $(R, \m )$ is a complete local ring, then the cohereditary and hereditary versions of pair operations that we constructed in Sections \ref{sec:residualversion} and \ref{sec:hereditaryversions} are dual to each other. In Section \ref{sec:JbfJbe}, we continue the study of $J$-basically full closures and $J$-basically empty interiors 
we began in \cite{ERGV-nonres}, by discussing the residual version of the $J$-basically full closure (which is a hereditary closure) and the absolute version of the $J$-basically empty interior (which is a cohereditary interior). In Section \ref{sec:preenveloping}, we introduce pair operations defined by pre-enveloping classes. In Section \ref{sec:integralclosure}, we analyze known closures such as integral closure and tight closure through the lens of pre-enveloping classes.
In \cite{ERGV-nonres}, for any ideal $J$, we discussed the $\Jcolsym J$-closure and introduced the $\Jintrelsym{J}$-interior. In Propositions \ref{jcolhereditary} and \ref{pr:jbecohereditary} we show that $\Jcolsym J$ is  hereditary and $\Jintrelsym J$ is cohereditary.  Then in  Propositions \ref{pr:EHUishereditary} and \ref{pr:fgaspreenveloping}, we show that EHU-integral closure is hereditary, and more generally, that closures constructed through pre-enveloping classes are hereditary. 




\section{Background}\label{sec:background}

\subsection{Pair Operations}

In this section we define pair operations and describe some of their properties. The purpose of using a pair operation is to have a structure that generalizes both closure operations and interior operations.

\begin{defn}[{\cite[Definition 2.2]{ERGV-nonres}}]
\label{def:pairop}
Let $\cM$ be a category of $R$-modules. Let $\cP$ be a collection of pairs $(L,M)$, where $L$ is a submodule of $M$, and $L,M \in \cM$, 
such that whenever $\phi:M \to M'$ is an isomorphism in $\cM$ and $(L,M) \in \cP$, $(\phi(L),M') \in \cP$ as well.

A \emph{pair operation} is a function $p$ that sends each pair $(L,M) \in \cP$ to a submodule $p(L,M)$ of $M$, in such a way that whenever $\phi: M \ra M'$ is isomorphism in $\cM$
 and $(L,M) \in \cP$, then 
  $\phi(p(L,M)) = p(\phi(L),M')$.

A pair operation $p$ on a class $\cP$ of pairs of $R$-modules $(L,M)$ as above is
\begin{itemize}
    \item \emph{idempotent} 
    if whenever $(L,M) \in \cP$ and $(p(L,M), M) \in \cP$, we always have $p(p(L,M),M)=p(L,M)$;
    \item \emph{extensive} 
    if we always have $L \subseteq p(L,M)$;
    \item \emph{intensive} if we always have $p(L,M) \subseteq L$;
    \item \emph{order-preserving on submodules} if whenever $L \subseteq N \subseteq M$ such that $(L,M), (N,M) \in \cP$, we have $p(L,M) \subseteq p(N,M)$;
    \item \emph{order-preserving on ambient modules} 
    if whenever  $L \subseteq N \subseteq M$ 
    such that $(L,N), (L,M) \in \cP$,
    we have $p(L,N) \subseteq p(L,M)$;
    \item \emph{surjection-functorial} 
    if whenever $\pi:M \twoheadrightarrow M'$ is a surjection and $(L,M),(\pi(L),M') \in \cP$, we have \[\pi(p(L,M)) \subseteq p(\pi(L),M').\]  Equivalently,  when $(L,M) \in \cP$ and for 
    $U \subseteq M$, $((L+U)/U,M/U) \in \cP$, then \[(p(L,M)+U)/U \subseteq p((L+U)/U, M/U);\] 
    \item \emph{functorial} if whenever $g: M \ra M'$ and $(L,M),(g(L),M') \in \cP$, we have \[g(p(L,M)) \subseteq p(g(L), M').\] (Note that if $(g(L),g(M))$ is also in $\cP$, it is equivalent that $p$ be both  order-preserving on ambient modules and surjection-functorial, by the usual epi-monic factorization); 
    \item a \emph{closure operation} if it is extensive, order-preserving on submodules, and idempotent;
    \item a \emph{(relative) interior operation} if it is intensive, order-preserving on submodules, and idempotent;
    \item \emph{restrictable} if whenever $L,N \subseteq M$ such that $(L \cap N,N),(L,M) \in \cP$, $p(L \cap N,N) \subseteq p(L,M)$.
\end{itemize}
\end{defn}

\begin{rem}
In \cite[5.5.10]{elliottclosureops}, Elliott introduces a notion related to that of the pair operation, namely a ``system of (order-preserving) operations''.  While we take the requirement of isomorphism-preservation as basic, he names this property ``invariant'', while on the other hand he takes the requirement of order-preservation (on submodules) as basic.  With this in mind, we may compare our terminology to his.  For instance, our definition of \emph{idempotent} is a generalization of his ``left semiexact''.  His notion of ``left functorial'' becomes equivalent to our notion of \emph{order-preserving on ambient modules}, given our requirement to be invariant in his sense.  And his notion of ``right functorial'' matches our \emph{surjection-functorial}.
\end{rem}

\begin{rem}
We severely limit the hypotheses on $\cP$ so that our results cover pair operations not only on modules or \fg\ modules, but also more restrictive classes like ideals, graded modules (with graded homomorphisms), or $\m$-primary modules.
Hypotheses like ``closed under taking submodules" or ``closed under taking quotient modules" would be too strong for our purposes:
Consider \[\cP_1=\{(N,M)\mid  N \text{ is } \m\text{-primary in } M\}.\]  For any $L \subseteq N \subseteq M$, with $L$ not $\m$-primary in $N$, but $N$ $\m$-primary in $M$, $(L,M)$ and $(L,N)$ will not be in $\cP_1$ however $(N,M),(N/L,M/L) \in \cP_1$.  Whereas for \[\cP_2=\{(N,M)\mid  N \text{ and } M\   \m\text{-primary modules with }N \subseteq M\}\] if $(N,M) \in \cP_2$ and $R$ has dimension at least 1, $(N/N,M/N) \notin \cP_2$ even if $(N,M)$ is.
\end{rem}



We recall the definition of the dual of a pair operation from \cite{ERGV-nonres}:

\begin{defn}[{\cite[Definition 3.1]{ERGV-nonres}}]
\label{def:smiledual}
Let $R$ be a complete local ring and let $^\vee$ denote the Matlis duality operator. Let $p$ be a pair operation on a class of pairs of Noetherian and/or Artinian $R$-modules $\cP$ as in Definition \ref{def:pairop}. Set $\cP^\vee := \{(A,B) \mid ((B/A)^\vee, B^\vee) \in \cP\}$, where $(B/A)^\vee$ is seen as a submodule of $B^\vee$ via the dual of the projection map $\pi:B \onto B/A$.
It's an easy exercise to show that $(\cP^\vee)^\vee=\cP$.

We define the dual $p^\dual$ of $p$ by
\[p^\dual(A,B):= \left(\frac{B^\vee}{p((B/A)^\vee,B^\vee)} \right)^\vee. \]
\end{defn}

\begin{rem}\label{rem:Matlis}
    Recall that Matlis duality ${}^\vee$ is an exact contravariant functor, and hence preserves short exact sequences, turning injections into surjections and vice-versa.  Moreover,  for Noetherian or Artinian $R$-modules $M$, there is a functorial isomorphism $M \cong M^{\vee\vee}$. Throughout the paper, we will view submodules of $M^{\vee\vee}$ as submodules of $M$ via this isomorphism.
\end{rem}

The next definition will come into play when we discuss $\cl$-cores and $\int$-hulls later.

\begin{defn}[{\cite[Definitions 2.11 and 3.11] {ERGV-nonres}}]
\label{def:nakayama}
Let $(R,\m)$ be a Noetherian local ring.

Let $\cl$ be a closure operation on the class of pairs of \fg\ $R$-modules.  We say that $\cl$ is a \emph{Nakayama closure} if for $L \subseteq N \subseteq M$ \fg\ $R$-modules, if $L\subseteq N \subseteq (L+\m N)^{\cl}_M$ then $L^{\cl}_M=N^{\cl}_M$.

Let $\int$  be a (relative) interior operation on the class of pairs of Artinian $R$-modules. We say that $\int$ is a \emph{Nakayama interior} if for any Artinian $R$-modules $A \subseteq C \subseteq B$, if 
$(A:_C \m)^B_i\subseteq A$, then $A^B_{\int} = C^B_{\int}$ (or equivalently, $C^B_{\int} \subseteq A$).
\end{defn}

\subsection{Versions of integral and tight closure}

We define the versions of integral closure and tight closure discussed in this paper.

Integral closure is defined in a standard way on ideals. Most extensions of integral closure to modules use the definition below for submodules of free modules:

\begin{defn}[Integral closure in a free module]\label{def:icfree}  Let $R$ be a Noetherian ring and $L \subseteq F$
be $R$-modules where $F$ is free.  Let $\Sym(F)$ be the naturally graded symmetric algebra over $R$ defined by $F$ and $T$ be the subring of $\Sym(F)$ induced by the inclusion of $L \subseteq F$. Note that both $\Sym(F)$ and $T$ are $\mathbb{N}$-graded rings generated in degree 1 over $R$.  The integral closure of $L$ in $F$, denoted $L_F^{-}$ is the degree 1 part of the integral closure of the subring $T$ in $\Sym(F)$.
\end{defn}

The primary versions of integral closure that we will focus on in this paper are EHU-integral closure and liftable integral closure, both of which are functorial. We define them below.

\begin{defn}[{\cite[Definition 1.3]{nmeUlr-lint}, derived from \cite{EHU-Ralg}}]
\label{def:EHU}

(EHU-integral closure) Let $R$ be a Noetherian ring and $L \subseteq M$ be finite $R$-modules.  For $x \in M$, we say $x$ is in the $EHU$-integral closure of $L$ in $M$,  written $x \in L_M^{\EHUi}$ if for every $R$-module homomorphism $g:M \rightarrow F$ where $F$ is free we have $g(x) \in g(L)_F^{-}$. 
\end{defn}

\begin{defn}[\cite{nmeUlr-lint}] \label{def:li}
 (Liftable integral closure)  Let $L \subseteq M$ be $R$-modules. Let $\pi:F \twoheadrightarrow M$ be a surjection of a free $R$-module $F$ onto $M$. Let $K:=\pi^{-1}(L)$. Then the liftable integral closure of $L$ in $M$ is $\lic LM:=\pi(K_F^{-}).$
\end{defn}

\cite[Proposition 2.4(2)]{nmeUlr-lint} shows liftable integral closure is functorial. We will show that EHU-integral closure is functorial in Theorem \ref{thm:EHUic}.
We will also briefly mention Rees integral closure:

\begin{defn}[\cite{reesintegralclosure}]
(Rees integral closure for Noetherian domains) Let $R$ be a Noetherian domain and $L \subseteq M$ be $R$-modules.  We say $x \in M$ is in the Rees integral closure of $L$ in $M$ denoted $x \in L_M^{\Ri}$ if for every valuation ring $V$ between $R$ and its fraction field $Q$, $x$ is in the image $LV$ of $L \otimes_R V$ in $M \otimes_R Q$.
\end{defn}

\begin{defn}[\cite{reesintegralclosure}]
(Rees integral closure for Noetherian rings) Let $R$ be a Noetherian ring and $L \subseteq M$ be $R$-modules.  We say $x \in M$ is in the Rees integral closure of $L$ in $M$ denoted $x \in L_M^{\Ri}$ if for every minimal prime $\mathfrak{p}$, $x+\mathfrak{p}M \in \left(\displaystyle\frac{L+\mathfrak{p}M}{\mathfrak{p}M} \right)_{M/\mathfrak pM}^{\rm -Rs}$ as modules over the ring $R/\mathfrak{p}$.
\end{defn}

We show below that this closure is functorial.

\begin{prop}
Let $R$ be a Noetherian ring. Then Rees integral closure is functorial, i.e. if $R$ is a Noetherian ring, $L \subseteq M$ and $N$ are $R$-modules, and $f:M \to N$ is an $R$-module map, then $f(L_M^{\rm -Rs}) \subseteq f(L)_N^{\rm -Rs}$.
\end{prop}

\begin{proof}
First assume that $R$ is a domain and let $x \in L_M^{\rm -Rs}$. Let $V$ be a valuation ring between $R$ and its fraction field $Q$. Then $x$ (or $x \otimes 1$) is in the image of $\phi:L \otimes_R V \to M \otimes_R V$, say $x=\phi(z)$. Let $\psi:f(L) \otimes_R V \to N \otimes_R V$. Then $\psi(f(z))=f(x)$ (where $f$ is really $f \otimes$ the appropriate identity map in each place). So $f(x)$ is in the image of $\psi$, as desired.

Next let $R$ be any Noetherian ring and let $x \in L_M^{\rm -Rs}$. We want to show that for every minimal prime $\mathfrak{p}$ of $R$, $f(x)+\mathfrak{p}N \in \left(\displaystyle\frac{f(L)+\mathfrak{p}N}{\mathfrak{p}N} \right)_{N/\mathfrak{p}N}^{\rm -Rs}$. We know that $x+\mathfrak{p}M \in \left(\displaystyle\frac{L+\mathfrak{p}M}{\mathfrak{p}M} \right)_{M/\mathfrak{p}M}^{\rm -Rs}$. Since we are now working over a domain and functoriality holds in the domain case, we get
\[f\left(\left(\displaystyle\frac{L+\mathfrak{p}M}{\mathfrak{p}M} \right)_{M/\mathfrak pM}^{\rm -Rs}\right) \subseteq \left(\displaystyle\frac{f(L)+\mathfrak{p}N}{\mathfrak{p}N} \right)_{N/\mathfrak pN}^{\rm -Rs}.\]
So 
\[\bar{f}(x+\mathfrak{p}M)=f(x)+\mathfrak{p}N \in \left(\displaystyle\frac{f(L)+\mathfrak{p}N}{\mathfrak{p}N} \right)_{N/\mathfrak pN}^{\rm -Rs}=f(L)_N^{\rm -Rs}. \qedhere\]
\end{proof}

We note that as a result of the first author and Ulrich \cite[Corollary 1.5]{nmeUlr-lint} $L_M^{\Ri} \subseteq L_M^{\EHUi}$ with equality if $M$ is free or $R$ is a domain. In the same paper, there is an example given where $R$ is not a domain and the inequality is strict.

Simis, Ulrich, and Vasconcelos also define a version of integral closure on modules with rank \cite{SimisUlrichVasconcelos}. They note that if $M$ is a module with rank and is a torsionfree submodule of a free module, then the Rees algebra of $M$ agrees with Rees's definition of Rees algebra of $M$.  It seems that in EHU, they note that the EHU construction of the Rees algebra agrees with the Rees algebra when the module is an ideal.  It seems that this would imply that the EHU-integral closure of an ideal $J$ with rank inside another ideal $I$ agrees with the SUV-integral closure of $J$ in $I$. Hence we do not consider this version of integral closure in this paper.

We also give the definition of absolute tight closure, originally from \cite{HHmain}:

\begin{defn}[{\cite[Section 8]{HHmain}}]
\label{def:abstightclosure}
Let $R$ be a Noetherian ring of \charp\ and $N \subseteq M$. We define $N_M^{*fg}$, the \textit{finitistic tight closure} of $N$ in $M$, to be 
\[N_M^{*fg}:=\bigcup_{M' \subseteq M, M' \text{ f.g.}} (M' \cap N)_{M'}^*,\]
where $*$ denotes tight closure.

We define $N_M^{*abs}$, the \textit{absolute tight closure} of $N$ in $M$, to be
\[N_M^{*abs}=\{x \in M : x \in N_Q^{*fg} \text{ for some } Q \supseteq M\}.\]
\end{defn}

\section{New properties for pair operations}\label{sec:pairopssmselector}

We give some new properties for pair operations and show how they are related. These properties will be the focus of this paper.
\subsection{Functorial and cofunctorial pair operations} \label{subs:funccofunc}
Although not all closure operations are functorial, functoriality is an important assumption to discuss a closure on different rings.   We introduce cofunctoriality, the dual notion of functoriality and prove that these two properties are equivalent for pair operations which are order preserving on submodules.
\begin{defn}
\label{def:cofunct}
A pair operation $p$ on a class $\cP$ of pairs of $R$-modules $(L,M)$ with $L \subseteq M$ is \begin{itemize}
    \item \emph{surjection-cofunctorial} if whenever $\pi: M \ra M'$ is a surjective homomorphism and $(L,M'), (\pi^{-1}(L),M) \in \cP$, we have \[p(\pi^{-1}(L),M) \subseteq \pi^{-1} (p(L, M')).\]
    \item  \emph{cofunctorial} if whenever $g: M \ra M'$ and $(L,M'), (g^{-1}(L),M) \in \cP$, we have 
    \[p(g^{-1}(L),M) \subseteq g^{-1} (p(L, M')).\]
\end{itemize}
\end{defn}

 For closure and interior operations, we will see that functoriality and cofunctoriality are interchangeable under mild hypotheses (see Proposition~\ref{pr:cofunctorial}(\ref{it:pairopsub})).  However, this is not the case with pair operations. We will note in Example \ref{ex:RRop} and Remark \ref{rem:dualRRop}, that the Ratliff-Rush operation 
is a pair operation where functoriality and cofunctoriality can differ.

 \begin{rem}
One might wonder why we have not included the definitions of \emph{injection-functorial} in Definition~\ref{def:pairop} (i.e. functorial for injective homomorphisms) or \emph{injection-cofunctorial} in Definition~\ref{def:cofunct} (i.e. cofunctorial for injective homomorphisms).  The reason for this is that they are equivalent to notions that we have already defined, namely \emph{order preserving on ambient modules} and \emph{restrictable}, respectively.  We prove these notions are equivalent below.
\end{rem}

\begin{lemma}\label{lem:injfuncequiv}
Let $p$ be a pair operation on $\cP$.  
\begin{enumerate}
    \item \label{it:orderambient} $p$ is order preserving on ambient modules if and only if for all injective homomorphisms $i:M \rightarrow N$ and all submodules $L \subseteq M$ such that $(L,M), (i(L),N) \in \cP$, we have $i(p(L,M)) \subseteq p(i(L),N)$.
    \item \label{it:restrictable} $p$ is restrictable if and only if for all injective homomorphisms $i:N \rightarrow M$ and all submodules $L \subseteq M$ such that $(L,M), (i^{-1}(L),N) \in \cP$, we have $p(i^{-1}(L),N) \subseteq i^{-1}(p(L,M))$.
\end{enumerate}
\end{lemma}
\begin{proof}
\eqref{it:orderambient}:  If $p$ is order preserving on ambient modules, then for arbitrary $(L,N), (L,M) \in \cP$ with $L \subseteq N \subseteq M$, $p(L,N) \subseteq p(L,M)$.  Suppose $i:K \to M$ is an injective homomorphism and $L \subseteq K$ such that $(L,K),(i(L),M) \in \cP$. Since $i$ restricts to an isomorphism between $L$ and $i(L)$ (and between $K$ and $i(K)$), $(i(L),i(K)) \in \cP$. Since $p$ is order preserving on ambient modules, $p(i(L),i(K)) \subseteq p(i(L),M)$. But $i(p(L,K))=p(i(L),i(K))$ since $i$ is an isomorphism on $K$. So $i(p(L,K)) \subseteq p(i(L),M)$, as desired.

Now assume the condition:  for all injective homomorphisms $i:M \rightarrow N$ and all submodules $L \subseteq M$ such that $(L,M), (i(L),N) \in \cP$, we have $i(p(L,M)) \subseteq p(i(L),N)$.  If $M \subseteq N$, setting $i$ be the inclusion map of $M$ into $N$, we obtain
 $p(L,M)=i(p(L,M)) \subseteq p(i(L),N)=p(L,N)$ which implies that $p$ is order preserving on ambient modules.

\eqref{it:restrictable}:  Suppose $p$ is restrictable. Let $i: N \ra M$ be an injective homomorphism and $L \subseteq M$ such that $(L,M),(i^{-1}(L),N) \in \cP$. Since pair operations are invariant under isomorphisms, we may assume that $i$ is an inclusion map. Then $i^{-1}(L)=L \cap N$. So
\[p(i^{-1}(L),N)=p(L \cap N,N) \subseteq p(L,M).\]
Since $p(L \cap N,N)$ is by definition contained in $N$ as well, 
\[p(i^{-1}(L),N) \subseteq p(L,M) \cap N=i^{-1}(p(L,M)).\]

Now suppose that $p$ satisfies the condition given in the statement of the Lemma. Let $L \subseteq N \subseteq M$. Let $i:N \to M$ be the inclusion. Then
\[p(L \cap N,N)=p(i^{-1}(L),N) \subseteq i^{-1}(p(L,M)=p(L,M) \cap N,\]
so $p$ is restrictable.
\end{proof}

\begin{lemma}\label{lem:matlisimagepreimage}
Let $(R,\m)$ be a complete local ring.  Let $M$ and $N$ both be finite or Artinian $R$-modules and $g:M \ra N$ be an $R$-module homomomorphism.  Define $\phi=g^\vee: N^\vee \ra M^\vee$.
\begin{enumerate}
    \item 	\label{it:phipreimage}If $L \subseteq M$ is a submodule, then $\phi^{-1}((M/L)^\vee) = (N/g(L))^\vee$
    \item \label{it:phiimage} Let $K \subseteq N$ be a submodule, then $\phi((N/K)^\vee) = (M / g^{-1}(K))^\vee$

\end{enumerate}
\end{lemma}
\begin{proof}
\eqref{it:phipreimage} First assume $x \in (N/g(L))^\vee$. Then $x:N \to E$ and $g(L) \subseteq \text{ker}(x)$. So $\phi(x)=x \circ g$ is a map from $M$ to $E$, and its kernel contains $L$ because $(x \circ g)(L)=x(g(L))=0.$ 

Now assume $x \in \phi^{-1}((M/L)^\vee)$. Then $\phi(x)=x \circ g$ is a map $M$ to $E$ such that $(x \circ g)(L)=0$. Then $x(g(L))=0$, so $x \in (N/g(L))^\vee$.

\eqref{it:phiimage}  Assume $z \in (N/K)^\vee$. Then $z:N \ra E$ and $z(K)=0$. For $u \in g^{-1}(K)$, $g(u) \in K$, so $\phi(z)(u)=(z \circ g)(u)=0$ implying $\phi(z) \in (M/(g^{-1}(K))^\vee$. 

Now we need to show that $(M/g^{-1}(K))^\vee \subseteq \phi(N/K)^\vee$.  Let $x \in M^\vee$ such that $x$ kills $g^{-1}(K)$. We need to show there is some $y \in N^\vee$ that kills $K$ and such that $x=y\circ g$.
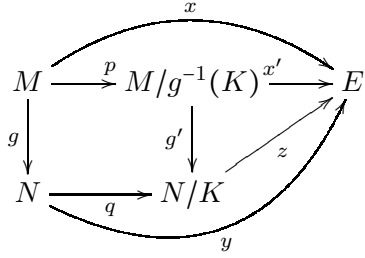
\[ \xymatrix{ M \ar[r]^p \ar[d]_g \ar@/^2pc/[rr]^x & M/g^{-1}(K) \ar[r]^{x'} \ar[d]_{g'} & E \\
N \ar[r]_q \ar@/_3pc/[rru]_y & N/K \ar[ur]_{z} & \\
}\]

Write $x=x' \circ p$, where $p: M \onto M/g^{-1}(K)$ is the natural map and $x': M/g^{-1}(K) \ra E$ is the induced map from the first isomorphism theorem.  Let $q: N \onto N/K$ be the quotient map.  Since $g^{-1}(K)$ is the kernel of the composite $q \circ g$, the first isomorphism theorem gives us an induced injective map $g': M/g^{-1}(K) \ra N/K$ such that $g' \circ p = q \circ g$.  Since $g'$ is an injective map and $E$ is an injective module, the map $x'$ can be extended across $g'$ to a map $z: N/K \ra E$.  That is, $z \circ g' = x'$.  But then $z\circ q \circ g = z \circ g' \circ p = x' \circ p = x$.  So if we set $y=z \circ q$, then $\phi(y)=x$ and $y \in (N/K)^\vee$.  
\end{proof}

In the following proposition, we include \eqref{it:iparatoifunc} and \eqref{it:injsurjfunc} for completeness.
\begin{prop}\label{pr:cofunctorial}
Let $p$ be a pair operation on a class $\cP$ of pairs of $R$-modules.  
\begin{enumerate}
    \item \label{it:sfunctospara} If $p$ is surjection-functorial then $p$ is surjection-cofunctorial.
    \item \label{it:iparatoifunc} \cite[Lemma 2.5 (1)]{ERGV-nonres} If $p$ is restrictable then $p$ is order preserving on ambient modules.
    \item \label{it:injsurjfunc} \cite[Definition 2.2]{ERGV-nonres} If $p$ is functorial then it is surjection-functorial and order preserving on ambient modules. The converse holds if whenever $g: M \ra M'$ is in $\cM$ and $(L,M) \in \cP$, we have $(g(L), g(M)) \in \cP$.
    \item \label{it:injsurjcofunc} If $p$ is cofunctorial then it is surjection cofunctorial and restrictable. The converse holds  if whenever $g:M \to M'$ is a homomorphism, $L \subseteq M'$, $(L,M'),(g^{-1}(L),M) \in \cP$, and we can write $g=i \circ \pi$ where $\pi$ is surjective and $i$ is injective, then $(i^{-1}(L),\pi(M)) \in \cP$.
    \item\label{it:pairopsub} If $p$ is order preserving on submodules, then 
    \begin{enumerate}
        \item \label{it:sfunceq} if for any surjection $\pi:M \to M'$ and $L \subseteq M$ such that $(L,M)$, $(\pi(L),M') \in \cP$, we have $(\pi^{-1}(\pi(L)),M) \in \cP$, then $p$ is surjection-functorial if and only if $p$ is surjection-cofunctorial. 
        \item \label{it:ifunceq} \cite[Lemma 2.5 (2)]{ERGV-nonres} 
        {if when $(L \cap N,N), (L,M) \in \cP$ we also have $(L \cap N,M) \in \cP$}, then $p$ is order preserving on ambient modules if and only if $p$ is restrictable.
        \item \label{it:funceq} If $p$ is functorial then $p$ is cofunctorial as long as when $(L \cap N,N),(L,M) \in \cP$, $(L \cap N,M)$ is also in $\cP$ and whenever $g:M \to M'$ is a homomorphism, $(L,M'),(g^{-1}(L),M) \in \cP$ and we can write $g=i \circ \pi$ where $\pi$ is surjective and $i$ is injective, then $(i^{-1}(L),\pi(M)) \in \cP$.  The converse holds if whenever $g: M \ra M'$ is in $\cM$ and $(L,M) \in \cP$, we have $(g(L), g(M)) \in \cP$, and whenever $\pi:M \to M'$ is a surjection with $(L,M),(\pi(L),M') \in \cP$, $(\pi^{-1}(\pi(L)),M) \in \cP$.
    \end{enumerate}
\end{enumerate}
\end{prop}

\begin{rem}
    In \cite[Lemma 2.5 (2)]{ERGV-nonres}, the hypothesis on $\cP$ is stated incorrectly and should match the hypothesis given in the statement of 5b above.
\end{rem}

\begin{rem}
    We note that the additional hypotheses on $\cP$ required for some parts of this result tend to hold for most choices of $\cP$ found in the literature. For example, if $\cP$ consists of all pairs of $R$-modules, \fg\ $R$-modules, \fg\ and Artinian $R$-modules, or graded $R$-modules (with graded homomorphisms), all of the extra hypotheses hold. Or more generally, if $\cM$ is an abelian category and $\cP$ consists of all nested pairs of modules in $\cM$, then all of the extra hypotheses hold.
\end{rem}

\begin{proof}[Proof of Proposition~\ref{pr:cofunctorial}]
\eqref{it:sfunctospara}  Suppose that $p$ is surjection-functorial.  Suppose $\pi:M \rightarrow M'$ is a surjective homomorphism of $R$-modules, $L \subseteq M'$, and $(L,M'),(\pi^{-1}(L),M) \in \cP$.  Since $\pi$ is surjective, we know that $\pi(\pi^{-1}(L))=L$.  Since $p$ is surjection-functorial and $(\pi^{-1}(L),M),(\pi(\pi^{-1}(L)),M')=(L,M') \in \cP$, we have \[\pi(p(\pi^{-1}(L),M)) \subseteq p(\pi(\pi^{-1}(L)),M')=p(L,M').\]
By definition of the preimage,
\[p(\pi^{-1}(L),M) \subseteq \pi^{-1}(p(L,M')),\] implying $p$ is surjection-cofunctorial.

\eqref{it:injsurjcofunc} The forward direction is immediate from Lemma \ref{lem:injfuncequiv}\eqref{it:restrictable}. For the reverse direction, suppose that $p$ is surjection cofunctorial and restrictable. Let $g:M \to M'$ be an arbitrary map, and $L \subseteq M'$ such that $(L,M'), (g^{-1}(L),M) \in \cP$. We can write $g$ as a composite
\[M \xrightarrow{\pi} M/\ker(g) \xrightarrow{i} M'.\] By our hypotheses, $(i^{-1}(L),M/\ker(g)) \in \cP$.
Then 
\begin{align*}
    p(g^{-1}(L),M) &=p((i \circ \pi)^{-1}(L),M)=p(\pi^{-1}(i^{-1}(L)),M) \\
    &\subseteq \pi^{-1}(p(i^{-1}(L),M/\ker(g)) \subseteq \pi^{-1}\left(i^{-1}(p(L,M'))\right) \\
    &=(i \circ \pi)^{-1}(p(L,M'))=g^{-1}(p(L,M')).
\end{align*}
 
\eqref{it:sfunceq}  It is enough to show that if $p$ is order preserving on submodules and surjection-cofunctorial then $p$ is surjection-functorial by \eqref{it:sfunctospara}.  Suppose that $p$ is order-preserving on submodules and surjection-cofunctorial.  Suppose $\pi:M \rightarrow M'$ is a surjective homomorphism of $R$-modules, $L \subseteq M$, and $(L,M),(\pi(L),M') \in \cP$. By hypothesis,  $(\pi^{-1}(\pi(L)),M) \in \cP$ as well.  By properties of images and preimages of functions we know $L \subseteq \pi^{-1}(\pi(L))$.  Since $p$ is surjection-cofunctorial and order preserving on submodules, we have \[p(L,M)\subseteq p(\pi^{-1}(\pi(L)),M) \subseteq \pi^{-1}(p(\pi(L),M')).\] Now take the image of each of the submodules to obtain
\[\pi(p(L,M)) \subseteq \pi(\pi^{-1}(p(\pi(L),M'))) = p(\pi(L),M').\]
 implying $p$ is surjection-functorial.

\eqref{it:funceq} First assume that $p$ is functorial and the other given hypotheses hold. By \eqref{it:injsurjfunc}, $p$ is surjection functorial and order-preserving on ambient modules. By \eqref{it:ifunceq}, $p$ is restrictable and by \eqref{it:sfunctospara}, $p$ is surjection cofunctorial. Then by \eqref{it:injsurjcofunc}, $p$ is cofunctorial.

Now assume that $p$ is cofunctorial and the other given hypotheses hold. By \eqref{it:injsurjcofunc}, $p$ is surjection cofunctorial and restrictable. By \eqref{it:iparatoifunc}, $p$ is order preserving on ambient modules. By \eqref{it:sfunceq}, $p$ is surjection functorial, and so by \eqref{it:injsurjfunc}, $p$ is functorial.
\end{proof}

\begin{example} \label{ex:RRop}
In \cite[Example 2.6]{ERGV-nonres}, we showed that the Ratliff-Rush operation given by
\[p(I,J)=I^{\RR}_J=\bigcup\limits_{n \geq 0} (I^{n+1}:_JI^{n}),\] when defined on pairs of ideals of a ring $R$ is not restrictable.  By Proposition \ref{pr:cofunctorial}\eqref{it:injsurjcofunc}, any cofunctorial operations must be restrictable; hence, the Ratliff-Rush operation is not cofunctorial.  

However, if $R$ is a domain, then the Ratliff-Rush operation is functorial. To see this, let $I \subseteq J$ be ideals of $R$, with $J$ nonzero.  Let $\phi: J \rightarrow J^{\prime}$ be an $R$-module homomorphism, where $J^{\prime}$ is another ideal.  Let $j \neq 0$ be a fixed nonzero element of $J$.  Then for any $x \in J$, we have $j\phi(x) = \phi(jx) = x\phi(j)$. Similarly, $\phi(j)I = j\phi(I)$.  Let $f\in I^{\RR}_J$ and suppose $I^n f \subseteq I^{n+1}$.  We will show that $\phi(I)^n \phi(f) \subseteq \phi(I)^{n+1}$ which will prove that the Ratliff-Rush operation is functorial.

To see this, it is enough to show that for any $a_1, \ldots, a_n \in I$, we have $\phi(a_1) \cdots \phi(a_n) \phi(f) \in \phi(I)^{n+1}$.  Multiplying the left hand side by $j^{n+1}$,  we have

\begin{align*}
j^{n+1} \phi(a_1) \cdots \phi(a_n) \phi(f) &= \phi(ja_1) \cdots \phi(ja_n) \phi(jf) \\
&= \phi(j)^{n+1} a_1 \cdots a_n f \in \phi(j)^{n+1} I^{n+1} \\
&= (\phi(j)I)^{n+1} = (j \phi(I))^{n+1} = j^{n+1} \phi(I)^{n+1}.
\end{align*}
Since $j^{n+1}$ is a regular element of $R$, we can cancel it to get the required inclusion.   
\end{example}

\begin{prop}\label{pr:functcofunct}
Let $R$ be a complete Noetherian local ring, assume $\cM$ contains Matlis-dualizable modules, and let $p$ be a pair operation on a class of pairs of Matlis-dualizable $R$-modules $\cP$ as in Definition \ref{def:pairop}.   Then $p$ is functorial if and only if $p^\dual$ is cofunctorial. 
\end{prop}
\begin{proof}
Suppose $p$ is functorial. Let $D \subseteq C$ and $\phi: B \rightarrow C$ such that $(D,C), (\phi^{-1}(D),B) \in \cP^\vee$.  We will show that $p^\dual$ is cofunctorial, i.e., $p^\dual(\phi^{-1}(D),B) \subseteq \phi^{-1}(p^\dual(D,C))$.  Set $N=B^\vee$, $M=C^\vee$, $L=(C/D)^\vee$ and $g:=\phi^\vee:M \rightarrow N$. Note that, by Matlis duality, $D=(M/L)^\vee$.
By Lemma \ref{lem:matlisimagepreimage}, \[p^\dual(\phi^{-1}(D),B)=p^\dual(\phi^{-1}((M/L)^\vee),B)=p^\dual((N/g(L))^\vee,B).\] 
By the definition of $p^\dual$, 
\[p^\dual((N/g(L))^\vee,B)=\left(\displaystyle\frac{B^\vee}{p((B/(N/g(L))^\vee)^\vee,B^\vee)} \right)^\vee=\left(\displaystyle\frac{N}{p(g(L),N)} \right)^\vee
\]  where the last inequality follows from Matlis duality.
By functoriality,
\[ \left(\displaystyle\frac{N}{p(g(L),N)} \right)^\vee \subseteq \left(\displaystyle\frac{N}{g(p(L,M))} \right)^\vee.
\]
Applying Lemma \ref{lem:matlisimagepreimage}, 
\[\left(\displaystyle\frac{N}{g(p(L,M))} \right)^\vee=\phi^{-1}\left((M/p(L,M))^\vee\right).
\]
Now by the definition of $p^\dual$,  and  since $p^{\dual \dual}=p$,
\[\phi^{-1}((M/p(L,M))^\vee)=\phi^{-1}(p^\dual((M/L)^\vee, C))=\phi^{-1}(p^\dual(D, C))
\]
which implies that $p^\dual$ is cofunctorial.

Suppose now that $p$ is cofunctorial, $K \subseteq N$ and $g:M \ra N$.  Set $B=N^\vee$ and $C=M^\vee$, $A=(N/K)^\vee$ and  $\phi:=g^\vee: B \ra C$ such that $(A,B),(\phi(A),C) \in \cP^\vee$.  We need to show that $\phi(p^\dual(A,B)) \subseteq p^\dual(\phi(A),C)$.  By definition of $p^\dual$,
\[\phi(p^\dual(A,B))=\phi\left( \left(\displaystyle\frac{B^\vee}{p((B/A)^\vee,B^\vee)} \right)^\vee\right)=\phi\left(\left(\displaystyle\frac{N}{p(K,N)} \right)^\vee \right).
\]
By Lemma \ref{lem:matlisimagepreimage},
\[\phi\left(\left(\displaystyle\frac{N}{p(K,N)} \right)^\vee \right)=\left(\displaystyle\frac{M}{g^{-1}(p(K,N))} \right)^\vee.
\]
Using the fact that $p$ is cofunctorial we obtain,
\[\left(\displaystyle\frac{M}{g^{-1}(p(K,N))} \right)^\vee \subseteq \left(\displaystyle\frac{M}{(p(g^{-1}(K),M)} \right)^\vee.
\]
By the definition of $p^\dual$ we have, 
\[\left(\displaystyle\frac{M}{p(g^{-1}(K),M)} \right)^\vee=p^\dual\left(\left( \displaystyle\frac{M}{g^{-1}(K)}\right)^\vee, M^\vee\right)=p^\dual\left(\left( \displaystyle\frac{M}{g^{-1}(K)}\right)^\vee, C\right).
\]
Applying Lemma \ref{lem:matlisimagepreimage}, we obtain
\[p^\dual\left(\left( \displaystyle\frac{M}{g^{-1}(K)}\right)^\vee, C\right)=p^\dual\left(\phi\left(\left(\displaystyle\frac{N}{K}\right)^\vee\right), C \right)=p^\dual(\phi(A), C )
\] completing the proof that $p^\dual$ is functorial. 
\end{proof}

\begin{rem}
The proof for Proposition \ref{pr:functcofunct} also provides the following dualities for pair operations $p$ and $p^\smile$ for Matlis dualizable pairs over a complete Noetherian local ring.
\begin{enumerate}
    \item \cite[Proposition 3.6(8)]{ERGV-nonres} $p$ is surjection-functorial if and only if $p^\smile$ is restrictable. 
    \item $p$ is surjection-cofunctorial if and only if $p^\dual$ is order preserving on ambient modules.
\end{enumerate}

In particular, (2) follows from the second part of the proof of Proposition~\ref{pr:functcofunct} by restricting the map $g$ to be a surjection.
\end{rem}

\begin{rem}
    Let  $p$ be the Ratliff-Rush operation, as defined in Example \ref{ex:RRop}, on pairs of ideals of a complete local domain. Proposition \ref{pr:functcofunct} implies that $p^\smile$ is an example of a pair operation on pairs of submodules of the injective hull of the residue field which is cofunctorial  but not functorial.
\end{rem}

\subsection{Hereditary and cohereditary pair operations}
 Here we define properties of pair operations that we use to better understand how various closure and interior operations have been extended from operations on ideals to operations on submodules.  For closure operations, the terms \emph{residual} and \emph{hereditary} may be familiar.  These two notions are dual for closure operations, but not for pair operations more generally.  Through the lens of pair operations, we discuss these properties and their duals.

\begin{defn}
A pair operation $p$ on a class $\cP$ of pairs of $R$-modules $(L,M)$ with $L \subseteq M$ is \begin{itemize}
    \item  \emph{absolute} if whenever $L \subseteq N \subseteq M$ with $(L,M), (L,N) \in \cP$, we have $p(L,N) = p(L,M)$ \cite[Definition 2.2]{ERGV-nonres};
    \item \emph{residual} if whenever $L \subseteq N \subseteq M$ with $(N,M), (N/L,M/L) \in \cP$, we have $p(N,M) = \pi^{-1}(p(N/L, M/L))$, where $\pi: M \onto M/L$ is the natural surjection \cite[Definition 2.2]{ERGV-nonres};
    \item \emph{hereditary} 
    if whenever $L \subseteq N \subseteq M$ with $(L,M), (L,N) \in \cP$, we have $p(L,N) = p(L,M) \cap N$;
    \item \emph{cohereditary} 
    if whenever $L \subseteq N \subseteq M$ with $(N,M), (N/L, M/L) \in \cP$, we have $p(N/L, M/L) = \frac{p(N,M) + L}L$.
\end{itemize}
\end{defn}

\begin{rem}
Note that by definition, hereditary implies restrictable.
\end{rem}

\begin{rem} \label{rem:dualRRop}
In \cite[5.5.10]{elliottclosureops}, Elliott calls ``left exact'' what we call hereditary, and he calls ``residual'' or ``right exact'' what we call cohereditary.  The conflict in terminology arises from our need here to distinguish between two notions that we call cohereditary and residual.
\end{rem}

\begin{lemma}\label{lem:absiffinther}
Let $p$ be a pair operation on a class $\cP$ of pairs of $R$-modules.  If $p$ is intensive and hereditary, then $p$ is absolute.  If we further assume that when $(L,M) \in \cP$, then $(L,L) \in \cP$, then the converse holds. 
\end{lemma}

\begin{proof}
Suppose $p$ is intensive and hereditary.  Let $L \subseteq N \subseteq M$ with $(L,N), (L,M) \in \cP$.  Then $p(L,N) = p(L,M) \cap N = p(L,M)$, with the first equality by hereditariness and the second by intensivity.

Now, suppose $p$ is absolute, and let $(L,M) \in \cP$.  
Since we also have $(L,L) \in \cP$, we have $p(L,L) = p(L,M)$ by absoluteness, and by definition of pair operation we have $p(L,L) \subseteq L$.  Hence $p$ is intensive.  Now let $L \subseteq N \subseteq M$ with $(L,N),(L,M) \in \cP$.  Then by intensivity, we have $p(L,M) \subseteq L \subseteq N$, so that $p(L,N) = p(L,M) = p(L,M) \cap N$.

\end{proof}

\begin{lemma}\label{lem:resiffextcoher}
Let $p$ be a pair operation on $\cP$.    If $p$ is extensive and cohereditary, then $p$ is residual.  If we further assume that  $(0,M/L) \in \cP$ for all $(L,M) \in \cP$ with $L \subseteq M$, then the converse holds. 
\end{lemma}

\begin{proof}
Suppose $p$ is cohereditary and extensive.  Let $L \subseteq N \subseteq M$ with $(N,M), (N/L, M/L) \in \cP$, setting $\pi: M \onto M/L$ as the canonical surjection.  Then \[\pi(p(N,M)) = \frac{p(N,M)+L}{L} = p(N/L, M/L),\] so that $p(N,M) \subseteq \pi^{-1}(p(N/L, M/L))$.  For the reverse containment, let $z \in \pi^{-1}(p(N/L,M/L))$.  Then $\pi(z) \in p(N/L, M/L) = \pi(p(N,M))$ by the cohereditary property.  Hence, $z \in p(N,M) + \ker \pi = p(N,M)+L = p(N,M)$, with the latter equality since $L \subseteq N \subseteq p(N,M)$ by extensivity.  Hence, $p$ is residual.

Now suppose $p$ is residual and that whenever $(L,M) \in \cP$, $(0,M/L) \in \cP$ as well.  Then it is extensive because if $(L,M) \in \cP$ and $\pi: M \onto M/L$ is the natural surjection, then setting $L=N$ in the definition, we have $L = \pi^{-1}(0)  \subseteq \pi^{-1}(p(0,M/L)) = p(L,M)$.  To see that it is cohereditary, from the fact that $\pi$ is surjective we have \[
\frac{p(N,M)+L}L = \pi(p(N,M)) = \pi(\pi^{-1}(p(N/L, M/L))) = p(N/L, M/L)
\]
as long as all of the relevant pairs are in $\cP$.
\end{proof}

\begin{prop}\label{pr:dualityprops}
Let $R$ be a complete Noetherian local ring, assume $\cM$ contains Matlis-dualizable modules, and let $p$ be a pair operation on a class of pairs of Matlis-dualizable $R$-modules $\cP$ as in Definition \ref{def:pairop}.  Then \begin{enumerate}
    \item\label{it:extintdual}\cite[Proposition 3.6(2) and (3)]{ERGV-nonres} $p$ is extensive if and only if $p^\dual$ is intensive.
    \item\label{it:hercoher} If $p$ is hereditary, then $p^\dual$ is cohereditary.
    \item\label{it:coherher} If $p$ is cohereditary, then $p^\dual$ is hereditary.
    \item\label{it:resabsdual}  Suppose for all $(L,M) \in \cP$, $(0,M/L) \in \cP$ and $(L,L) \in \cP$. Then  $p$ is residual if and only if $p^\dual$ is absolute.
\end{enumerate}
\end{prop}

\begin{proof}\ 
(\ref{it:hercoher}): Let $L \subseteq N \subseteq M$ with $(N/L, M/L), (N,M) \in \cP^\vee$.  Let $B := M^\vee$, and let $A := (M/N)^\vee$, $C := (M/L)^\vee$, considered as submodules of $B$.  Then $A \subseteq C \subseteq B$, and we have $(A,B) = ((M/N)^\vee, M^\vee) = (N,M)^\vee \in \cP$ and $(A,C) = ((M/N)^\vee, (M/L)^\vee) = (N/L, M/L)^\vee \in \cP$.  Thus, we have \[
  p^\dual(N,M) + L = \left(\frac B {p(A,B)}\right)^\vee + \left(\frac BC\right)^\vee = \left(\frac B {p(A,B) \cap C}\right)^\vee = \left(\frac B {p(A,C)}\right)^\vee.
\]
Here the second equality is by \cite[Lemma 6.15]{ERGV-chdual} and the third is since $p$ is hereditary. Hence, \[
\frac{p^\dual(N,M) + L}{L} = \frac{(B/p(A,C))^\vee}{(B/C)^\vee} = (C/p(A,C))^\vee = p^\dual(N/L, M/L).
\]

(\ref{it:coherher}): Let $L \subseteq N \subseteq M$ with $(L,M), (L,N) \in \cP^\vee$.   Let $B := M^\vee$, and let $A := (M/N)^\vee$, $C := (M/L)^\vee$, considered as submodules of $B$.  Then $(C,B) = ((M/L)^\vee, M^\vee) = (L,M)^\vee \in \cP$ and $(C/A, B/A) = ((N/L)^\vee, N^\vee) := (L,N)^\vee \in \cP$.  Thus we have \begin{align*}
    p^\dual(L,N) &= p^\dual((B/C)^\vee, (B/A)^\vee) = \left( \frac{B/A} {p(C/A, B/A)}\right)^\vee = \left(\frac{B/A} {(p(C,B)+A)/A}\right)^\vee \\
    &= \left(\frac B {p(C,B)+A}\right)^\vee = \left(\frac B {p(C,B)}\right)^\vee \cap (B/A)^\vee = p^\dual(L,M) \cap N.
\end{align*}
Here the third equality is because $p$ is cohereditary, the fourth equality is by the third isomorphism theorem, and the fifth is by \cite[Lemma 6.15]{ERGV-chdual}.

(\ref{it:resabsdual}): This follows from (\ref{it:hercoher}) and (\ref{it:coherher}), Lemmas~\ref{lem:absiffinther} and \ref{lem:resiffextcoher}, and (\ref{it:extintdual}).

\end{proof}

\begin{lemma}
\label{lem:residualnakayama}
Let $(R,\m)$ be a Noetherian local ring and $\cl$ be a residual closure operation on the class of pairs of  \fg\ $R$-modules. Then $\cl$ is Nakayama if and only if whenever $N \subseteq (\m N)_M^{\cl}$, $N \subseteq 0_M^{\cl}$.
\end{lemma}

\begin{proof}
The forward direction is immediate from Definition \ref{def:nakayama}. For the reverse direction, let $L \subseteq N \subseteq M$ such that $N \subseteq (L+\m N)_M^{\cl}$. Then 
\[N/L \subseteq (L+\m N)^{\cl}_M/L=\left(\frac{L+\m  N}{L}\right)_{M/L}^{\cl}.\] By our hypotheses, this implies that $N/L \subseteq 0_{M/L}^{\cl}$. Since $\cl$ is residual, $N \subseteq L_M^{\cl}$. This implies that $N_M^{\cl}=L_M^{\cl}$, as desired. 
\end{proof}

\subsection{Connection to submodule selectors}
There are two natural ways to define a pair operation from a submodule selector.  In this subsection, we discuss how properties of a given submodule selector compare with the properties of the associated pair operations.

\begin{rem}
In \cite{nmeRG-cidual}, we discussed hereditary and cohereditary properties in the context of submodule selectors. 
In that paper, the dual of a hereditary submodule selector is also called co-hereditary (Definition 6.1). Note that this is referring to the duality between submodule selectors of \cite{nmeRG-cidual}, whereas the duality of this paper takes submodule selectors to residual extensive operations and vice versa. The two notions are compatible in the sense that if $\cl$ is a residual closure operation, then the submodule selector $\alpha(M):=0_M^\cl$ is co-hereditary by definition. If a submodule selector $\alpha$ is co-hereditary, then its dual submodule selector $\alpha^\dual$ is hereditary by Proposition 6.5 of \cite{nmeRG-cidual}. Hence if we apply the duality of this paper to $\alpha^\dual$ to get an extensive operation, this extensive operation will be residual and agrees with what we get if instead we take the residual closure operation $\cl$ such that $\alpha(M)=0_M^\cl$ for $R$-modules $M \in \cM$ as in Construction 2.3 of \cite{nmeRG-cidual}.

We discuss residual and absolute in the context of submodule selectors in the next result.
\end{rem}

\begin{defn}[{\cite[Definition 2.1]{nmeRG-cidual}}]
\label{def:submodselect}
Let $\cM$ be a category of $R$-modules. A \textit{submodule selector} is a map $\alpha:\cM \to \cM$ sending each module $M$ to a submodule of $M$ such that if $M \cong N$ in $\cM$, $\alpha(M) \cong \alpha(N)$ under the restriction of the isomorphism.

Let $R$ be a complete local ring and $^\vee$ denote the Matlis duality operator. We define a dual $\alpha^\dual$ to a submodule selector $\alpha$ by
\[\alpha^\dual(M):=\left(\frac{M^\vee}{\alpha(M^\vee)} \right)^\vee.\]
\end{defn}

\begin{defn}
Let $\alpha$ be a submodule selector.  Then there is a residual operation $r = \rho(\alpha)$ given by $r(L,M) = \pi^{-1}(\alpha(M/L))$, where $\pi: M \onto M/L$ is the canonical map.

Also, there is an absolute operation $g=\gamma(\alpha)$ given by $g(L,M) = \alpha(L)$.
\end{defn}

\begin{prop}
Let $R$ be complete local commutative Noetherian.  Let $\alpha$ be a submodule selector on a category $\cM$ of Matlis dualizable modules.  Then $\rho(\alpha^\smile) = \gamma(\alpha)^\smile$, and $\rho(\alpha)^\smile = \gamma(\alpha^\smile)$.
\end{prop}

\begin{proof}
Let $(L,M) \in \cP^\vee$.  Then \begin{align*}
    \rho(\alpha^\smile)(L,M) &=\pi^{-1}(\alpha^\smile(M/L)) = \pi^{-1}\left( \left( \frac{(M/L)^\vee} {\alpha((M/L)^\vee)}\right)^\vee \right)\\
    &= \left(\frac{M^\vee} {\alpha((M/L)^\vee)}\right)^\vee = \left(\frac{M^\vee} {\gamma(\alpha)((M/L)^\vee, M^\vee)}\right)^\vee \\
    &= \gamma(\alpha)^\smile(L,M).   
\end{align*}
The third equality is by Lemma~\ref{lem:matlisimagepreimage}.

For the second statement, the above applies to yield the following: \[
\rho(\alpha)^\smile = \rho(\alpha^{\smile \smile})^\smile = \gamma(\alpha^\smile)^{\smile \smile} = \gamma(\alpha^\smile). \qedhere
\]
\end{proof}

 We discussed functorial and cofunctorial pair operations in Section \ref{subs:funccofunc}.  We now remind the reader of the definition of functoriality for submodule selectors. 
\begin{defn} \cite[Definition 2.1]{nmeRG-cidual} \cite[Definition 5.1]{nmeRG-cidual}
\label{def:submodprops}
    Let $\cM$ be a category of modules, $\cP$ be the set of pairs $(L,M)$ such that $L \subseteq M$ and $L,M \in \cM$ and $\alpha$ be a submodule selector on $\cM$ (see Definition \ref{def:submodselect}). Suppose $M,N \in \cM$.  We say that
    \begin{itemize}
    \item $\alpha$ is \emph{order-preserving} if for any $(L,M) \in \cP$, $\alpha(L) \subseteq \alpha(M)$.
    \item $\alpha$ is \emph{surjection-functorial} if for $\pi:M \onto N$, $\pi(\alpha(M)) \subseteq \alpha(N)$.
    \item $\alpha$ is \emph{functorial} if for $f:M \rightarrow N$, $f(\alpha(M)) \subseteq \alpha(N)$.
    \item $\alpha$ is \emph{idempotent} if for any $M \in \cM$, $\alpha(\alpha(M))=\alpha(M)$.
    \end{itemize}
    We define the \emph{finitistic version} $\alpha_f$ of $\alpha$ to be 
\[\alpha_f(M)=\sum \{\alpha(L) \mid L \subseteq M \text{ is \fg\ and }L \in \cM\}.\]
We say that $\alpha$ is \textit{finitistic} if for every $M \in \cM$, $\alpha=\alpha_f$.
\end{defn}
Note that functoriality here is equivalent to order-preservation plus surjection-functoriality, assuming that for any $f: M \ra N$ in $\cP$, there is some epi-monic factorization $M \onto Q \into N$ such that $Q \in \cM$.  Note also that since 
\[\alpha(M) \subseteq f^{-1}(f(\alpha(M))) \subseteq f^{-1}(\alpha(N)),\] then if we were to define ``cofunctorial'' to be $f^{-1}(\alpha(N)) \supseteq \alpha(M)$, this would be the same as $\alpha$ being functorial in the sense of \cite[Definition 2.1]{nmeRG-cidual}.  


\begin{prop}
Let $\alpha$ be a submodule selector on an abelian category $\cM$ of modules and $\cP$ be as in Definition \ref{def:submodprops}. Let $g=\gamma(\alpha)$ and $r=\rho(\alpha)$.  \begin{enumerate}
    \item \label{it:opgop} The following are equivalent: \begin{enumerate}
        \item $\alpha$ is order-preserving.
        \item $g$ is order-preserving on submodules.
        \item $g$ is restrictable. 
        \item $r$ is cofunctorial. 
     \end{enumerate}
    \item \label{it:sfgfunct} The following are equivalent: \begin{enumerate}
    \item $\alpha$ is surjection-functorial.
    \item $g$ is functorial.
    \item $r$ is order-preserving on submodules. 
    \end{enumerate}
    \item\label{it:func} The following are equivalent:
    \begin{enumerate}
    \item $\alpha$ is functorial.
    \item $g$ is functorial and order preserving on submodules.
    \item $r$ is cofunctorial and order preserving on submodules. 
    \item $r$ is functorial and order preserving on submodules. 
    \end{enumerate}
    \item\label{it:idem} $\alpha$ is idempotent $\iff g$ is idempotent.
    \end{enumerate}
\end{prop}

For more such equivalences, see \cite[Proposition~2.6]{nmeRG-cidual}.

\begin{proof}\




(\ref{it:opgop}) (a) $\iff$ (b): Let $L \subseteq N \subseteq M$ such that $(L,M),(N,M) \in \cP$. Since $g(L,M)=\alpha(L)$ and $g(N,M)=\alpha(N)$, we have $\alpha(L) \subseteq \alpha(N)$ if and only if $g(L,M) \subseteq g(N,M)$.

(\ref{it:opgop}) (b) $\iff$ (c): First assume that $g$ is order-preserving on submodules. Then $g$ is restrictable if and only if $g$ is order-preserving on ambient modules. Let $L \subseteq N \subseteq M$ such that $(L,N),(L,M) \in \cP$. We have $g(L,N)=\alpha(L)=g(L,M)$, so $g$ is order-preserving on ambient modules and hence is restrictable. 

Now assume that $g$ is restrictable. Note that by definition of $\cP$, we meet the hypotheses of Proposition \ref{pr:cofunctorial}(\ref{it:ifunceq}). 
 This means that for $L,N \subseteq M$, $g(L \cap N,L) \subseteq g(N,M)$. Assume that $L \subseteq N \subseteq M$. Then $g(L \cap N,L)=g(L,L)=\alpha(L)$ and $g(N,M)=\alpha(N)$. So $\alpha(L) \subseteq \alpha(N)$, which implies that $g(L,M) \subseteq g(N,M)$, as desired.

(\ref{it:opgop}) (a) $\iff$ (d): Suppose $\alpha$ is order-preserving.
Let $f: M \rightarrow N$ be a map, and $K \subseteq N$ a submodule.  Let $\pi: M \onto M/f^{-1}(K)$ and $q: N \onto N/K$ be the natural surjections, 
and let $e: M/f^{-1}(K) \rightarrow N/K$ be the natural injective map guaranteed by the first isomorphism theorem.  Then $e \circ \pi = q \circ f$.  Now let $x\in r(f^{-1}(K),M)$.  Then by definition of $r$, we have $x \in \pi^{-1}(\alpha(M/f^{-1}(K)))$.  That is, $\pi(x) \in \alpha(M/f^{-1}(K))$.  Thus, \[q(f(x)) = e(\pi(x)) \in e(\alpha(M/f^{-1}(K))) \subseteq \alpha(N/K)\] since $\alpha$ is order-preserving.  But then $x\in f^{-1}(q^{-1}(\alpha(N/K))) = f^{-1}(r(K,N))$.  Thus, $r$ is cofunctorial.

Conversely, suppose $r$ is cofunctorial.  Let $L \subseteq M$ be $R$-modules, with $i: L \rightarrow M$ the inclusion map. Cofunctoriality of $r$ means that \begin{align*}\alpha(L) & = r(0,L) = r(i^{-1}(0),L) \subseteq i^{-1}(r(0,M)) \\ &= r(0,M) \cap L \subseteq r(0,M) = \alpha(M).\end{align*}

(\ref{it:sfgfunct}) (a) $\iff$ (b): Suppose $\alpha$ is surjection-functorial.  Let $f: M \rightarrow N$ be a map, $L$ a submodule of $M$, and $e: L \onto f(L)$ the induced surjective map on submodules.  Then \[f(g(L,M))= f(\alpha(L)) = e(\alpha(L))\subseteq \alpha(f(L)) = g(f(L),N).\]

Conversely suppose $g$ is functorial.  Let $q: M \onto N$ be a surjective map.  Then $q(\alpha(M)) = q(g(M,M)) \subseteq g(q(M),N) = g(N,N) = \alpha(N)$.

(\ref{it:sfgfunct}) (a) $\iff$ (c):  Suppose $\alpha$ is surjection-functorial.
Let $L \subseteq N \subseteq M$ be $R$-modules.  Let 
$\pi_L: M \onto M/L$, $q: M/L \onto M/N$, and $\pi_N: M \onto M/N$ be the natural maps, so that $\pi_N = q \circ \pi_L$.  Then by surjection-functoriality of $\alpha$, $q(\alpha(M/L)) \subseteq \alpha(M/N)$.  Accordingly, let $x\in r(L,M)$.  Then $\pi_L(x) \in \alpha(M/L)$, so \[\pi_N(x) = q(\pi_L(x)) \in q(\alpha(M/L)) \subseteq \alpha(M/N).\]  Then by definition, $x \in r(N,M)$.

Conversely, suppose $r$ is order-preserving on submodules.  Let $q: M \onto N$ be a surjection.  Without loss of generality, $N=M/L$ for some submodule $L$ of $M$, and $q$ is the natural map.  Then \[\alpha(M) = r(0,M) \subseteq r(L,M) = q^{-1}(\alpha(M/L)),\] so for any $x\in \alpha(M)$, we have $q(x) \in \alpha(M/L) = \alpha(N)$. 

(\ref{it:func}): This follows from \ref{it:opgop} and \ref{it:sfgfunct} and Proposition \ref{pr:cofunctorial} (\ref{it:funceq}). 




(\ref{it:idem}): Let $L \subseteq M$ such that $(L,M),(g(L,M),M) \in \cP$. Then $g(g(L,M),M)=g(\alpha(L),M)=\alpha(\alpha(L))$. So $\alpha$ is idempotent if and only if $g$ is idempotent.

{\cvl 
}

\end{proof}

Suppose $p$ and $q$ are pair operations on a class of pairs $\cP$.  We say that $p \leq q$ if for all $(L,M) \in \cP$, $p(L,M) \subseteq q(L,M)$.  As we have pointed out both in this and our previous paper  \cite{ERGV-nonres}, pair operations are a generalization of submodule selectors.  In \cite[Proposition 7.5] {nmeRG-cidual}, the first and second named authors proved that if $\alpha$ and $\beta$ are submodule selectors on a class of Matlis-dualizable modules $\cM$ 
with $\alpha \leq \beta$, 
then $\beta^\dual \leq \alpha^\dual$.  A similar statement holds for pair operations:

\begin{prop}\label{pr:comparepairops}
Let $(R,\m)$ be a complete local ring. Let $p$ and $q$ be pair operations defined on a class of Matlis dualizable pairs $\cP$ satisfying $p \leq q$.  Then $q^\dual \leq p^\dual$.
\end{prop}

\begin{proof}
The proof follows the same steps as the proof of \cite[Proposition 7.5]{nmeRG-cidual} replacing the submodule selectors $\alpha$ and $\beta$ by the pairs $p$ and $q$.
\end{proof}

\subsection{Finitistic pair operations}

We define the finitistic version of a pair operation and prove some of its properties. This material will be used in Section \ref{sec:residualversion} to give a method for computing certain interiors and in Section \ref{sec:integralclosure} to compare EHU-integral closure to its hereditary version.

\begin{defn}[c.f. {\cite[Definition~3.7]{ERGV-nonres}}]
\label{def:finitistic}
Let $R$ be a Noetherian ring, and let $p$ be a pair operation on a class $\cP$ of pairs of nested $R$-modules. We define the \textit{finitistic version} $p_f$ of $p$ to be 
\[p_f(L,M)=\bigcup \{p(L\cap U,U) \mid U \subseteq M \text{ is \fg\ and } (L \cap U,U) \in \cP\}.\]
We say that $p$ is \textit{finitistic} if for every $(L,M) \in \cP$, $p=p_f$.
\end{defn}

\begin{lemma}[c.f. {\cite[Lemma 3.2]{ERGV-chdual}}]
\label{lem:twoversionsoffinitistic}
If $\cP$ is a class of pairs of $R$-modules such that if $(L\cap U,U) \in \cP$ then $(L,L+U) \in \cP$, and if $p$ is a restrictable pair operation, then for every $(L,M) \in \cP$,
\[p_f(L,M) \subseteq \bigcup_{L \subseteq N \subseteq M,\ N/L\, \text{f.g.}, (L,N) \in \cP} p(L,N).\]
If $(L,N) \in \cP$ if and only if $(0,N/L) \in \cP$ and $p$ is also residual, then the containment is an equality.
\end{lemma}

\begin{proof}
The forwards containment holds because $p(L \cap U,U) \subseteq p(L,U+L)$ and $(U+L)/L \cong U/(L \cap U)$ is \fg\ when $U$ is. 

For the backwards containment, let $z \in p(L,N)$ for some $L \subseteq N \subseteq M$ with $N/L$ \fg\ and $(L,N) \in \cP$. Set $x_1=z,x_2,\ldots,x_t \in N$ whose images generate $N/L$. Then $N=L+U$, where $U=\sum_{i=1}^t Rx_i$ is \fg.  
Set $\pi:N \to N/L$ to be the quotient map. By the Second Isomorphism Theorem, there is an isomorphism $j: \frac{L+U}L \ra \frac U{L\cap U}$ such that for any $y\in U$, $j(\bar y) = \bar y$. Since $p$ is residual, 
\[\pi(z) \in p(0,N/L)=p(0,(L+U)/L).
\]
Thus, $j(\pi(z)) = \bar z \in p(0, U/(L \cap U))$.
Hence $z \in p(L \cap U,U)$, as desired.
\end{proof}

If we have a pair operation defined only on \fg\ $R$-modules, its finitistic version gives us an extension to all $R$-modules:

\begin{lemma}
\label{lem:finitisticequalsorigonfg}
Let $R$ be a Noetherian ring and $p$ a pair operation defined on all nested pairs of \fg\ $R$-modules. If $p$ is restrictable, 
then $p_f(L,M)=p(L,M)$ when $L \subseteq M$ are \fg\ $R$-modules.
\end{lemma}

\begin{proof}
Let $L \subseteq M$ be \fg\ $R$-modules. By the definition of $p_f$, $p(L,M) \subseteq p_f(L,M)$. For the other containment, notice that for any $U \subseteq M$, $U$ is \fg, and $p(L \cap U,U) \subseteq p(L,M)$. So $p_f(L,M) \subseteq p(L,M)$.
\end{proof}

\begin{lemma}
\label{lem:finitisticpairopprops}
If $R$ is Noetherian and $\cl$ is a restrictable closure operation on all nested pairs of \fg\ $R$-modules,
then $\cl_f$ is a closure operation on all $R$-modules.
\end{lemma}

\begin{proof}
Since $\cl$ is extensive, for each $U \subseteq M$ \fg\ such that $(L \cap U,U) \in \cP$, $L \cap U \subseteq (L \cap U)_U^{\cl}$. So $\cl_f$ is extensive.

Let $L \subseteq N \subseteq M$. Since $\cl$ is order-preserving on submodules, for each $U \subseteq M$ with $U$ \fg,
\[(L \cap U)_U^{\cl} \subseteq (N \cap U)_U^{\cl}.\]
Consequently,
\[\bigcup_{U \subseteq M, U\, \text{f.g.}} (L \cap U)_U^{\cl} \subseteq \bigcup_{U \subseteq M, U\, \text{f.g.}} (N \cap U)_U^{\cl},\]
so $\cl_f$ is order-preserving on submodules.



For idempotence, let $L \subseteq M$ be $R$-modules and let $x\in (L^{\cl_f}_M)^{\cl_f}_M$.  Then there is some finitely generated submodule $U \subseteq M$ such that $x\in (L^{\cl_f}_M \cap U)^\cl_U$.  Choose a finite generating set $y_1, \ldots, y_t$ for $L^{\cl_f}_M \cap U$.  Since each $y_j$ is in $L_M^{\cl_f}$, there are finitely generated submodules $V_1, \ldots, V_t$ of $M$ such that for $1\leq j \leq t$, we have $y_j \in (L \cap V_j)^\cl_{V_j}$. Set $V := U + \sum_{j=1}^t V_j$.  Then for each $j$, we have \[
y_j \in (L \cap V_j)^\cl_{V_j} = ((L \cap V) \cap V_j)^\cl_{V_j} \subseteq (L \cap V)^\cl_V
\]
by restrictability of $\cl$.  Hence, $L^{\cl_f}_M \cap U \subseteq (L \cap V)^\cl_V$.  Thus we have \[
x \in (L^{\cl_f}_M \cap U)^\cl_U \subseteq (L^{\cl_f}_M \cap U)^\cl_V \subseteq ((L \cap V)^\cl_V)^\cl_V = (L \cap V)^\cl_V \subseteq L^{\cl_f}_M.
\]
The first containment holds because $\cl$ is order-preserving on ambient modules by Proposition~\ref{pr:cofunctorial}(\ref{it:iparatoifunc}).  The second containment holds because $\cl$ is \opsub. The first equality follows from idempotence of $\cl$. The last containment is by definition of $\cl_f$.
Since $x$ was arbitrarily chosen and $\cl_f$ is extensive, idempotence of $\cl_f$ follows.



\end{proof}

\section{Cohereditary versions of pair operations}
\label{sec:residualversion}

In this section we show how to define a cohereditary version $\po p {\ch}$ of a pair operation $p$ that is not necessarily cohereditary.  Since residual operations are cohereditary and extensive (see Lemma~\ref{lem:resiffextcoher}), we will apply our results to closure operations to define a residual version of a closure operation that is not necessarily residual. We end by extending a result of \cite{ERGV-chdual} to non-residual closure operations by working with their residual versions.

\begin{notation}
\label{not:coheredpairs}
Throughout this section, we assume the following: 
\begin{enumerate}
    \item $\cM$ will denote an abelian category of $R$-modules that has enough projectives.
    \item $\cP$ will denote a set of pairs of $R$-modules $(N,P)$ with $N \subseteq P$, $N,P \in \cM$, and $P$ projective.
    \item $\cP'$ will denote a set of pairs such that $\cP \subseteq \cP'$ and for every $(L,M) \in \cP'$ with $L \subseteq M$, there is a projective module $P$ $\in \cM$ and a surjection $\pi:P \to M$ in $\cM$ with $(\pi^{-1}(L),P) \in \cP$.
    \item In addition, if $(L,M),(N,M) \in \cP'$ with $L \subseteq N$, then there exists $\pi:P \twoheadrightarrow M$ such that $(\pi^{-1}(L),P),(\pi^{-1}(N),P) \in \cP$.
    \item We further assume:
    \begin{itemize}
        \item[(*)] Given $L,M,P,\pi$ as above, if we have a commutative diagram
    \[\xymatrix{
P  \ar@{->>}[d]^{\pi} \ar[r]^{\tilde{\phi}} & Q \ar@{->>}[d]^q  \\
M \ar@{->>}[r]^{\phi} & N  \\
}\]
    in $\cM$ satisfying
    \begin{enumerate}
        \item $\phi:M \to N$ is surjective,
        \item $q:Q \to N$ is a surjection from a projective module onto $N$,
        \item $(q^{-1}(\phi(L)),Q) \in \cP$,
        \item $\tilde{\phi}:P \to Q$ is any map  such that $\phi \circ \pi = q \circ \tilde{\phi}$ (such a $\tilde{\phi}$ always exists by the projectiveness of $P$),
    \end{enumerate}
    then $(\tilde{\phi}(\pi^{-1}(L)),Q) \in \cP$.
\end{itemize}
\end{enumerate} 
\end{notation}

\begin{rem}
\label{rem:cohereditaryhypotheses}
In making this definition, the examples we have in mind are:
\begin{enumerate}
    \item $\cM=$ $R$-modules, $\cP=$ all nested pairs where the second module is projective, and $\cP'=$ all nested pairs in $\cM$.
    \item $\cM=$ \fg\ $R$-modules, $\cP=$ all nested pairs in $\cM$ where the second module is projective, and $\cP'=$ all nested pairs in $\cM$.
    \item $\cM=$ graded $R$-modules with graded $R$-module homomorphisms, $\cP=$ all nested pairs in $\cM$ where the second module is graded projective, and $\cP'=$ all nested pairs in $\cM$.
\end{enumerate}
The last example demonstrates the need for some of our hypotheses. Working this generally also allows us to include cases where $\cP'$ consists of finite-length modules or $\m$-primary modules over a local ring.
\end{rem}

\begin{prop}
\label{pr:definecohereditaryversion}
Let $\cM$, $\cP$, and $\cP'$ be as in Notation \ref{not:coheredpairs}, and let $p$ be a cofunctorial pair operation defined on $\cP$.
Then we can define a cohereditary, cofunctorial pair operation $\po p {\ch}$ on $\cP'$ as follows: for a pair $(L,M) \in \cP'$, let $\pi:P \to M$ be a surjection in $\cM$ with $P$ projective such that $(\pi^{-1}(L),P) \in \cP$. Define
\[\po p {\ch} (L, M) :=\pi(p(\pi^{-1}(L),P)).\]

 In particular, if $p$ is a cohereditary, cofunctorial  pair operation defined on $\cP'$, then $\po p {\ch}=p$.
\end{prop}

\begin{proof}
We first show that $p_c(L,M)$ is well-defined, that is, independent of our choice of $\pi$ and $P$.  Accordingly, suppose $\pi: P \onto M$ and $\pi': P' \onto M$ in $\cM$, where $P, P'$ are projective and $(\pi^{-1}(L),P),((\pi')^{-1}(L),P') \in \cP$.  Let $z\in \po p {\ch}(L, M)$ with respect to $\pi$.  That is, there is some $y \in p(\pi^{-1}(L),P)$ with $\pi(y) = z$.  

Since $\pi'$ is surjective and $P$ is projective, there is a map $\phi: P \ra P'$ in $\cM$ with $\pi = \pi' \circ \phi$.  
By the cofunctoriality hypothesis, 
\[p(\pi^{-1}(L),P)=p((\pi' \circ \phi)^{-1}(L),P)=p(\phi^{-1}((\pi')^{-1}(L)),P) \subseteq \phi^{-1}(p((\pi')^{-1}(L),P')).\]
So $\phi(y) \in p((\pi')^{-1}(L),P')$. Further, $\pi'(\phi(y))=\pi(y)=z$. So $z \in \po p {\ch}(L,M)$ with respect to $\pi'$. Hence $\po p {\ch}(L,M)$ with respect to $P$ and $\pi$ is contained in $\po p {\ch}(L,M)$ with respect to $P'$ and $\pi'$, and by symmetry they are equal.

Next we show that $\po p {\ch}$ satisfies the isomorphism criterion. Suppose $\phi:M \rightarrow M'$ is an isomorphism, $(L,M) \in \cP'$, and $\pi:P \onto M$ with  $P$ projective and $(\pi^{-1}(L),P)=((\phi\circ\pi)^{-1}(\phi(L)),P) \in \cP$. Then  
\begin{align*}
    \phi(\po{p}{\ch}(L,M))&=\phi(\pi(p(\pi^{-1}(L),P)))=\phi(\pi(p(\pi^{-1}(\phi^{-1}(\phi(L))),P)))\\
    &=\phi\circ\pi(p((\phi\circ\pi)^{-1}(\phi(L)),P))=\po{p}{\ch}(\phi(L),M'),\\
\end{align*} which shows that $\po{p}{\ch}$ is invariant under isomorphisms. 

Now we show that $\po p {\ch}$ is cohereditary.  Let $L \subseteq N \subseteq M$ with \[(N,M),(N/L,M/L) \in \cP',\] and let $\pi: P \onto M$ in $\cM$ with $P$ projective such that $(\pi^{-1}(N),P) \in \cP$.  Let $q: M \onto M/L$ be the canonical surjection.
Note that $q^{-1}(N/L)=N$, so that $(q^{-1}(N/L),M) \in \cP'$, and $(q \circ \pi)^{-1}(N/L)=\pi^{-1}(N)$, so that $((q \circ \pi)^{-1}(N/L),P) \in \cP$. Then 
\begin{align*}
\po p {\ch}(N/L, M/L) &= (q \circ \pi) (p((q \circ \pi)^{-1}(N/L),P))\\
 &= q(\pi p(\pi^{-1}(q^{-1}(N/L)),P))) \\
 &= q( \po p {\ch}(q^{-1}(N/L),M))\\
& = \displaystyle\frac{\po p {\ch} (q^{-1}(N/L),M)+L}{L}\\
&=\displaystyle\frac{\po p {\ch} (N,M)+L}{L}\
\end{align*}
Hence, $\po p {\ch}$ is cohereditary.

We now show that $\po p {\ch}$ is cofunctorial. Accordingly, let $L \subseteq N$ and let $\phi: M \to N$ be a homomorphism in $\cM$ with $(L,N),(\phi^{-1}(L),M) \in \cP'$. Choose maps $\pi : P \onto M$ and $q: Q \onto N$ in $\cM$ with $P,Q$ projective such that $(\pi^{-1}(\phi(L)),P),(q^{-1}(L),Q) \in \cP$.  Then since $P$ is projective and $q$ is surjective, there is some map $\tilde \phi: P \ra Q$ in $\cM$ with $q \circ \tilde \phi = \phi \circ \pi$.  

Now let $z\in \po p {\ch} (\phi^{-1}(L),M)$. Then there is some $y \in p(\pi^{-1}(\phi^{-1}(L)), P)$ with $z=\pi (y)$.  Then by cofunctoriality of $p$ when the ambient modules are projective, 
\begin{align*}p(\pi^{-1}(\phi^{-1}(L)),P)&=p((\phi \circ \pi)^{-1}(L),P)
=p((q \circ \tilde{\phi})^{-1}(L),P)\\
&=p(\tilde{\phi}^{-1}(q^{-1}(L)),P)\subseteq \tilde{\phi}^{-1}(p(q^{-1}(L)),Q).\end{align*}
So $\tilde{\phi}(y) \in p(q^{-1}(L),Q)$. Hence $\phi(z)=q(\tilde{\phi}(y)) \in \po p {\ch}(L,N)$, which shows that $\po p {\ch}$ is cofunctorial.

To see that $p=\po p{\ch}$ when $p$ is cohereditary, first we show that if $P$ is projective, $K \subseteq L \subseteq P$, $(L,P) \in \cP$, and $(L/K,P/K) \in \cP'$, then $p(L/K,P/K)=\po p {\ch} (L/K,P/K)$.  Let $\pi: P \rightarrow P/K$.
Then \[\po p{\ch}(L/K,P/K)=\pi(p(L,P))=\displaystyle\frac{p(L,P)+K}{K}=p(L/K,P/K)\]
since $p$ is cohereditary.  

Now let $(N,M) \in \cP'$ and $\pi:P \twoheadrightarrow M$ in $\cM$ with $P$ projective and $(\pi^{-1}(N),P) \in \cP$. Set $\ker(\pi)=K$ so that $\bar{\pi}:P/K \to M$ is an isomorphism. 
Then setting $L/K=\bar{\pi}^{-1}(N)$ where $K \subseteq L \subseteq P$, we have 
\[
\po p{\ch}(N,M)=\po p{\ch}(L/K,P/K)=p(L/K,P/K)=p(N,M). \qedhere
\]
\end{proof}

\begin{prop}\label{pr:chorderpreserving}
Let $\cM$, $\cP$, and $\cP'$ be as in Proposition~\ref{pr:definecohereditaryversion}.
Let $p$ be a cofunctorial (eq. functorial) pair operation defined on $\cP$ that is order preserving on submodules. 
Then $\po p {\ch}$ is a cohereditary cofunctorial pair operation that is order-preserving on submodules. In particular, if for morphisms $g:M \to M'$ in $\cM$, $(L,M) \in \cP'$, then $(g(L),g(M)) \in \cP'$, and if for surjections $\pi:M \to M'$ in $\cM$, if $(L,M),(\pi(L),M') \in \cP'$, then $(\pi^{-1}(\pi(L)),M) \in \cP'$, then
$\po p {\ch}$ is functorial.
\end{prop}

\begin{proof}
By Proposition \ref{pr:definecohereditaryversion}, it suffices to show that $\po p {\ch}$ is order-preserving on submodules. The final comment follows by Proposition \ref{pr:cofunctorial}(\ref{it:funceq}). 

Suppose $L \subseteq N \subseteq M$ and $(L,M), (N,M) \in \cP'$  and suppose  $P$ is a projective module such that $\pi:P \onto M$ in $\cM$ and  $(\pi^{-1}(L),P), (\pi^{-1}(N),P) \in \cP$ as in Notation \ref{not:coheredpairs}.

Since $\pi^{-1}(L) \subseteq \pi^{-1}(N) \subseteq P$, and $p$ is order-preserving on submodules then $p(\pi^{-1}(L),P) \subseteq p(\pi^{-1}(N),P)$. 
Now observe that 
\[\po p{\ch}(L,M)=\pi(p(\pi^{-1}(L),P)) \subseteq \pi(p(\pi^{-1}(N),P))= \po p{\ch}(N,M)
\]
implying that $\po p{\ch}$ is order-preserving on submodules.
\end{proof}

\begin{cor}
\label{cor:cohereditaryextensive}
Let $\cM$, $\cP$, and $\cP'$ be as in Proposition~\ref{pr:definecohereditaryversion}, and assume too that for every $(L,M) \in \cP'$, $(0,M/L) \in \cP'$.
Let $p$ be a cofunctorial pair operation defined on $\cP$ that is extensive. 
Then $\po p {\ch}$ is a residual, cofunctorial pair operation on $\cP'$.
In particular, if $p$ is a residual, cofunctorial  pair operation defined on $\cP'$, then $\po p {\ch}=p$.
\end{cor}

\begin{proof}
Suppose $(L,M) \in \cP'$ and $\pi:P \onto M$ is a surjective homomorphism in $\cM$ with $(\pi^{-1}(L),P) \in \cP$. Since $p$ is extensive \[L = \pi(\pi^{-1}(L)) \subseteq \pi(p( \pi^{-1}(L), P))=\po p {\ch}(L,M)\] implying that  $\po p {\ch}$ is extensive.
Now by Lemma~\ref{lem:resiffextcoher}, an extensive cohereditary pair operation is residual.  Thus by Proposition~\ref{pr:definecohereditaryversion}, $\po p {\ch}$ is residual and cofunctorial.  If $p$ is a cofunctorial residual pair operation, then $p$ is a cofunctorial extensive cohereditary pair operation by Lemma~\ref{lem:resiffextcoher} and again Proposition~\ref{pr:definecohereditaryversion} gives us that $\po p {\ch}=p$ when $p$ is residual.  
\end{proof}

\begin{notation}
When $p$ is extensive, we will denote $\po p {\ch}$ by $\po p {\re}$ from now on.

When $p=\cl$ is a closure operation and $L \subseteq M$, we will denote $p(L,M)=L_M^{\cl}$.
\end{notation}

\begin{prop}
\label{pr:defineresidualversion}
Let $\cM$, $\cP$, and $\cP'$ be as in Proposition \ref{pr:definecohereditaryversion}, and $\cl$ a cofunctorial closure operation defined on $\cP$. Then the cohereditary version $\po {\cl} {\re}$ of $\cl$ is a residual, cofunctorial closure operation on $\cP$ given as follows:
\[\rcl L M :=\pi((\pi^{-1}(L))^\cl_P).\]
In particular, if $\cl$ is itself residual, then $\cl=\po {\cl} {\re}$.
\end{prop}

\begin{proof}
Since every 
cofunctorial closure operation is a 
cofunctorial pair operation, then Corollary \ref{cor:cohereditaryextensive} gives us that $\po {\cl} {\re}$ is a cofunctorial residual pair operation on $\cP'$ and if $\cl$ is a residual closure on $\cP'$, then $\cl=\po {\cl} {\re}$.  By Proposition 
\ref{pr:chorderpreserving} and Corollary \ref{cor:cohereditaryextensive}, we see that $\po{\cl}{\re}$ is extensive and order-preserving on submodules.  
When we have shown $\po{\cl}{\re}$ is idempotent, our proof will be complete.

Let $(N, M) \in \cP'$ and $P$ be a projective $R$-module with submodule $L$ such that $M \cong P/K$ and $N \cong L/K$ in $\cM$, and $(L,P) \in \cP$.  It will be enough to show that $\po{\cl}{\re}$ is idempotent on submodules of  $P/K$.  Suppose $\pi:P \onto P/K$, then by definition, \[(L/K)_{P/K}^{\po{\cl}{\re}}=\displaystyle\frac{L_P^{\cl}+K}{K}=\displaystyle\frac{L_P^{\cl}}{K}\] where the last equality holds because $\cl$ is extensive.
Now applying $\po{\cl}{\re}$ to $(L/K)_{P/K}^{\po{\cl}{\re}}$ and using the above equality and the idempotence of $\cl$ we obtain \[((L/K)_{P/K}^{\po{\cl}{\re}})_{P/K}^{\po{\cl}{\re}}=\displaystyle\frac{(L_P^{\cl})_P^{\cl}}{K}=\displaystyle\frac{L_P^{\cl}}{K}=(L/K)_{P/K}^{\po{\cl}{\re}}\] 
concluding our proof.
\end{proof}

\begin{notation}
When $\int$ is an interior operation, we denote $L_{\int}^M:=\int(L,M)$.
\end{notation}

\begin{prop}
\label{pr:definecohereditaryinterior}
Let $\cM$, $\cP$, and $\cP'$ be as in Notation \ref{not:coheredpairs}, and let $\int$ be a cofunctorial (eq. functorial) interior operation defined on $\cP$. Assume that whenever $(L,P) \in \cP$ and $(L/K,P/K) \in \cP'$, then $(L_{\int}^P,P), (L_{\int}^P+K,P) \in \cP$.
Then the cohereditary version $\int_{\ch}$ of $\int$ is a cohereditary, cofunctorial (eq. functorial) interior operation given by:
\[L_{\po {\int}{\ch}}^M :=\pi((\pi^{-1}(L))^P_{\int}).\]
In particular, if $\int$ is a cohereditary, cofunctorial (eq. functorial) interior operation on $\cP'$, then $\int_{\ch}=\int$.
\end{prop}

\begin{proof}
First we show that $\po {\int}{\ch}$ is intensive. Note that if $(L,M) \in \cP'$ and $\pi:P \onto M$ is a surjective homomorphism in $\cM$ with $(\pi^{-1}(L),P) \in \cP$, then since $\int$ is intensive \[L_{\po{\int}{\ch}}^M= \pi((\pi^{-1}(L))_{\int}^P)\subseteq \pi(\pi^{-1}(L))=L\] implying that  $\po {\int} {\ch}$ is intensive.
 
Proposition \ref{pr:chorderpreserving} implies that $\po {\int}{\ch}$ is order-preserving on submodules. Thus to show that $\po{\int}{\ch}$ is a cohereditary interior operation, we need only show that $\po{\int}{\ch}$ is idempotent.  Suppose $P$ is a projective $R$-module and $K \subseteq L \subseteq P$ are such that $(L,P) \in \cP$ and $(L/K, P/K) \in \cP'$.  By definition of $\po{\int}{\ch}$ and intensivity we have
 \[(L/K)_{\po {\int}{\ch}}^{P/K}=\displaystyle\frac{L_{\int}^P+K}{K} \subseteq L/K\]
 Assuming further that $((L/K)_{\po{\int}{\ch}}^{P/K}, P/K) \in \cP'$, then 
  \[(L/K_{\po{\int}{\ch}}^{P/K})_{\po{\int}{\ch}}^{P/K} \subseteq (L/K)_{\po{\int}{\ch}}^{P/K}\]
  as $\po{\int}{\ch}$ is order-preserving on submodules.  
  We also know that $L_{\int}^P \subseteq L_{\int}^P+K$, so since $(L_{\int}^P,P), (L_{\int}^P+K,P) \in \cP$ by hypothesis, then using the fact that $\int$ is idempotent we have
  \[L_{\int}^P=(L_{\int}^P)_{\int}^P \subseteq (L_{\int}^P+K)_{\int}^P\]
  which implies
  \begin{align*}(L/K)_{\po{\int}{\ch}}^{P/K}&=\displaystyle\frac{L_{\int}^P+K}{K}=\displaystyle\frac{(L_{\int}^P)_{\int}^P+K}{K}\\
  &\subseteq \displaystyle\frac{(L_{\int}^P+K)_{\int}^P+K}{K}=((L/K)_{\po{\int}{\ch}}^{P/K})_{\po{\int}{\ch}}^{P/K}.
  \end{align*}
  Now for any $(N,M)\in \cP'$, with $M \cong P/K$ and $N \cong L/K$ in $\cM$, we have \[(N_{\po{\int}{\ch}}^M)_{\po{\int}{\ch}}^M = N_{\po{\int}{\ch}}^M,\] which implies $\po{\int}{\ch}$ is idempotent and hence a cohereditary functorial 
  interior operation.
\end{proof}

\begin{prop}\label{pr:sfcohereditarypairineq}
Let $\cM$, $\cP$, and $\cP'$ be as in Notation \ref{not:coheredpairs}, let $p$ be a cofunctorial pair operation defined on $\cP'$, and let $\po p {\ch}$ be its cohereditary version on $\cP'$.  Then $\po p {\ch} \leq p$.
\end{prop}

\begin{proof}
Suppose $(L, M) \in \cP'$, and let $\pi:P \onto M$ be a surjection in $\cM$ such that $(\pi^{-1}(L),P) \in \cP$. By definition $\po p {\ch}(L,M)= \pi(p(\pi^{-1}(L), P))$.  Since $p$ is cofunctorial, then 
\[\pi(p(\pi^{-1}(L), P)) \subseteq \pi(\pi^{-1}(p(L,M)))=p(L,M)\]
since $\pi$ is surjective.
Hence, $\po p{\ch} \leq p$.
\end{proof}

\begin{cor}\label{pr:rclincl}
Let $\cl$ be a cofunctorial (eq. functorial) closure operation defined on $\cP$. Then $\po {\cl} {\re} \le \cl$.
\end{cor}

\begin{proof}
Since $\cl$ is both cofunctorial and extensive, and cohereditary, extensive operations are residual by Lemma \ref{lem:resiffextcoher}, this is a direct consequence of Proposition~\ref{pr:sfcohereditarypairineq}.
\end{proof}

\begin{example}\label{ex:lirverstrict}
We note here that the inequality in Proposition~\ref{pr:rclincl} can be strict.  Namely, let $R$ be any Noetherian local domain that is not a DVR or a field, let $L$ be a non-integrally closed ideal of $R$, and let $N$ be the integral closure of $L$.  Then the usual integral closure of $L$ in $N$ is $N$ itself, but the \emph{liftable} (i.e. residual) integral closure of $L$ in $N$ cannot be $N$. (See Definition \ref{def:li}.)  Indeed, by \cite[Proposition 2.4 (11)]{nmeUlr-lint}, we have $\lic LN \subseteq L + \m N$, so by Nakayama's lemma we cannot have $\lic LN = N$.

For a concrete example, let $R=k[\![x,y]\!]$ where $k$ is any field, let $N$ be the ideal $(x^2, xy, y^2)$, and let $L$ be the sub-ideal $(x^2, y^2)$.  Then we have a surjection $\pi: R \ra N/L$ given by $1 \mapsto xy + L$.  The kernel is $L:N = (x^2, y^2):xy = (x,y) = \m$, so $\lic LN / L = \pi(\m^-_R) = \pi(\m) = 0$, whence $\lic LN = L$.

Indeed we can expand this example to give a case where $0^\cl_M$ and $0^{\cl_{\rm r}}_M$ can differ.  For this, let $M = N/L$, where $R=R=k[\![x,y]\!]$, $N=(x^2, xy, y^2)$, and $L=(x^2, y^2)$.  Then as above, $\lic 0 M = \frac{\lic LN}L = 0$, but $0^{-\rm{EHU}}_M = M$. To see the latter equality, note first that $\Hom_R(M,R)=0$.  This is because if $g:M \ra R$ is $R$-linear and $c=g(xy+L)$, then since $x^2 y \in L$, we have $0 = g(x^2y+L) = g(x(xy+L)) = xg(xy+L) = xc$, but since $x$ is an $R$-regular element, it follows that $c=0$.  Thus, $\Hom_R(M,F)=0$ for any free module $F$, so that $0^{-\rm{EHU}}_M = M$ by Definition~\ref{def:EHU}.  Since $\li = {\rm EHU}_r$, we are done.
\end{example}

In Example \ref{ex:residualversionsdisagree}, we present an example where the residual versions of two closures agree on a particular pair, but the original closures disagree on the same pair.

\begin{lemma}\label{lem:compch}
Let $\cM$, $\cP$, and $\cP'$ be as in Notation \ref{not:coheredpairs}.
If $p \leq q$ are cofunctorial pair operations defined on $\cP$, then $\po{p}{\ch} \leq \po{q}{\ch}$ on $\cP'$.  In particular, if $c \leq d$ are closure operations on $\cP$, then $\po{c}{\re} \leq \po{d}{\re}$ on $\cP'$.
\end{lemma}

\begin{proof}
If $p \leq q$, then for all $(N,P) \in \cP$, $p(N,P) \subseteq q(N,P)$.  If $(L,M) \in \cP'$ then there exists a projective module $P$ and surjection $\pi:P \onto M$ in $\cM$ such that $\po{p}{\ch}(L,M)=\pi(p(\pi^{-1}(L),P))$ and $\po{q}{\ch}(L,M)=\pi(q(\pi^{-1}(L),P))$.  Since $p(\pi^{-1}(L),P) \subseteq q(\pi^{-1}(L),P)$, then \[\po{p}{\ch}(L,M)=\pi(p(\pi^{-1}(L),P)) \subseteq \pi(q(\pi^{-1}(L),P))=\po{q}{\ch}(L,M)\] giving us the result.
\end{proof}

\begin{lemma}\label{lem:inverseimagercl}
Let $\cM$, $\cP$, and $\cP'$ be as in Notation \ref{not:coheredpairs}. Let $\cl$ be a cofunctorial (eq. functorial) closure operation on $\cP$ with residual version $\po {\cl} {\re}$.  Let $(L,M) \in \cP'$.  Then for any projective module $P$ and surjection $\pi: P \onto M$ in $\cM$ with $(\pi^{-1}(L),P) \in \cP$, we have $\pi^{-1}(\rcl L M) \subseteq \pi^{-1}(L)^\cl_P$.
\end{lemma}

\begin{proof}
Let $y\in P$ with $\pi(y) \in \rcl LM$.  Let $q: M \onto M/L=:\bar M$ be the natural surjection. 
Then by functoriality of $\po {\cl} {\re}$ and by definition, we have \[
(q \circ \pi)(y) \in q(\rcl L M) \subseteq \rcl 0 {M/L} = (q \circ \pi)\left((\ker (q \circ \pi))^\cl_P\right) = (q \circ \pi)\left(\pi^{-1}(L)^\cl_P\right).
\]
Hence, there is some $a \in \pi^{-1}(L)^\cl_P$ with $q(\pi(y)) = q(\pi(a))$.  That is, $\pi(y-a) \in \ker q = L$.  Thus, $y-a \in \pi^{-1}(L)$, whence \[
y=(y-a)+a \in \pi^{-1}(L)+ \pi^{-1}(L)^\cl_P = \pi^{-1}(L)^\cl_P. \qedhere\]
\end{proof}

The next result allows us to extend the notion of a $\cl$-core from a functorial closure operation to its residual version.

\begin{prop}
\label{prop:residualversionnakayama}
Let $(R,\m)$ be a Noetherian local ring. Let $\cM$ be the category of \fg\ $R$-modules with $R$-module maps. Let $\cl$ be a cofunctorial (eq. functorial) Nakayama closure operation on pairs of \fg\ $R$-modules with residual version $\po {\cl} {\re}$.  Then $\po {\cl} {\re}$ is Nakayama as well.
\end{prop}

\begin{proof}
Since $\po{\cl}{\re}$ is residual, by Lemma \ref{lem:residualnakayama} it is enough to show that if $L \subseteq M$ are \fg\ modules with $L \subseteq (\m L)_M^{\po{\cl}{\re}}$, we have $L \subseteq 0_M^{\po{\cl}{\re}}$. Let $P$ be a \fg\ free module and $\pi:P \twoheadrightarrow M$ a surjection. Let $K=\ker \pi$. We claim that $\pi^{-1}(\m L)=K+\m \pi^{-1}(L)$. To see this, first note that $\pi(K+\m \pi^{-1}(L)) \subseteq \pi(K)+\m \pi(\pi^{-1}(L)) \subseteq \m L$, whence $K+\m \pi^{-1}(L) \subseteq \pi^{-1}(\m L)$. For the reverse inclusion, let $x \in \pi^{-1}(\m L)$. That is, $\pi(x) \in \m L$. Then there exist $y_j \in \m$ and $z_j \in L$ such that $\pi(x)=\sum_{j=1}^n y_jz_j$. By the surjectivity of $\pi$, there exist $z_j' \in \pi^{-1}(L)$ with $\pi(z_j')=z_j$. It follows that $x-\sum_{j=1}^n y_jz_j' \in \ker \pi=K$. Thus $x \in K+\m \pi^{-1}(L)$.

Now let $z \in \pi^{-1}(L)$. Then since $\pi(z) \in L \subseteq (\m L)_M^{\po{\cl}{\re}}$, by Lemma \ref{lem:inverseimagercl} we have $z \in (\pi^{-1}(\m L))_P^{\cl}=(K+\m \pi^{-1}(L))_P^\cl$. Thus we have $K \subseteq \pi^{-1}(L) \subseteq (K+\m \pi^{-1}(L))_P^\cl$, so by the Nakayama property for cl $\pi^{-1}(L) \subseteq K_P^\cl=(\ker \pi)_P^\cl$. Then by the definition of $\cl_r$, $L=\pi(\pi^{-1}(L)) \subseteq \pi((\ker \pi)_P^\cl)=0_M^{\po{\cl}{\re}}$.
\end{proof}

Liftable integral closure and integral closure are Nakayama closures.  Hence, both the liftable integral-core and the (integral)-core of a finitely generated module exist.  (See for example \cite{CPU2002}, \cite{foulivassilevclcore} and \cite{ERGV-chdual} for definitions of core and $\cl$-core.)  We will illustrate below that these can differ by continuing our analysis of
Example \ref{ex:lirverstrict}.
\begin{example}
We continue with the notation of Example \ref{ex:lirverstrict}. Note, that not only can we see that $\lic L{\m^2}=L$ for $L=(x^2,y^2)$, but the same argument we used in Example \ref{ex:lirverstrict},  will illustrate that $\lic L{\m^2}=L$ for any minimal (integral reduction) $L$ of $\m^2$.  Hence, $\li\core_{\m^2}{\m^2}=\m^2$, because the only $\li$-reduction of $\m^2$ in $\m^2$ is $\m^2$.  However, by \cite[Example 3.2]{CPU2002}, ${\rm core}_R(\m^2)=\m^3={\rm core}_{\m^2}(\m^2)$, since all the minimal reductions of $\m^2$ in $R$ are minimal reductions of $\m^2$ in $\m^2$,  illustrating that ${\rm core}_{\m^2}(\m^2)$ is properly contained in $\li\core_{\m^2}(\m^2)$.
\end{example}


\begin{prop}
\label{pr:clrfinitistic}
Let $\cl$ be a cofunctorial (eq. functorial) closure operation on a class $\cP$ of pairs of $R$-modules that is finitistic, and let $\po{\cl}{\re}$ be its residual version, defined on a class $\cP'$ of pairs of $R$-modules. If $(0,N/L) \in \cP'$ if and only if $(L,N) \in \cP'$, then $\po{\cl}{\re}$ is also finitistic.
\end{prop}

\begin{proof}
First we prove that if $(0,M) \in \cP'$, then \[0_M^{\po{\cl}{\re}}=\bigcup_{(0,N) \in \cP', N \subseteq M \text{ f.g.}} 0_N^{\po{\cl}{\re}}.\]
Let $(0,M) \in \cP'$, and let $\pi:P \to M$ be a surjection from a projective module such that $(\ker(\pi),P) \in \cP$. Let $x \in 0_M^{\po{\cl}{\re}}=\pi((\ker(\pi))_P^{\cl})$. Since $\cl$ is finitistic, there is some $\ker \pi \subseteq U \subseteq P$ such that $U/\ker(\pi)$ is \fg, $(\ker(\pi),U) \in \cP$, and $x \in \pi((\ker(\pi))_U^{\cl})$. Since $U$ is a submodule of a projective module, $x \in \pi((\ker(\pi))_U^{\po{\cl}{\re}})$. Since $\po{\cl}{\re}$ is functorial by Proposition \ref{pr:definecohereditaryversion}, and by our hypothesis $(0,\pi(U)) \in \cP'$, we have $\pi((\ker(\pi))_U^{\po{\cl}{\re}}) \subseteq 0_{\pi(U)}^{\po{\cl}{\re}}$. Since $\pi(U) \cong U/\ker(\pi)$, it is \fg\ and $(0,\pi(U)) \in \cP'$, and so $\pi(U)$ is a \fg\ submodule $N$ of $M$ such that $x \in 0_N^{\po{\cl}{\re}}$. This proves our first claim.

Next we prove that this implies that $\po{\cl}{\re}$ is finitistic on $\cP'$. Let $(L,M) \in \cP'$ such that $(0,M/L) \in \cP'$. Since $\po{\cl}{\re}$ is residual, $L_M^{\po{\cl}{\re}}=\pi^{-1}(0_{M/L}^{\po{\cl}{\re}})$, where $\pi:M \to M/L$ is the quotient map. Let $x \in L_M^{\po{\cl}{\re}}$. By our first claim, there is some \fg\ $N/L \subseteq M/L$ such that $(0,N/L) \in \cP'$ and $x \in \pi^{-1}(0_{N/L}^{\po{\cl}{\re}})$. Let $q:N \to N/L$ be the quotient map. We claim that $x \in q^{-1}(0_{N/L}^{\po{\cl}{\re}})$. We know that $q$ agrees with the composition $N \hookrightarrow M \twoheadrightarrow M/L$. Since $x \in \pi^{-1}(0_{N/L}^{\po{\cl}{\re}})$, $x \in \pi^{-1}(N/L) = N$. 
So $x \in L_N^{\po{\cl}{\re}}=q^{-1}(0_{N/L}^{\po{\cl}{\re}})$, as desired.
\end{proof}

\section{Hereditary versions of pair operations}\label{sec:hereditaryversions}

As the duality for pair operations discussed in \cite{ERGV-chdual,ERGV-nonres} only holds for complete local rings, we cannot make use of our duality to define hereditary and absolute pair operations in full generality.  Hence, the proofs in this section are necessary but precisely dual to those in Section~\ref{sec:residualversion}. We will explore this duality further in Section \ref{sec:duality}. We end the section by showing that the finitistic version of a hereditary closure operation is also hereditary, which we will need in Section \ref{sec:integralclosure}.

\begin{notation}
\label{not:heredpairs}
Throughout this section, we assume the following:
\begin{enumerate}
    \item $\cM$ will denote an abelian category of $R$-modules that has enough injectives.
    \item $\cP$ will denote a set of pairs of $R$-modules $(C,E)$ with $C \subseteq E$, $C,E \in \cM$ and $E$ injective.
    \item $\cP'$ will denote a set of pairs such that $\cP \subseteq \cP'$ and for every $(A,B) \in \cP'$ with $A \subseteq B$, there is an injective module $E$ and an injection $i:B \to E$ in $\cM$ with $(i(A),E) \in \cP$.
    \item In addition, if $(A,B),(C,B) \in \cP'$ with $A \subseteq C$, there is some injective module $E$ and an injection $i:B \to E$ in $\cM$ such that $(i(A),E),(i(C),E) \in \cP$.
    \item We further assume:
    \begin{itemize}
    \item[(**)] Given $A,B,E,i$ as above, if we have a commutative diagram in $\cM$
    \[\xymatrix{
E  \ar[r]^{\tilde{\psi}} & F  \\
B \ar@{^{(}->}[u]^i \ar@{^{(}->}[r]^\psi & C \ar@{^{(}->}[u]^j \\
}\]
    satisfying
    \begin{enumerate}
        \item $\psi:B \to C$ is injective,
        \item $j:C \into F$ where $F$ is an injective module,
        \item $(j(\psi(A)),F) \in \cP$,
        \item $\tilde{\psi}:E \to F$ is any map  such that $\tilde{\psi} \circ i=j \circ \psi$ (such a $\tilde{\psi}$ always exists by the injectiveness of $F$),
    \end{enumerate}
    then $(\tilde{\psi}^{-1}(j(\psi(A))),E) \in \cP$.
\end{itemize}
\end{enumerate}
\end{notation}

\begin{rem}
\label{rem:hereditaryhypotheses}
In making this definition, the examples we have in mind are:
\begin{enumerate}
    \item $\cM=$ $R$-modules, $\cP=$ all pairs where the second module is injective, and $\cP'=$ all pairs in $\cM$.
    \item $R$ is local, $\cM=$ Artinian $R$-modules, $\cP=$ all pairs in $\cM$ where the second module is injective, and $\cP'=$ all pairs in $\cM$.
    \item $\cM=$ graded $R$-modules with graded $R$-module homomorphisms, $\cP=$ all graded injective modules, and $\cP'=$ all pairs in $\cM$.
\end{enumerate}
As with Remark \ref{rem:cohereditaryhypotheses}, we can also apply our results in much more restrictive cases.
\end{rem}

\begin{prop}
\label{pr:definehereditaryversion}
Let $\cM$, $\cP,$ and $\cP'$ be as in Notation \ref{not:heredpairs} and 
let $p$ be a functorial pair operation defined on $\cP$. 
Then we can define a hereditary, functorial pair operation $\po p {\h}$ as follows: For a pair $(A,B) \in \cP'$, let $i:B \to E$ be a monomorphism in $\cM$ with $E$ injective and $(i(A),E) \in \cP$; we define
\[\po p {\h} (A,B) :=i^{-1}(p(i(A),E)).\]
Viewing $i$ as the inclusion map, this is equal to $p(A,E) \cap B.$

In particular, if $p$ is a hereditary, functorial  pair operation defined on $\cP'$, then $\po p {\h}=p$.
\end{prop}

\begin{proof}
We first show that this is well-defined, independent of $i$ and $E$.  Accordingly, let $(A,B) \in \cP'$ and $i: B \into E$ and $i': B \into E'$ be inclusions in $\cM$, where $E, E'$ are injective and $(i(A),E),(i'(A),E') \in \cP$.

Since $i'$ is injective and $E$ is injective, there is some $R$-linear map $\phi:E' \to E$ in $\cM$ with $i=\phi \circ i'$. Let $z \in \po p {\h}(A,B)$ with respect to $i'$, so there is some $y \in p(i'(A),E')$ with $i'(z)=y$. Then $i(z)=(\phi \circ i')(z)=\phi(y)$. We have $\phi(y) \in \phi(p(i'(A),E'))$. By functoriality, this is contained in $p(\phi(i'(A)),E)=p(i(A),E)$. But then $i(z) \in p(i(A),E)$, so $z \in \po p {\h}(A,B)$ with respect to $E$, as desired. This gives us one inclusion, and the other follows by symmetry.

Next we show that $\po p {\h}$ satisfies the isomorphism criterion. Suppose $\phi:M \rightarrow M'$ is an isomorphism, $(L,M), (\phi(L),M') \in \cP'$, and $i:M\into E$ with  $E$ injective and $(i(L),E)=((i\circ \phi^{-1})(\phi(L)),P) \in \cP$. Then  
\begin{align*}
    \phi(\po{p}{\h}(L,M))&=\phi(i^{-1}(p(i(L),E)))=\phi(i^{-1}(p(i(\phi^{-1}(\phi(L))),E)))\\
    &=(i\circ \phi^{-1})^{-1}(p((i\circ\phi^{-1})(\phi(L)),E))=\po{p}{\h}(\phi(L),M'),\\
\end{align*} which shows that $\po{p}{\h}$ is invariant under isomorphisms. 

Now we show that $\po p {\h}$ is hereditary.  Let $A \subseteq C \subseteq B$ with $(A,C),(A,B) \in \cP'$, and let $i: B \into E$ in $\cM$ with $E$ injective such that $(i(A),E) \in \cP$ as in Notation \ref{not:heredpairs}.  Let $j: C \into B$ be the canonical inclusion.  Then 
\begin{align*}
\po p {\h}(A,C) &= (i \circ j)^{-1} (p((i \circ j)(A),E))\\
 &= j^{-1}(i^{-1}( p(i(j(A)),E))) \\
 &= j^{-1}( \po p {\h}(j(A),B))\\
& = \po p {\h} (j(A),B) \cap C\\
&=\po p {\h}(A,B) \cap C\\
\end{align*}
Note that $j(A)=A$ and  $(i \circ j)(A)=i(A)$, making $(j(A),B) \in \cP'$ and $((i \circ j)(A),E) \in \cP$.
Hence, $\po p {\h}$ is hereditary.

We now show that $\po p {\h}$ is functorial. Accordingly, let $(A,B) \in \cP'$ and let $\phi: B \ra D$ in $\cM$ such that $(\phi(A),D) \in \cP'$. Choose maps $i : B \into E$ and $i': D \into E'$ in $\cM$ with $E,E'$ injective such that $(i(A),E),(i'(\phi(A)),E') \in \cP$.  Then since $E'$ is injective and $i$ is injective, there is some $R$-linear map $\tilde \phi: E \ra E'$ in $\cM$ with $i' \circ \phi = \tilde\phi \circ i$.  Now let $z\in \po p {\h} (A,B)$. Then $y=i(z) \in p(i(A),E)$. Hence $\tilde \phi(y) \in \tilde \phi (p(i(A),E))$. By functoriality of $p$ when the ambient modules are injective, this is contained in $p(\tilde \phi (i(A)),E')=p(i'(\phi(A)),E')$. Since $\tilde{\phi}(y)=\tilde{\phi}(i(z))=i'(\phi(z))$ and $i'$ is injective, $(i')^{-1}(\tilde{\phi}(y))=\{\phi(z)\}.$ This element must be in $(i')^{-1}(p(i'(\phi(A)),E'))=\po p {\h}(\phi(A),D)$.
Hence $\phi(z) \in \po p {\h}(\phi(A),D)$, so $\po p {\h}$ is functorial.

For the final statement, assume that $p$ is hereditary. For $(A,B) \in \cP'$ and $i:B \to E$ in $\cM$ such that $(i(A),E) \in \cP$, we have $\po p {\h} (A,B)=p(A,E) \cap B=p(A,B),$ as desired.
\end{proof}

\begin{prop}
\label{pr:hereditaryorderpreserving}
Let $\cM$, $\cP$, and $\cP'$ be as in 
Notation \ref{not:heredpairs}.  Let $p$ be a functorial pair operation defined on $\cP$ that is order preserving on submodules.  Then $\po p {\h}$ is a hereditary functorial pair operation that is order preserving on submodules.
\end{prop}

\begin{proof}
Since $\po p {\h}$ is a functorial hereditary pair operation by Proposition~\ref{pr:definehereditaryversion}, it suffices to show that for $A \subseteq C \subseteq B$ with $(A,B),(C,B) \in \cP'$,
$\po p {\h}(A,B) \subseteq \po p {\h}(C,B)$. By Notation \ref{not:heredpairs}, there is some injective module $E$ containing $B$, with the inclusion morphism $i$ in $\cM$, such that $(i(A),E),(i(C),E) \in \cP$, and by definition $\po p{\h}(A,B)=p(i(A),E) \cap B$.
Since $p$ is order preserving on submodules we have  $p(i(A),E) \subseteq p(i(C), E)$, and intersecting each of these modules with $B$ we obtain  $\po p{\h}(A,B)=p(i(A),E) \cap B \subseteq p(i(C), E)\cap B=\po p{\h}(C,B)$ which is what we needed to show.  
\end{proof}

\begin{cor}
\label{cor:absolutepair}
Let $\cM$, $\cP$, and $\cP'$ be as in 
Notation \ref{not:heredpairs}. Further assume that if $(A,B) \in \cP'$, then $(A,A) \in \cP'$. Let $p$ be a functorial pair operation defined on $\cP$ that is intensive.  Then $\po p {\h}$ is an absolute, functorial pair operation on $\cP'$. In particular if $p$ is an absolute pair operation defined on $\cP'$, then $\po p {\h}=p$.
\end{cor}

\begin{proof}
Suppose $(A,B) \in \cP'$ and $i:B \to E$ is an injective homomorphism in $\cM$ with $(i(A),E) \in \cP$. Since $p$ is intensive,
\[\po p {\h}(A,B)=p(i(A),E) \cap B \subseteq i(A) \cap B=A,\]
implying that $\po p {\h}$ is intensive.
By Lemma~\ref{lem:absiffinther}, an intensive hereditary pair operation is absolute.  Thus by Proposition~\ref{pr:definehereditaryversion}, $\po p {\h}$ is absolute and functorial.  If $p$ is a functorial absolute pair operation, then $p$ is a functorial intensive hereditary pair operation by Lemma~\ref{lem:absiffinther} and again Proposition~\ref{pr:definehereditaryversion} gives us that $\po p {\h}=p$ when $p$ is residual.  
\end{proof}


\begin{notation}
As intensive hereditary pair operations are absolute by Lemma~\ref{lem:absiffinther}, we will denote $\po p{\h}=\po p{\ab}$ when $p$ is additionally intensive.
\end{notation}

\begin{prop}
\label{pr:absoluteinterior}
Let $\cM$, $\cP$, and $\cP'$ be as in Notation \ref{not:heredpairs} and let $\int$ be a functorial interior operation on $\cP$. Then the absolute version $\po{\int}{\ab}$ of $\int$ is an absolute interior operation on $\cP'$.
If $\int$ is defined on $\cP'$ and is absolute, then $\int=\po{\int}{\ab}$.
\end{prop}

\begin{proof}
Since every absolute interior operation is an absolute pair operation, then Corollary \ref{cor:absolutepair} gives us that $\po {\int} {\ab}$ is a functorial absolute pair operation on $\cP'$ and if $\int$ is an absolute interior on $\cP'$, then $\int=\po {\int} {\ab}$.  By Proposition 
\ref{pr:hereditaryorderpreserving} we see that $\po{\int}{\ab}$ is order-preserving on submodules.  When we have shown $\po{\int}{\ab}$ is idempotent, our proof will be complete.

Let $(A,B) \in \cP'$ and $E$ be an injective $R$-module containing $B$ such that the inclusion $i:B \to E$ is in $\cM$ and $(A,E) \in \cP$.  Note that \[A^{B}_{\po{\int}{\ab}}=A^E_{\int} \cap B=A^E_{\int}\] where the last equality holds because $\int$ is intensive. 
Now applying $\po{\int}{\ab}$ to $A^{B}_{\po{\int}{\ab}}$ and using the above equality and the idempotence of $\int$ we obtain \[(A^{B}_{\po{\int}{\ab}})^B_{\po{\int}{\ab}}
=(A^E_{\int})^E_{\int}=A^E_{\int}=A^{B}_{\po{\int}{\ab}},\] 
concluding our proof.
\end{proof}

\begin{prop}
\label{pr:definehereditaryclosure}
Let $\cM$, $\cP$, and $\cP'$ be as in Notation \ref{not:heredpairs} and let $\cl$ be a functorial closure operation on $\cP$. Then the hereditary version $\po{\cl}{\h}$ of $\cl$ is a hereditary closure operation on $\cP'$.
If $\cl$ is defined on $\cP'$ and is hereditary, then $\cl=\po{\cl}{\h}$ on $\cP'$.
\end{prop}

\begin{proof}
Note that if $A \subseteq B \subseteq E$, with $E$ injective, the inclusion morphism in $\cM$, and $(A,E) \in \cP$, since $\cl$ is extensive we have
\[
A_B^{\po{\cl}{\h}}=A_E^{\cl} \cap B \supseteq A \cap B=A
\]
implying that $\po{\cl}{\h}$ is extensive.  Proposition~\ref{pr:hereditaryorderpreserving} implies that $\po{\cl}{\h}$ is order-preserving on submodules.  
To show that $\po{\cl}{\h}$ is a closure operation, we need only show that $\po{\cl}{\h}$ is idempotent. After this we will be done by Proposition~\ref{pr:definehereditaryversion} which tells us that $\po{\cl}{\h}$ is hereditary and that if $\cl$ is hereditary, then $\cl=\po{\cl}{\h}$.

Let $A \subseteq B \subseteq E$ where $E$ is injective with inclusion morphism  in $\cM$, and $(A,E) \in \cP$.  By Proposition~\ref{pr:hereditaryorderpreserving} and the fact that $\po{\cl}{\h}$ is extensive, we have
\[
A_B^{\po{\cl}{\h}} \subseteq (A_B^{\po{\cl}{\h}})_B^{\po{\cl}{\h}}.
\]

Since $A_B^{\po{\cl}{\h}}=A_E^{\cl} \cap B \subseteq A_E^{\cl}$ and $\cl$ is order preserving on submodules,
\[
(A_B^{\po{\cl}{\h}})_E^{\cl}=(A_E^{\cl} \cap B)_E^{\cl} \subseteq (A_E^{\cl})_E^{\cl}=A_E^{\cl}
\]
where the last equality holds by idempotence of $\cl$.
Intersecting with $B$ we obtain
\[
(A_B^{\po{\cl}{\h}})_B^{\po{\cl}{\h}}=(A_B^{\po{\cl}{\h}})_E^{\cl}\cap B \subseteq A_E^{\cl} \cap B=A_B^{\po{\cl}{\h}}
\]
and we conclude that $\po{\cl}{\h}$ is idempotent.
\end{proof}

\begin{prop}\label{pr:sfhereditarypairineq}
Let $p$ be a functorial pair operation defined on $\cP'$.  Then $\po p {\h} \geq p$.
\end{prop}

\begin{proof}
Suppose $A \subseteq B$ with $(A, B) \in \cP$.  Let $E$ be a injective module and $i: B \to E$ an injective $R$-module homomorphism in $\cM$ such that $(i(A),E) \in \cP$. Since $p$ is functorial,
\[i(p(A,B)) \subseteq p(i(A),E).\]
Identifying $p(A,B)$ with its image in $E$ since $i$ is injective,
\[p(A,B) \subseteq p(i(A),E) \cap B=\po p {\h}(A,B).\]
Hence, $\po p{\h} \geq p$.
\end{proof}

\begin{cor}\label{pr:aintineq}
Let $\int$ be a functorial interior operation defined on $\cP'$. Then $\po {\int} {\ab} \geq \int$.
\end{cor}

\begin{proof}
This is a direct consequence of Proposition~\ref{pr:sfhereditarypairineq}.
\end{proof}

The above inequality may be strict; we illustrate this in Example \ref{ex:absineqex}.

\begin{lemma}\label{lem:comph}
    Let $\cM$, $\cP$, and $\cP'$ be as in Notation \ref{not:heredpairs}.
   If $p \leq q$ are pair operations defined on $\cP$, then $\po{p}{\h} \leq \po{q}{\h}$ on $\cP'$.  In particular, if $i \leq j$ are interior operations on $\cP$, then $\po{i}{\ab} \leq \po{j}{\ab}$ on $\cP'$.
\end{lemma}

\begin{proof}
If $p \leq q$, then for all $(C,E) \in \cP$, $p(C,E) \subseteq q(C,E)$.  If $(A,B) \in \cP'$ then there exists an injective module $E$ containing $B$ such that the inclusion morphism is in $\cM$ and $(A,E) \in \cP$. Then $\po{p}{\h}(A,B)=p(A,E) \cap B$ and $\po{q}{\h}(A,B)=q(A,E) \cap B$.  Since $p(A,E) \subseteq q(A,E)$, then \[\po{p}{\h}(A,B)=p(A,E) \cap B \subseteq q(A,E) \cap B=\po{q}{\h}(A,B),\] as desired.
\end{proof}

\begin{prop}
\label{prop:absoluteversionnakayama}
Let $(R,\m)$ be a Noetherian local ring and let $\cM$ be the category of Artinian $R$-modules with $R$-module maps. Let $\int$ be a functorial Nakayama interior operation on pairs of Artinian $R$-modules with absolute version $\po {\int} {\ab}$.  Then $\po {\int} {\ab}$ is Nakayama as well.
\end{prop}

\begin{proof}
Let $A \subseteq C \subseteq B$ be Artinian $R$-modules and suppose that $(A:_C \m)_{\po{\int}{\ab}}^B \subseteq A$. We will show that $A_{\po{\int}{\ab}}^B=C_{\po{\int}{\ab}}^B$. Let $E=E_R(R/\m)$ and $j:B \into E^t$ for some $t$ (this exists because $B$ is Artinian). 
Since $\int$ is intensive, $(A:_C \m)_{\po{\int}{\ab}}^B=(A:_C \m)_{\int}^{E^t}$, and  so $(A:_C \m)_{\int}^{E^t} \subseteq A$. Since $\int$ is Nakayama, this implies that $A_{\int}^{E^t}=C_{\int}^{E^t}$. By definition of ${\po{\int}{\ab}}$, we have
\[A_{\po{\int}{\ab}}^B=A_{\int}^{E^t}=C_{\int}^{E^t} \text{ and }
C_{\po{\int}{\ab}}^B=
C_{\int}^{E^t}.\]
Consequently, $C_{\po{\int}{\ab}}^B = A_{\po{\int}{\ab}}^B$. Hence ${\po{\int}{\ab}}$ is Nakayama.
\end{proof}

We end the section with a result on finitistic hereditary closure operations, which will be used in Section \ref{sec:integralclosure}.

\begin{prop}
\label{pr:finitisticextensionhereditary}
Let $R$ be a Noetherian ring.  Let $\cl$ be a hereditary closure operation defined on finitely generated $R$-modules.  Then $\cl_f$ is a  hereditary closure operation on all $R$-modules.  Hence $\cl_f = \po{\cl_f}{\h}$.
\end{prop}

\begin{proof}\
Since $\cl$ is hereditary, it is restrictable. Hence by Lemma \ref{lem:finitisticpairopprops} $\cl_f$ is a closure operation. It remains to prove that $\cl_f$ is hereditary.

Let $L \subseteq N \subseteq M$ be $R$-modules. 
Since for $U \subseteq N$ \fg, $U \subseteq M$ as well, $L_N^{\cl_f} \subseteq L_M^{\cl_f} \cap N$. For the reverse direction, let $x \in L_M^{\cl_f} \cap N$. Then there is some $V \subseteq M$ \fg\ such that $x \in (L \cap V)_V^{\cl}$. Set $U:=V \cap N$, which is a \fg\ submodule of $N$. Since $L \cap V \subseteq L \subseteq N$, $L \cap V=L \cap U$. Since $\cl$ is hereditary, 
\[(L \cap U)^{\cl}_U=(L \cap V)^{\cl}_U=(L \cap V)^{\cl}_V \cap U=(L \cap V)^{\cl}_V \cap N.\]
Since $x \in N$ by assumption, $x \in (L \cap U)^{\cl}_U$, so $x \in L_N^{\cl_f}$. Hence $L_M^{\cl_f} \cap N \subseteq L_N^{\cl_f}$.
\end{proof}

\section{Duality of hereditary and cohereditary versions of pair operations}
\label{sec:duality}

Next we compare the dual operations obtained by taking the cohereditary version of a pair operation (resp. the residual version of a closure) and taking its dual by the method of \cite{nmeRG-cidual}, versus using the non-residual dual from \cite{ERGV-nonres}. Note that if the closure is already residual, the two computations agree by the last sentence of Proposition \ref{pr:defineresidualversion}.

\begin{prop}\label{prop:dualsetsofpairs}
Let $(R,\m)$ be a complete local ring.  
\begin{enumerate}
    \item If $\cM$, $\cP$, and $\cP'$ are as defined in Notation \ref{not:coheredpairs}, with the additional assumptions that all modules are Noetherian, then $\cP^\vee$ and $(\cP')^\vee$ are sets of pairs as defined in Notation \ref{not:heredpairs}.
    \item If $\cM$, $\cP$, and $\cP'$ are as defined in Notation \ref{not:heredpairs}, with the additional assumptions that all modules are Artinian, then $\cP^\vee$ and $(\cP')^\vee$ are sets of pairs as defined in Notation \ref{not:coheredpairs}.
\end{enumerate}

\begin{rem}
The hypotheses involving Noetherian and Artinian conditions are
both so that if $E$ is injective, then $E^\vee$ is projective, and in order to dualize the diagrams used to prove that Condition (*) from Notation \ref{not:coheredpairs} implies Condition (**) from Notation \ref{not:heredpairs} and vice versa.
Note that the hypotheses are consistent with where we will apply this result--to sets of pairs $\cP'$ consisting of Noetherian and/or Artinian modules.
\end{rem}
\end{prop}

\begin{proof}[Proof of Proposition~\ref{prop:dualsetsofpairs}]
First we note that given a pair $(N,P) \in \cP$ as in Notation \ref{not:coheredpairs}, $P$ is projective. This implies $P^\vee$ is injective by \cite[Theorem 3.2.9]{EJ-book}. Hence pairs $((P/N)^\vee,P^\vee) \in \cP^\vee$ are valid pairs in the sense of Notation \ref{not:heredpairs}. Conversely, if $(C,E) \in \cP$ as in Notation \ref{not:heredpairs}, since $E$ is Artinian by hypothesis, $E^\vee$ is \fg\ and flat, and hence projective, by \cite[Theorem 3.2.16]{EJ-book}. Thus pairs $((E/C)^\vee,E^\vee) \in \cP^\vee$ are valid pairs in the sense of Notation \ref{not:coheredpairs}.

Further, since we have assumed all modules are Matlis dualizable, $\cP^{\vee \vee} = \cP$ and $(\cP')^{\vee \vee} = \cP'$ for $\cP$ and $\cP'$ in the sense of either Notation \ref{not:coheredpairs} or Notation \ref{not:heredpairs}. It will suffice to show that the conditions of Notation \ref{not:coheredpairs} hold for $\cP$ and $\cP'$ if and only if the conditions of Notation \ref{not:heredpairs} hold for $\cP^\vee$ and $(\cP')^\vee$.

Let $(L,M) \in \cP'$ and $\pi:P \twoheadrightarrow M$ in $\cM$ be such that $P$ is projective and $(\pi^{-1}(L),P) \in \cP$. Set $B=M^\vee$, $E=P^\vee$, $i=\pi^\vee$, and $A=(M/L)^\vee$. We claim that $(i(A),E) \in \cP^\vee$ if and only if $(\pi^{-1}(L),P) \in \cP$. By Lemma \ref{lem:matlisimagepreimage}, we have
\[i(A)=\pi^\vee((M/L)^\vee)=(P/\pi^{-1}(L))^\vee.\]
Hence
\[(i(A),E)=((P/\pi^{-1}(L))^\vee,P^\vee) \in \cP^\vee.\]
This proves our claim.

Next assume that $(L,M),(N,M) \in \cP'$, with $L \subseteq M$. Suppose there is some $\pi:P \to M$ with $(\pi^{-1}(L),P),(\pi^{-1}(N),P) \in \cP$. Set $B=M^\vee$, $E=P^\vee$, $i=\pi^\vee$, $A=(M/L)^\vee$, and $C=(M/N)^\vee$. Then 
\[i(A)=(P/\pi^{-1}(L))^\vee\]
and
\[i(C)=(P/\pi^{-1}(N))^\vee\]
as above. So
\[(i(A),E)=((P/\pi^{-1}(L))^\vee,P^\vee) \in \cP^\vee,\]
and the same for $(i(C),E)$. The reverse implication follows similarly.

To finish the proof, we show that Condition (*) of Notation \ref{not:coheredpairs} holds on $\cP$ and $\cP'$ if and only if Condition (**) of Notation \ref{not:heredpairs} holds on $\cP^\vee$ and $(\cP')^\vee$. Assume the following diagrams are Matlis duals:

\begin{center}
\begin{tabular}{c c}
$\xymatrix{
P  \ar@{->>}[d]^{\pi} \ar[r]^{\tilde{\phi}} & Q \ar@{->>}[d]^q  \\
M \ar@{->>}[r]^{\phi} & N  \\
}$ &
$\xymatrix{
F    & E \ar[l]^{\tilde{\psi}}   \\
C  \ar@{^{(}->}[u]^{j}  & B \ar@{^{(}->}[l]^{\psi} \ar@{^{(}->}[u]^i \\
}$ \\
\end{tabular}
\end{center}
where $F=P^\vee$, $E=Q^\vee$, $C=M^\vee$, and $N=B^\vee$. Further assume that $P$ and $Q$ are projective (so $F$ and $E$ are injective).
Set $L \subseteq M$ and assume $(L,M),(\phi(L),N) \in \cP'$ and $(\pi^{-1}(L),P),(q^{-1}(\phi(L)),Q) \in \cP$. Dually, if we set $A=(N/\phi(L))^\vee$, we get $(A,B),(\psi(A),C) \in (\cP')^\vee$ and $(i(A),E),(j(\psi(A)),F) \in \cP^\vee$ by repeated use of Lemma \ref{lem:matlisimagepreimage}.

We claim that $(\tilde{\phi}(\pi^{-1}(L)),Q) \in \cP$ if and only if $(\tilde{\psi}^{-1}(j(\psi(A))),E) \in \cP^\vee$. This will show that Condition (*) holds on $\cP$ and $\cP'$ if and only if Condition (**) holds on $\cP^\vee$ and $(\cP')^\vee$. By Lemma \ref{lem:matlisimagepreimage}, we have
\[j(\psi(A))=\pi^\vee((M/L)^\vee)=(P/\pi^{-1}(L))^\vee.\]
Further, \[\tilde{\psi}^{-1}((P/\pi^{-1}(L))^\vee)=(\tilde{\phi}^\vee)^{-1}((P/\pi^{-1}(L))^\vee).\]
We can see that the set of maps $f:Q \to E_R(k)$ such that $\pi^{-1}(L) \subseteq \ker(f \circ \tilde{\phi})$ are exactly the maps with $\tilde{\phi}(\pi^{-1}(L)) \subseteq \ker(f)$. So
\[(\tilde{\phi}^\vee)^{-1}((P/\pi^{-1}(L))^\vee)=(Q/\tilde{\phi}(\pi^{-1}(L)))^\vee.\] Hence
\[(\tilde{\psi}^{-1}(j(\psi(A))),E)=((Q/\tilde{\phi}(\pi^{-1}(L)))^\vee,Q^\vee),\]
which completes the proof.
\end{proof}

\begin{prop}\label{prop:completedualhandcversions}
Let $(R,\m)$ be a complete local ring.  
\begin{enumerate}
    \item \label{item:chdual} Let $\cM$, $\cP$, and $\cP'$ be defined as in Notation \ref{not:coheredpairs} with the additional assumption that all modules are \fg.
     If $p$ is a {cofunctorial} pair operation defined on $\cP$ and $\po{p}{\ch}$ is the cohereditary version of $p$ defined on $\cP'$, then $(\po{p}{\ch})^\dual=\po{(p^\dual)}{\h}$ is the hereditary version of $p^\dual$ defined on $(\cP')^\vee$. 
     In particular, if $\cl$ is a {cofunctorial (hence, functorial)} closure operation defined on $\cP$ and $\po{\cl}{\re}$ is the residual version of $\cl$ defined on $\cP'$, then $(\po{\cl}{\re})^\dual=\po{(\cl^\dual)}{\ab}$ is the absolute version of the interior $\cl^\dual$.
     
      \item \label{item:hdual} Let $\cM$, $\cP$, and $\cP'$ be defined as in Notation \ref{not:heredpairs} with the additional assumption that all modules are Artinian.
      If $p$ is a {functorial} pair operation defined on $\cP$ and $\po{p}{\h}$ is the hereditary version of $p$ defined on $\cP'$, then $(\po{p}{\h})^\dual=\po{(p^\dual)}{\ch}$ is the cohereditary version of $p^\dual$ defined on $(\cP')^\vee$. 
      
      In particular, if $\int$ is a {functorial} interior operation defined on $\cP$ and $\po{\int}{\ab}$ is the absolute version of $\int$ defined on $\cP'$, then $(\po{\int}{\ab})^\dual=\po{(\int^\dual)}{\re}$ is the residual version of the closure $\int^\dual$. 
\end{enumerate}
\end{prop}

\begin{proof}
\eqref{item:chdual}:  By Proposition \ref{prop:dualsetsofpairs}, if $\cP$ and $\cP'$ are  sets of finitely generated pairs as defined in Notation \ref{not:coheredpairs}, then $\cP^\vee$ and $(\cP')^\vee$ are sets of Artinian pairs as defined in Notation \ref{not:heredpairs}.  Let $(A,B) \in (\cP')^\vee$ and $i:B \into E$ be an injective homomorphism in $\cM$ so that $(i(A), E) \in \cP^\vee$.  Choose $L \subseteq M$ such that $B=M^\vee$ and $A=(M/L)^\vee$  so that $(L,M) \in \cP'$.  Set $P=E^\vee$ and $\pi=i^\vee$.  {Since $p$ is cofunctorial, $p^\smile$ is functorial by Proposition \ref{pr:functcofunct}.  Thus,} we know that $\po{p}{\ch}(L,M)=\pi(p(\pi^{-1}(L),P))$ and $\po{(p^\dual)}{\h}(A,B)=i^{-1}(p^\dual(i(A),E))$. By definition, 
\begin{align*}(\po{p}{\ch})^\dual(A,B)&=\left(\displaystyle\frac{B^\vee}{\po{p}{\ch}((B/A)^\vee,B^\vee)}\right)^\vee\\
&=\left(\displaystyle\frac{M}{\po{p}{\ch}(L,M)}\right)^\vee=\left(\displaystyle\frac{M}{\pi(p(\pi^{-1}(L),P))}\right)^\vee\\
&=i^{-1}\left(\left(\displaystyle\frac{P}{p(\pi^{-1}(L),P)}\right)^\vee\right)=i^{-1}(p^\dual((P/\pi^{-1}(L))^\vee,P^\vee))\\
&=i^{-1}(p^\dual(i(A),E))=\po{(p^\dual)}{\h}(A,B),\\
\end{align*}
where the fourth to last equality follows from Lemma~\ref{lem:matlisimagepreimage}.

\eqref{item:hdual}:  By Proposition \ref{prop:dualsetsofpairs}, if $\cP$ and $\cP'$ are sets of Artinian pairs as defined in Notation \ref{not:heredpairs}, then $\cP^\vee$ and $(\cP')^\vee$ are sets of \fg\ pairs as defined in Notation \ref{not:coheredpairs}.  Let $(L,M) \in (\cP')^\vee$ and $\pi:P \onto M$ be a surjective homomorphism in $\cM$ so that $(\pi^-1(L), P) \in \cP^\vee$.  Choose $A \subseteq B$ such that $M=B^\vee$ and $A=(M/L)^\vee$  so that $(A,B) \in \cP'$.  Set $E=P^\vee$ and $i=\pi^\vee$.  Since $p$ is functorial, $p^\smile$ is cofunctorial by Proposition \ref{pr:functcofunct}.  Thus we know that $\po{p}{\h}(A,B)=i^{-1}(p(i(A),E))$ and $\po{(p^\dual)}{\ch}(L,M)=\pi(p^\dual(\pi^{-1}(L),P))$. By definition, 
\begin{align*}(\po{p}{\h})^\dual(L,M)&=\left(\displaystyle\frac{M^\vee}{\po{p}{\h}((M/L)^\vee,M^\vee)}\right)^\vee\\
&=\left(\displaystyle\frac{B}{\po{p}{\h}(A,B)}\right)^\vee=\left(\displaystyle\frac{B}{i^{-1}(p(i(A),E))}\right)^\vee\\
&=\pi\left(\left(\displaystyle\frac{E}{p(i(A),E)}\right)^\vee\right)=\pi(p^\dual((P/\pi^{-1}(L))^\vee,E^\vee))\\
&=\pi(p^\dual(\pi^{-1}(L),P))=\po{(p^\dual)}{\h}(L,M),
\end{align*}
where the fourth to last equality follows from Lemma \ref{lem:matlisimagepreimage}.
\end{proof}

\begin{rem}
In the special case that $p$ is defined on $\cP'$, we can give a different proof of Proposition \ref{prop:completedualhandcversions}, which uses the following lemma:
\end{rem}

\begin{lemma}\label{lem:dualinequality}
Let $(R,\m)$ be a complete local ring. 
\begin{enumerate}
    \item \label{item:chineq} Let $\cM$, $\cP$, and $\cP'$ be defined as in Notation \ref{not:coheredpairs} with the additional assumption that all modules are finitely generated. If $p$ is a {cofunctorial} pair operation defined on $\cP'$ and $\po{p}{\ch}$ is the cohereditary version of $p$ defined on $\cP'$, then $p^\dual \leq (\po{p}{\ch})^\dual$.  
    
    In particular, if $\cl$ is a {cofunctorial (hence functorial)} closure operation defined on $\cP'$ and $\po{\cl}{\re}$ is the residual version of $\cl$ defined on $\cP'$, then $\cl^\dual \leq (\po{\cl}{\re})^\dual$.
    
    \item \label{item:hineq} Let $\cM$, $\cP$, and $\cP'$ be defined as in Notation \ref{not:heredpairs} with the additional assumption that all modules are Artinian. If $p$ is a {functorial} pair operation defined on $\cP'$ and $\po{p}{\h}$ is the hereditary version of $p$ defined on $\cP'$, then $p^\dual \geq (\po{p}{\h})^\dual$.   
    
    In particular, if $\int$ is a {functorial} interior operation defined on $\cP'$ and $\po{\int}{\ab}$ is the absolute version of $\int$ defined on $\cP'$, then $\int^\dual \geq (\po{\int}{\ab})^\dual$.
\end{enumerate}
\end{lemma}

\begin{proof}
Both \eqref{item:chineq} and \eqref{item:hineq} are direct consequences of Proposition \ref{pr:comparepairops}, Proposition \ref{pr:sfcohereditarypairineq} and Proposition \ref{pr:sfhereditarypairineq}.
\end{proof}

\begin{prop}\label{pr:dualequality}
Let $(R,\m)$ be a complete local ring. 
\begin{enumerate}
    \item \label{item:cheq} Let $\cM$, $\cP$, and $\cP'$ be defined as in Notation \ref{not:coheredpairs} with the additional assumption that all modules are \fg.
    If $p$ is a {cofunctorial} pair operation defined on $\cP'$ and $\po{p}{\ch}$ is the cohereditary version of $p$ defined on $\cP'$, then $\po{(p^\dual)}{\h} = (\po{p}{\ch})^\dual$ for all pairs in $(\cP')^\vee$.  
    
    In particular, if $\cl$ is a {cofunctorial (hence, functorial)} closure operation defined on $\cP'$ and $\po{\cl}{\re}$ is the residual version of $\cl$ defined on $\cP'$, then $\po{(\cl^\dual)}{\ab} = (\po{\cl}{\re})^\dual$ for all pairs in $(\cP')^\vee$.
    
    \item \label{item:heq} Let $\cM$, $\cP$, and $\cP'$ be defined as in Notation \ref{not:heredpairs} with the additional assumption that all modules are Artinian.
    If $p$ is a {functorial} pair operation defined on $\cP'$ and $\po{p}{\h}$ is the hereditary version of $p$ defined on $\cP'$, then $\po{(p^\dual)}{\ch} = (\po{p}{\h})^\dual$ for all pairs in $(\cP')^\vee$. 
    
    In particular, if $\int$ is a {functorial} interior operation defined on $\cP'$ and $\po{\int}{\ab}$ is the absolute version of $\int$ defined on $\cP'$, then $\po{(\int^\dual)}{\re} = (\po{\int}{\ab})^\dual$ for all pairs in $(\cP')^\vee$.
\end{enumerate}
\end{prop}

\begin{proof}
\eqref{item:cheq}:  By Proposition \ref{pr:definecohereditaryversion}, $\po{p}{\ch}$ is a  cofunctorial, cohereditary pair operation on $\cP'$.  Now by Propositions \ref{pr:functcofunct} and \ref{pr:dualityprops}, $(\po{p}{\ch})^\dual$ is a functorial and hereditary pair operation on $(\cP')^\vee$.  Hence, $\po{((\po{p}{\ch})^\dual)}{\h}=(\po{p}{\ch})^\dual$ according to Proposition \ref{pr:definehereditaryversion}. By Proposition \ref{pr:sfcohereditarypairineq}, we have $\po{p}{\ch} \leq p$.  Applying Lemma \ref{lem:dualinequality}\eqref{item:chineq}, $(\po{p}{\ch})^\dual \geq p^\dual$ and by Lemma \ref{lem:comph}, 
\begin{equation}\label{eq:heredineq1} (\po{p}{\ch})^\dual=\po{((\po{p}{\ch})^\dual)}{\h} \geq \po{(p^\dual)}{\h}.\tag{$\ast$}\end{equation}

By Proposition \ref{pr:sfhereditarypairineq}, $p^\dual \leq \po{(p^\dual)}{\h}$ is a pair operation on $(\cP')^\vee$ and $(\cP')^\vee$ satisfies Notation \ref{not:heredpairs} by Proposition \ref{prop:dualsetsofpairs}.  Since for all pairs $(L,M) \in \cP'$, $M$ is finitely generated, their Matlis duals are Artinian implying that $(A,B) \in (\cP')^\vee$.  Applying Lemma \ref{lem:dualinequality}\eqref{item:heq}, we obtain $p=(p^\dual)^\dual \geq (\po{(p^\dual)}{\h})^\dual$ and $(\po{(p^\dual)}{\h})^\dual$ is {cofunctorial and} cohereditary on $\cP'$ by Propositions  \ref{pr:functcofunct} and \ref{pr:dualityprops}. It follows from Lemma \ref{lem:compch} that  $\po{p}{\ch} \geq \po{((\po{(p^\dual)}{\h})^\dual)}{\ch}=(\po{(p^\dual)}{\h})^\dual$.  Now by 
Proposition \ref{pr:comparepairops}, 
\begin{equation}\label{eq:heredineq2}
    (\po{p}{\ch})^\dual \leq ((\po{(p^\dual)}{\h})^\dual)^\dual=\po{(p^\dual)}{\h}. \tag{$\ast\ast$}
\end{equation}
Combining Equations \eqref{eq:heredineq1} and \eqref{eq:heredineq2}, we obtain the result.

The proof of \eqref{item:heq} follows the same reasoning.
\end{proof}

\section{Hereditary and cohereditary versions of {$J$}-basically full closure and {$J$}-basically empty interior}\label{sec:JbfJbe}

 In this section we will consider the closure operation $\Jcolsym J$ and the interior operation $\Jintrelsym J$ through the lenses of hereditary and cohereditary pair operations. 

In \cite{ERGV-nonres}, we dualized $\Jcolsym J$ by developing a duality for closures that were not necessarily residual. Here we present an alternate method of dualizing $\Jcolsym J$, by taking its residual version $\po{\Jcolsym J}{\re}$ and dualizing that instead. Note that since $\po{\Jcolsym J}{\re}$ is residual, its dual is the same whether computed by the method of \cite{nmeRG-cidual} or the method of \cite{ERGV-nonres} (see \cite[Proposition 3.2]{ERGV-nonres}).

\begin{defn} \cite[Definition 4.2]{ERGV-nonres} Let $R$ be a commutative ring.
Let $J$ be an ideal of $R$.  Then for any submodule inclusion $L \subseteq M$, we define the $J$-basically full closure of $L$ in $M$ as follows:
\[
\Jcol JLM := (JL :_M J).
\]
\end{defn}

\begin{defn} \cite[Definition 4.9]{ERGV-nonres} Let $R$ be a commutative ring.
Let $J$ be an ideal of $R$.  Then for any submodule inclusion $L \subseteq M$, we define the $J$-basically-empty interior of $L$ in $M$ as follows:
\[
\Jintrel JML := J(L :_M J).
\]
\end{defn}

We proved in  \cite[Theorem 4.12]{ERGV-nonres} that $\Jcolsym J$ and $\Jintrelsym J$ are dual pair operations on pairs of finitely generated modules or pairs of Artinian modules over a complete local ring.

We first consider  $\Jcolsym J$.  We showed that  $\Jcolsym J$ is functorial  in \cite[Proposition 4.7]{ERGV-nonres}.  Although,  $\Jcolsym J$ is not residual, we can show that  $\Jcolsym J$ is hereditary:

\begin{prop}
\label{jcolhereditary}
Let $R$ be a commutative ring. For any ideal $J$ of $R$, the operation $\Jcolsym J$ is a hereditary closure operation.
\end{prop}

\begin{proof}
Let $L \subseteq N \subseteq M$ be $R$-modules.  Clearly $\Jcol JLN \subseteq \Jcol JLM \cap N$ by functoriality.  Now suppose that $z \in \Jcol JLM \cap N$. Since $\Jcol JLM=(JL:_M J)$, then $Jz \subseteq JL$ and $z \in N$.  Hence $z \in (J L:_NJ)=\Jcol JLN$.
\end{proof}
By Proposition \ref{pr:definehereditaryclosure}, $\po{\Jcolsym J}{\h}=\Jcolsym J$.
Since $\Jcolsym J$ is functorial, we define the \emph{residual} version of $\Jcolsym J$, which is a residual closure by Proposition \ref{pr:defineresidualversion}:

\begin{defn}\label{def:rJcol}  Let $R$ be a commutative ring.
 For any module $M$, let $\pi: P \onto M$ be a surjection with $P$ projective.  Then for a submodule inclusion $L \subseteq M$, we set $L^{\po{\Jcolsym J}{\re}}_M = \pi\left({\pi^{-1}(L)}^{\Jcolsym J}_P\right)$.
\end{defn}

By Corollary \ref{pr:rclincl} we have the following comparison between $\po {\Jcolsym J}{\re}$ and $\Jcolsym J$:
\begin{cor}
Let $R$ be a commutative ring and $J$ an ideal of $R$. For any $R$-submodule inclusion $L \subseteq M$, $L^{\po{\Jcolsym J}{\re}}_M \subseteq \Jcol J L M$.
\end{cor}

Similar to Example \ref{ex:lirverstrict}, the inequality 
 $\po{\Jcolsym J}{\re} \le \Jcolsym J$ can be strict, as seen in the example below:

\begin{example}
Let $R=k[\![x,y]\!]$ where $k$ is any field, let the module $M$ be the ideal $(x^3, x^2y^2, y^3)$, and let $L$ be the sub-ideal $(x^3, y^3)$.  Then we have a surjection $\pi: R \ra M/L$ given by $1 \mapsto x^2y^2 + L$.  The kernel is 
\[L:_R (x^2y^2)=(x^3, y^3):_R (x^2y^2) = (x,y) = \m,\]
so $L^{\po{\Jcolsym{\m}}{\re}}_M / L = \pi(\Jcol{\m}{\m}R) = \pi(\m) = 0$, whence $L^{\po{\Jcolsym{\m}}{\re}}_M = L$.  However $\Jcol{\m}LM=M$.
\end{example}

The next result tells us that we can compute $\po{\Jcolsym{J}}{\re}$-cores.

\begin{cor}
Let $R$ be a commutative ring and $J$ an ideal of $R$. The closure operation $\po{\Jcolsym{J}}{\re}$ is a Nakayama closure.
\end{cor}

\begin{proof}
This follows from 
\cite[Proposition 4.7]{ERGV-nonres} and Proposition \ref{prop:residualversionnakayama}.
\end{proof}

The next example exhibits an instance where the residual basically full core of a submodule is not the same as the basically full core of the submodule  (See \cite[Sections 4 and 8]{ERGV-nonres} for relevant definitions.):
\begin{example}
Let $R=k[[t^2,t^3]]$, $\m=(t^2,t^3)$ and $L_a=(t^2+at^3)$ for any $a \in k$. First note that $(\m L_a:_\m \m)=(\m^2:_\m \m)=\m$ for any $a \in k$ making $\Jcol{\m}{(L_a)}{\m}=\m$.  Hence, $L_a$ is an $\Jcolsym \m$-reduction of $\m$ for all $a \in k$.  Note that $(\m (t^3,t^4):_\m \m) =((t^5,t^6):_\m \m)=(t^3,t^4)$ so no other ideal in the lattice of ideals in $\m$ is an $\Jcolsym \m$-reduction of $\m$.  Thus 
\[\Jcolsym \m\core_\m \m= \bigcap_{a\in k} L_a=\m^2.
\]

To see that $\po{\Jcolsym \m}{\re}\core_\m \m=\m \neq \m^2$, we will show that none of the $L_a$ are $\po{\Jcolsym \m}{\re}$-reductions of $\m$ in $\m$.  Consider $\pi_a:R \onto \m/L_a$, given by $\pi_a(1)=t^3$. Since ${\rm ker}(\pi_a)=\m$ and $\m_R^{\Jcolsym \m}=\m$, $0_{\m/L_a}^{\po{\Jcolsym \m}{\re}}=0$. Since $\po{\Jcolsym \m}{\re} \le \po{\Jcolsym \m}{}$, $(t^3,t^4)_{\m}^{\po{\Jcolsym \m}{\re}} \ne \m$.

Since $\po{\Jcolsym \m}{\re}$ is residual, $(L_a)_{\m}^{\po{\Jcolsym \m}{\re}}=L_a \ne \m$.
\end{example}

Next, we show that $\po{\Jcolsym J}{\re}$ is finitistic, provided that $J$ is \fg.

\begin{prop}\label{pr:rJfg}
Let $R$ be a commutative ring and $J$ a  finitely generated ideal of $R$. Let $\alpha$ be the preradical on $R$-modules given by $\alpha(M) = 0^{\po{\Jcolsym J}{\re}}_M$.  Then $\alpha = \alpha_f$, i.e., $\alpha$ is finitistic.
It follows that ${\po{\Jcolsym J}{\re}}$ is a finitistic closure operation.
\end{prop}


\begin{proof}
 Let $y \in \alpha(M)$, so that $y \in 0_M^{{\Jcolsym J}_r}$. Then there exist a free module $F$ and a submodule $L \subseteq F$ such that $\pi: F \twoheadrightarrow M$ is a surjection with kernel $L$ and there is some $z \in L_F^{\Jcolsym J}$ with $\pi(z)=y$.
 Write $F = \bigoplus_{\alpha \in \Lambda} R e_\alpha$, where the $e_\alpha$s form a free basis for $F$ over $R$.  
 Then as $z\in F$, $z$ is a linear combination of only finitely many $e_\alpha$s.  
Let $j_1, \ldots, j_m$ be a system of generators for the ideal $J$.  Thus, $Jz$ involves only these $e_\alpha$s.  
Similarly, as each $j_i z \in JL$, there is some finitely generated submodule $L'$ of $L$ such that each $j_i z \in J L'$, so that $Jz \subseteq JL'$.  As $L'$ is finitely generated, its generators use only finitely many $e_\alpha$'s.  Thus, there is some finite collection $\{\alpha_1, \ldots, \alpha_n\} \subseteq \Lambda$ such that, writing $G := \bigoplus_{i=1}^n R e_{\alpha_i}$, we have $L' \subseteq G$ and $z\in G$.  Thus, $Jz \subseteq JL'$, and since also $z\in G$ we have $z \in JL':_GJ = \Jcol J {(L')} G$.  Thus since $L' \subseteq L = \ker \pi$, we have $y = \pi(z) \in 0^{\po{\Jcolsym J}{\re}}_{\pi(G)} = \alpha(\pi(G))$, so that since $\pi(G)$ is a \fg\ submodule of $M$, $\alpha = \alpha_f$.

For the final statement, let $L \subseteq M$ and let $x \in L^{\po{\Jcolsym J}{\re}}_M$.  Then $\bar x \in 0^{\po{\Jcolsym J}{\re}}_{M/L} = \alpha(M/L)$ (by residuality), so that since $\alpha$ is finitistic, there is some \fg\ $T \subseteq M/L$ with $\bar x \in \alpha(T)$.  But by the correspondence theorem, $T = N/L$, for some $N$ with $L \subseteq N \subseteq M$.  That is, $\bar x \in 0^{\po{\Jcolsym J}{\re}}_{N/L}$, so that since ${\po{\Jcolsym J}{\re}}$ is residual, $x \in L^{\po{\Jcolsym J}{\re}}_N$.  Then by \cite[Lemma 3.2]{ERGV-chdual} (see also Lemma~\ref{lem:twoversionsoffinitistic}), we are done.


\end{proof}

The assumption that $J$ is finitely generated is necessary in Proposition~\ref{pr:rJfg}, as shown in the following example.

\begin{example}
\label{ex:Jbfnotfin}
Let $(R,\m)$ be a valuation ring whose value group is $\Q$.  Then $\m^2 = \m$ by \cite[Exercise 6.29]{integralclosure}, since $\m$ is not principal as whenever $x$ is a nonunit with value $\alpha$, there is an element of value $\alpha/2$ in $\m \setminus xR$. Thus, $\Jcol \m \m R = (\m^2 :_R \m) = (\m :_R \m) = R$.  On the other hand, let $I$ be a nonzero finitely generated $R$-submodule of $\m$ -- i.e. a finitely generated proper ideal of $R$.  Then there is some $x\in \m$ with $I=xR$.  Then $\Jcol \m I R = \m I : \m = \m x : \m  \subseteq \m$, since $x \notin \m x$ (since every element of $\m x$ has value strictly greater than the value of $x$), but $x\in \m$.  Thus, $\m^{\m {\rm bf}_f}_R =\m \neq R = \Jcol \m \m R$.  Thus, in the notation of the Proposition, $\alpha(R/\m) = R/\m$ while $\alpha_f(R/\m) = 0$.
\end{example}

The following example illustrates that even when liftable integral closure (see Definition \ref{def:li}) and residual $\m$-basically full closure agree, one can still have a proper containment with respect to their non-residual versions. 

\begin{example}
\label{ex:residualversionsdisagree}
Let $R=k[\![x,y]\!]$ where $k$ is any field, let $M$ be the ideal $(x^3,x^2y, xy^2, y^3)$, and let $L$ be the subideal $(x^3, y^3)$.  Then we have a surjection $\pi: R^2 \ra M/L$ given by $(1,0) \mapsto x^2y + L$ and $(0,1) \mapsto xy^2 + L$.  The kernel is $K=\langle (x,0), (y,-x),(0,y)\rangle$.  
As \[\m K=\langle(x^2,0), (xy,-x^2), (0,xy),(xy,0), (y^2,-xy),(0,y^2) \rangle=\m^2\oplus \m^2,\] we see that $\Jcol{\m}{K}{R^2}=\m K:_{R^2} \m=\m\oplus \m$ making \[\pi(\Jcol{\m}{K}{R^2})= L^{\po{\Jcolsym{\m}}{\re}}_M  / L= \displaystyle\frac{(x^2y^2)+L}{L}.\]   Observe also that  $\Jcol{\m}LM=(x^3,x^2y^2,y^3)=L^{\po{\Jcolsym{\m}}{\re}}_M$.   Whereas, we can use the symmetric algebra $\Sym(R^2)$ of $R^2$ to determine that $K$ is a reduction of $\m \oplus \m$ which is integrally closed in $R^2$. Note that \[\Sym(K) \cong R[xt_1, yt_1-xt_2,yt_2] \subseteq R[t_1,t_2] \cong \Sym(R^2).\]  We need only show that $yt_1, xt_2 \in \sqrt{(xt_1,yt_1-xt_2,yt_2)R[\m t_1, \m t_2]} $ by \cite[Theorem 16.2.3]{integralclosure}.  By the following computation, we see that 
\[(yt_1)^2=yt_1(yt_1-xt_2)-xt_1(yt_2) \quad \text{and} \quad (xt_2)^2=-xt_2(yt_1-xt_2)-xt_1(yt_2)\] implying that $K^{-}_{R^2}=\m \oplus \m$.  Since $R^2$ is free, both the Rees integral closure and Eisenbud-Huneke-Ulrich integral closures are equal.  This is why we are labeling the integral closure of $K$ in $R^2$ without specification of Rees or EHU.  Thus
 \[\lic LM / L = \pi(K^-_{R^2}) = \pi(\m \oplus \m)= \displaystyle\frac{(x^2y^2)+L}{L}.\] Since $M \subseteq R$, $L^-_M=L^-_R=M$.  
 Comparing our computations above, we see that $\lic LM=L_M^{\po{\Jcolsym{\m}}{\re}}$, but $\Jcol{\m} LM \subsetneq L^-_M$.
\end{example}
 
The following result holds by duality when the ring is complete and local, but we prove it in more generality here.

\begin{prop}
\label{pr:jbecohereditary}
For any ideal $J$, the operation $\Jintrelsym J$ is a cohereditary interior operation.
\end{prop}

\begin{proof}
Let $L \subseteq N \subseteq M$ be $R$-modules.  Suppose $\pi:M \onto M/L$ is the natural epimorphism.  Since $\Jintrelsym J$ is functorial by \cite[Proposition 4.10]{ERGV-nonres},
\[\pi(\Jintrel JMN)=\displaystyle\frac{\Jintrel JMN+L}{L} \subseteq  \Jintrel J{M/L}{(N/L)}.\]    As  $\Jintrel J{M/L}{(N/L)}=J(N/L:_{M/L} J)$, the elements of  $\Jintrel J{M/L}{(N/L)}$ are of the form $\sum\limits_{i=1}^n j_i(z_i+L)$ where $j_i \in J$ and $z_i+L \in (N/L:_{M/L} J)$; thus, $J(z_i+L) \subseteq N/L$ which implies $Jz_i \subseteq N$.  Hence $z_i \in (N:_M J)$ and  $\sum\limits_{i=1}^n j_i(z_i+L) \subseteq \displaystyle\frac{\Jintrel JMN+L}{L}$ giving us the equality $\displaystyle\frac{\Jintrel JMN+L}{L} = \Jintrel J{M/L}{(N/L)}$ which implies that $\Jintrelsym J$ is cohereditary.
\end{proof}

Now by Proposition \ref{pr:definecohereditaryinterior}, $\po{\Jintrelsym J}{\ch}=\Jintrelsym J$.  Though $\Jintrelsym J$ is not absolute, since $\Jintrelsym{J}$ is functorial, we can define the absolute version via Proposition \ref{pr:absoluteinterior}.

\begin{defn}
Let $L \subseteq M$ and $E$ be any injective module containing $M$, the absolute $\Jintrelsym J$ interior is defined by
\[L^M_{\po{\Jintrelsym J}{\ab}}:=\Jintrel J EL \cap M
\]
\end{defn}

By Corollary \ref{pr:aintineq}, we obtain the following:

\begin{cor}
For any $R$-module inclusion $L \subseteq M$, $L^M_{\Jintrelsym J} \subseteq L^M_{\po{\Jintrelsym J}{\ab}}$.
\end{cor}

To see that the inequality $\Jintrelsym J \leq \po{\Jintrelsym J}{\ab}$ is strict consider the following example:

\begin{example}\label{ex:absineqex}
Let $R=k[[t^2,t^5]]$.  The quotient field $Q=k((t^2,t^5))=k((t))$ is an injective module containing $R$ 
 by \cite[Proposition 3.34(i)]{Rot-hom}. Consider $J=(t^5,t^6)$.  Note that $t^3 \in Q \setminus R$.  Hence, $(J:_Q\m)=(t^3,t^4)$ but $(J:_R\m)=(t^4,t^5)$.  Thus $\Jintrel \m Q J=\m (J:_Q\m)=(t^5,t^6)$ making $J^R_{\po{\Jintrelsym \m}{\ab}}=(t^5,t^6) \cap R=(t^5,t^6)$ which properly contains $(t^6,t^7)=\m(J:_R\m)=\Jintrel \m R J$. 
\end{example}

The following result tells us that we can compute $\po{\Jintrelsym J}{\ab}$-hulls.

\begin{cor}
The interior operation $\po{\Jintrelsym J}{\ab}$ is a Nakayama interior.
\end{cor}

\begin{proof}
This follows from Proposition \cite[Proposition 4.10]{ERGV-nonres} and Proposition \ref{prop:absoluteversionnakayama}.
\end{proof}

We will also demonstrate that the $\m$-basically empty hull can differ from the absolute version of $\m$-basically empty hull.  First we include a result on numerical semigroup rings.  Before stating this result, we recall some nice numerical invariants of numerical semigroup rings.  

Recall that a ring $R=k[[t^S]]$ where $S=\langle s_1,s_2, \ldots, s_n\rangle \subseteq \mathbb{N}$ satisfying $s_1 <s_2< \ldots <s_n$ and gcd$(s_1, s_2 \ldots, s_n)=1$ is called a numerical semigroup ring.  
The multiplicity of the semigroup $S$ is $e=s_1$ (this is also equal to the multiplicity of $R$).  The conductor $c$ of $S$ is the smallest positive integer $c$ such that $n \in S$ for all $n \geq c$.  The conductor ideal of $R$ will in fact be $\langle t^n \mid n \geq c\rangle$.

In the following result, we describe a situation where the $\m$-basically empty interior agrees with its absolute version.

\begin{prop}\label{pr:abs=relnsgr}
Let $R=k[[t^S]]$ where $S=\langle s_1,s_2,\ldots, s_n\rangle \subseteq \mathbb{N}$ satisfying 
\[s_1 <s_2< \ldots <s_n \text{ and {\rm gcd}}(s_1, s_2 \ldots, s_n)=1.\]  Set $Q=k((t))$, the field of fractions of $R$.  Let $I=(f_1, \ldots, f_r)$ be an ideal of $R$ and suppose that the least power of $t$ appearing in the expression of $f_i$ as a power series is $m_i$.  Let $m=\text{\rm min}\{m_1, \ldots, m_r\}$. If $m \geq e+c$, then $\m (I:_R \m)=\m (I:_Q \m)$ if $m \geq e+c$.  In other words,
$\Jintrel{\m}RI=I_{\po{\Jintrelsym{\m}}{\ab}}^R$ if $m \geq e+c$.
\end{prop}

\begin{proof}
Clearly, $\m(I:_R \m) \subseteq  \m(I:_Q \m)$. Let $f \in (I:_Q \m)$.  Then $t^ef \in I$.  So if $f=\sum\limits_{i=n}^{\infty} a_i t^i$, then $t^e f=\sum\limits_{i=n}^{\infty} a_i t^{i+e}$. Since this is in $I$, $n+e \geq m$.  So $n \geq m- e \geq e+c-e=c$.  Hence $f \in \mathfrak{c} \subseteq R$ and $\m(I:_Q \m) \subseteq  \m(I:_R \m)$ finishing the proof.
\end{proof}

Since $\po{\Jcolsym{J}}{\re}$ is residual, we can dualize it by the method of \cite{nmeRG-cidual}, or equivalently by the method of \cite[Proposition 3.2]{ERGV-nonres}. Note that by Example \ref{ex:absineqex}, its dual will not always agree with $\Jintrelsym J$.

\begin{cor}
Let $R$ be a complete local ring. The pair operation $\po{\Jcolsym{J}}{\re}^\dual$ is an absolute interior operation and $\left(\po{\Jcolsym{J}}{\re}\right)^\dual=\po{\Jintrelsym J}{\ab}$ when $\po{\Jcolsym{J}}{\re}$ is considered as an operation on \fg\ $R$-modules.
\end{cor}

\begin{proof}
This follows from Proposition \ref{pr:dualequality}.
\end{proof}

We use the following identities in the next example:

\begin{lemma}
    Let $R$ be an Artinian Gorenstein local ring. Then for any ideals $I$ and $J$ of $R$,
    \[J(I:_R J)=\ann(J(\ann I):_R J) \text{ and } (JI:_R J)=\ann(J((\ann I):_R J)).\]
\end{lemma}

\begin{proof}
By \cite[Exercise 3.2.15]{BH}, for any ideal $I$, $\ann(\ann I))=I$. 

Next, we show that for any two ideals $I$ and $J$, we have $(I:_RJ) = \ann(J \ann I)$.  To see this, first let $x\in (I:J)$, $j\in J$, and $b\in \ann I$.  Then $xj \in I$, so $xjb = 0$.  Thus, $x\in \ann(J \ann I)$.  Conversely, let $z \in \ann(J \ann I)$ and $j\in J$.  Then $zj \in \ann(\ann I) = I$, so $z \in (I :_R J)$.

It then follows that $\ann(I :_R J) = \ann(\ann(J \ann I)) = J \ann I$.

Next we show that for ideals $I$ and $J$, $\ann(IJ)=(\ann I):_R J$. Let $x \in (\ann I):_R J$ and $y \in IJ$, say $y=\sum r_is_i$ with the $r_i \in I$ and $s_i \in J$. Then $xy=\sum r_ixs_i$. Since each $xs_i \in \ann I$ and each $r_i \in I$, $xy=0$. For the other inclusion, let $x \in \ann IJ$ and $y \in J$. for any $z \in I$, $xyz \in x(IJ)=0$.

To prove the first identity in the statement of the lemma, we have
\begin{align*}
    \ann(J(\ann I):_R J) &= J \ann(J(\ann I)) \\
    &=J((\ann \ann I):_R J) = J(I:_R J).
\end{align*}
For the second identity, we have
\begin{align*}
    \ann(J((\ann I):_R J)) &=(\ann((\ann I):_R J)):_R J \\
    &=(J (\ann \ann I)):_R J = (JI):_R J.
\end{align*}
\end{proof}

\begin{example}
As a special case, suppose our local ring $R$ is Artinian Gorenstein.  Then $E=E_R(k) = R$.  Hence, the duality given above between $\po{\Jcolsym{J}}{\re}$ and $\po{\Jintrelsym J}{\ab}$  restricts to a duality between a \emph{closure operation on ideals of $R$} and an \emph{interior operation on ideals of $R$}.

In particular, by passing through our usual smile duality, one obtains for any ideal $I$ that \begin{align*}
I_{\po{\Jintrelsym{J}}{\ab}}^R&=\Jintrel J R I = J(I:_RJ) = \ann(J(\ann I) :_RJ) \\
&= \ann(\Jcol J {(\ann I)}R)=\ann((\ann I)^{\po{\Jcolsym{J}}{\re}}_R),
\end{align*}
and 
\begin{align*}
I_R^{\po{\Jcolsym{J}}{\re}} &=\Jcol J I R = (JI :_R J) = \ann(J ((\ann I) :_RJ))\\
&=\ann\left( \Jintrel{J}{R}{(\ann I)}\right)= \ann\left((\ann I)^{R}_{ \po{\Jintrelsym{J}}{\ab}}\right).
\end{align*}
\end{example}

\section{Pair operations derived from a pre-enveloping class} \label{sec:preenveloping}

In this section we develop the theory of pair operations derived from a pre-enveloping class. This is a generalization of the way that EHU-integral closure extends integral closure to modules. While it is in some ways similar to the hereditary version of a closure developed in Section \ref{sec:hereditaryversions}, here we do not assume the embeddings are into injective modules. As a result, we need some additional hypotheses to prove comparable results.

\begin{defn}
\cite[Definition 6.1.1]{EJ-book}
Let $\cM$ be a category of $R$-modules.  Let $\cC \subseteq \cD$ be two classes of modules in $\cM$.  We say that $\cC$ is a \emph{pre-enveloping class for $\cM$ in $\cD$} if for any $M \in \cM$, there is some $C \in \cC$ and some morphism $\alpha: M \ra C$ in $\cM$, such that for any morphism $g: M \ra D$ in $\cM$ with $D \in \cD$, there is some morphism $\tilde g: C \ra D$ in $\cM$ with $g = \tilde g \circ \alpha$.  In this case the map $\alpha$ (or by abuse of notation, $C$ itself) is called a \emph{$\cC$-preenvelope of $M$ in $\cD$}.

If $\cC = \cD$, we 
drop the ``in $\cD$'' portion of the terminology.
\end{defn}

Here are some examples: 

\begin{example} If $\cM$ is the class of $R$-modules, and $\cC$ is the class of injective $R$-modules, then any injective map $M \ra C$ with $C \in \cC$ is a $\cC$-preenvelope.  This follows from the definition of injective module.
\end{example}

\begin{example}
By \cite[Proposition 6.5.1]{EJ-book}, a ring $R$ is coherent if and only if the class of flat modules is pre-enveloping.
\end{example}

Next, we recall the biduality natural transformation:
\begin{defn}
Let $R$ be a commutative ring. Let $1$ be the identity functor on $R$-modules, and $\beta$ the biduality functor, given on objects by $\beta(M) = M^{**}= \Hom_R(\Hom_R(M,R),R)$, and on morphisms by the $R$-duality functor composed with itself.  That is, if $\alpha: M \ra N$ is a morphism, $\alpha^*: N^* \ra M^*$ is defined by $\alpha^*(g) := g \circ \alpha$, and $\beta(\alpha) := \alpha^{**} := (\alpha^*)^*$.

Then define 
$h: 1 \ra \beta$ by the rule that for any $R$-module $M$, $h_M: M \ra M^{**}$ is given by $h_M(x)(g) := g(x)$.
\end{defn}

The fact that the above is a natural transformation is well-known, but we provide a proof:

\begin{lemma}\label{lem:hnat}
The rule $h$ as defined above is a natural transformation.
\end{lemma}

\begin{proof}
Let $f: M \ra N$ be a map of $R$-modules.  Let $x\in M$ and $g \in N^*$.  Then $f^{**}(h_M(x))(g) = h_M(x)(g \circ f) = (g \circ f)(x) = g(f(x)) = h_N(f(x))(g)$.  Since both $g$ and $x$ were arbitrarily chosen, it follows that $h_N \circ f = f^{**} \circ h_M$.  Since $f$ was arbitrary, it follows that $h$ is a natural transformation.
\end{proof}

The following lemma is also well-known, but for lack of a reference we provide a statement and a proof.

\begin{lemma}\label{lem:hPiso}
Let $P$ be a finitely generated, projective $R$-module.  Then $h_P$ is an isomorphism.
\end{lemma}

\begin{proof}
First we consider the case where $P$ is free. Say $P = R^n$.  Choose a free basis $e_1, \ldots, e_n$ for $R^n$.  Then $P^{**}$ is a free module on the basis $h_P(e_1), \ldots, h_P(e_n)$.  One can therefore make an inverse map $i: P^{**} \ra P$ by sending each $h_P(e_j)$ to $e_j$.

Now we consider the general case.  Then there is a free $R$-module $F = R^n$ and there are maps $\pi: F \ra P$ and $j: P \ra F$ with $\pi \circ j = 1_P$.  Accordingly, $j$ is injective.  Moreover, by functoriality of the biduality functor $\beta$, we have $1_{P^{**}} = \beta(1_P) = \beta (\pi \circ j) = \beta(\pi) \circ \beta(j) = \pi^{**} \circ j^{**}$.  The following diagram is then commutative by Lemma~\ref{lem:hnat}. \[
\xymatrix{
P \ar[r]_j \ar[d]_{h_P} \ar@/^1.0pc/[rr]^{1_P} & F \ar[r]_\pi \ar[d]_{\cong}^{h_F} &P \ar[d]^{h_P}\\
P^{**} \ar[r]^{j^{**}} \ar@/_1.0pc/[rr]_{1_{P^{**}}} &F^{**}  \ar[r]^{\pi^{**}} &P^{**}
}
\]

Let $x\in P$ with $h_P(x)=0$.  Then $0 = (j^{**} \circ h_P)(x) = (h_F\circ j)(x) = h_F(j(x))$. But since both $h_F$ and $j$ are injective, it follows that $x=0$.  Hence, $h_P$ is injective.

Finally, let $c\in P^{**}$.  Since $h_F$ is surjective, there is some $d\in F$ with $h_F(d) = j^{**}(c)$.  Then $h_P(\pi(d)) = \pi^{**}(h_F(d)) = \pi^{**}(j^{**}(c)) = c$.  Thus, $h_P$ is surjective.
\end{proof}

\begin{prop}
\label{pr:fgpreenvproj}
If $\cM$ is the class of \emph{finitely generated} $R$-modules, then the class of finitely generated projective $R$-modules is pre-enveloping. 
In fact, the class of finitely generated free modules is pre-enveloping in the class of finitely generated projectives.  
\end{prop}
\begin{proof}
Let $M$ be a finitely generated $R$-module, and choose an $R$-module surjection $\pi: F \onto M^*$, where $F$ is finitely generated and free.  Then we claim that $\alpha := \pi^* \circ h_M: M \ra F^*$ is both a \fg\ free preenvelope and a \fg\ projective preenvelope.  To see this, let $P$ be a finitely generated projective module and $\beta: M \ra P$ an $R$-linear map.  Then since $P^*$ is projective and $\pi$ is surjective,, there is some map $\gamma: P^* \ra F$ such that $\pi \circ \gamma = \beta^*$.  Then by Lemma~\ref{lem:hnat}, we have 
\begin{align*}
h_P \circ \beta &= \beta^{**} \circ h_M = (\pi \circ \gamma)^* \circ h_M 
= \gamma^* \circ \pi^* \circ h_M = \gamma^* \circ \alpha.
\end{align*}
By Lemma~\ref{lem:hPiso}, $h_P$ is an isomorphism.  Compose both sides of the above display on the left with $h_P^{-1}$.  Then $\beta = (h_P^{-1} \circ \gamma^*) \circ \alpha$.
\end{proof}

\begin{prop}    
Let $\cM$ be the class of \emph{finitely generated} $R$-modules.
The class of finitely generated free modules is pre-enveloping for $\cM$ in the class of \emph{flat} modules.
\end{prop}

\begin{proof}
Indeed, let $M$ be a finitely generated module, let $\alpha: M \ra F^*$ be as above, let $G$ be a flat $R$-module and let $\beta: M \ra G$ be an $R$-linear map.  By Lazard's theorem \cite{Laz-plat}, there is some directed system $\{G_i\}_{i \in \cI}$ of finitely generated free $R$-modules such that $G = \lim_{\rightarrow} G_i$.  Accordingly, let $g_{ij}: G_i \ra G_j$ be the maps in the directed system and let $g_i: G_i \ra G$ be the induced maps to the direct limit.  Let $\{x_1, \ldots, x_n\}$ be a generating set for the $R$-module $M$. For each $k$, there is some $i_k \in \cI$ and $y_k \in G_{i_k}$ such that $\beta(x_k) = g_{i_k}(y_k)$.  Choose $j \in \cI$ such that $j \geq i_k$ for $k=1, \ldots, n$.
Set $z_k := g_{i_k,j}(y_k) \in G_j$.  Now, let $q: R^n \onto M$, where $R^n$ has basis $\{e_1, \ldots, e_n\}$ be the surjective map that sends $e_k \mapsto x_k$.  Define $\beta': R^n \ra G_j$ to be the unique map such that $\beta'(e_k) = z_k$ for all $k=1, \ldots, n$.  Then $g_j(\beta'(e_k)) = g_j(z_k) = g_{i_k}(y_k) = \beta(x_k) = \beta(q(e_k))$.  Since this holds for all $k$, we have $g_j \circ \beta' = \beta \circ q$.  Let $K = \ker q$; choose a generating set $\{u_1, \ldots, u_t\}$ for $K$.  Then for each $1\leq \ell \leq t$, we have $g_j(\beta'(u_\ell)) = \beta(q(u_\ell)) = \beta(0) = 0$.  Again using direct limit properties, there is then some $h\geq j$ such that $\beta'(u_\ell) \in \ker g_{jh}$ for all $1 \leq \ell \leq t$.  That is, $g_{jh} \circ \beta'$ vanishes on $\ker q$, so there is an induced map $\gamma: M \ra G_h$ such that $\gamma \circ q = g_{jh} \circ \beta'$. But then $g_h \circ \gamma \circ q = g_h \circ g_{jh} \circ \beta' = g_j \circ \beta' = \beta \circ q$.  Then since $q$ is surjective, and hence right-cancellable, we have $\beta = g_h \circ \gamma$.
    
But since $G_h$ is a finitely generated free module, we already know by Proposition \ref{pr:fgpreenvproj} that there is some $\tilde \gamma: F^* \ra G_h$ such that $\gamma = \tilde\gamma \circ \alpha$.  Hence, $\beta = g_h \circ \gamma = g_h \circ \tilde \gamma \circ \alpha$, so that $\beta$ factors through $\alpha$ as required.
\end{proof}

\begin{prop}\label{pr:versalop}
Let $\cM$ be a category of modules, and let $\cC$ be a pre-enveloping subclass. Let $\cP'$ be a class of pairs $(L,M)$ such that $L,M \in \cM$ 
and the inclusion map $L \hookrightarrow M$ is in $\cM$ as well.  Assume that whenever $\phi: M \ra N$ is a morphism in $\cM$ and $(L,M) \in \cP'$, then $(\phi(L), N) \in \cP'$. Let $p$ be a functorial pair operation defined on $\cP=\cP' |_\cC := \{(L,C) \in \cP' \mid C \in \cC\}$.  Define the pair operation $p_{h(\cC)}$ on $\cP'$ so that when $\alpha: M \ra C$ is a $\cC$-preenvelope, $p_{h(\cC)}(L,M) = \alpha^{-1}(p(\alpha(L),C))$. 
Then $p_{h(\cC)}$ is well-defined and functorial.  Moreover,

\begin{enumerate}
    \item\label{p'opsub} If $p$ is order-preserving on submodules, then so is $p_{h(\cC)}$.
    \item\label{p'ext} If $p$ is extensive, then so is $p_{h(\cC)}$.
    \item\label{p'idem} If $p$ is a closure operation,
    then $p_{h(\cC)}$ is also idempotent. 
    \item\label{altp'} Let $(L,M) \in \cP'$ and $x\in M$.  Then $x\in p_{h(\cC)}(L,M)$ if and only if for every map $g: M \ra C$ in $\cM$ with $C \in \cC$, we have $g(x) \in p(g(L),C)$. 
\end{enumerate}
Hence, if $p$ is a closure operation, then so is $p_{h(\cC)}$.
\end{prop}

\begin{rem}
This result resembles the results in Section \ref{sec:hereditaryversions}. The key difference is that when we use injective modules as our pre-enveloping class, we don't need to assume $p$ is extensive in order to prove that if $p$ is idempotent, then so is $p_{h(\cC)}$. So this result and the results of Section \ref{sec:hereditaryversions} overlap, but each includes cases not covered by the other results. Note that $p_h$ as defined in previous sections is $p_{h(\cC)}$, where $\cC$ is the class of injective modules.
\end{rem}

\begin{proof}
For well-definedness, let $\alpha: M \ra C$ and $\beta: M \ra C'$ be $\cC$-preenveloping maps.  Then by definition, there exist maps $\tilde \alpha: C' \ra C$ and $\tilde \beta: C \ra C'$ such that $\alpha = \tilde \alpha \circ \beta$ and $\beta = \tilde \beta \circ \alpha$.  Let $x\in \alpha^{-1}(p(\alpha(L),C))$. That is, $\alpha(x) \in p(\alpha(L), C)$.  Then by functoriality of $p$, we have $\beta(x) = \tilde \beta( \alpha(x)) \in p(\tilde \beta(\alpha(L)), C') = p(\beta(L), C')$.  That is, $x \in \beta^{-1}(p(\beta(L), C'))$.  By symmetry, we get the reverse inclusion, and so $p_{h(\cC)}$ is well-defined.

For functoriality, let $(L,M) \in \cP'$, and let $\phi: M \ra N$ in $\cM$. Let $x\in p_{h(\cC)}(L,M)$.  Choose $\cC$-preenveloping maps $\alpha: M \ra C$ and $\beta: N \ra D$.  Then since $\alpha$ is a $\cC$-preenvelope, there is some $\tilde \phi: C \ra D$ such that $\beta \circ \phi = \tilde \phi \circ \alpha$.  We have $\alpha(x) \in p(\alpha(L),C)$, so functoriality of $p$ gives $\beta(\phi(x))= \tilde \phi(\alpha(x)) \in p(\tilde \phi(\alpha(L)), D) = p(\beta(\phi(L)), D).$  Hence by definition, $\phi(x) \in p_{h(\cC)}(\phi(L), N)$.

Now suppose $p$ is order-preserving on submodules. Let $L \subseteq N \subseteq M$ be submodule inclusions in $\cM$ with $(L,M), (N,M) \in \cP'$.  Let $\alpha:M \ra C$ be a $\cC$-preenvelope.  Let $x \in p_{h(\cC)}(L,M)$.  Then $\alpha(x) \in p(\alpha(L), C) \subseteq p(\alpha(N), C)$ (since $\alpha(L) \subseteq \alpha(N)$ and $p$ is order-preserving on submodules).  Hence, $x \in \alpha^{-1}(p(\alpha(N),C)) = p_{h(\cC)}(N,M)$, whence $p_{h(\cC)}$ is order-preserving on submodules.

Next, suppose $p$ is extensive.  Let $(L,M) \in \cP'$ and let $\alpha: M \ra C$ be a $\cC$-preenvelope.  Let $x \in L$.  Then $\alpha(x) \in \alpha(L) \subseteq p(\alpha(L),C)$ (since $p$ is extensive).  Thus, $x\in \alpha^{-1}(p(\alpha(L), C)) = p_{h(\cC)}(L,M)$, whence $p_{h(\cC)}$ is extensive.

Next, suppose $p$ is extensive \opsub, and idempotent.  Let $(L,M) \in \cP'$ such that $(p_{h(\cC)}(L,M),M) \in \cP'$ as well.  Let $\alpha: M \ra C$ be a $\cC$-preenvelope and let $x\in p_{h(\cC)}(p_{h(\cC)}(L,M),M)$.  Then 
\begin{align*}
\alpha(x) &\in  p(\alpha(p_{h(\cC)}(L,M)),C)=p(\alpha(\alpha^{-1}(p(\alpha(L),C))),C) \\
&\subseteq p(p(\alpha(L),C),C)=p(\alpha(L),C),
\end{align*}
where the containment follows since $p$ is extensive and the last equality follows by the idempotence of $p$.
Hence $x \in p_{h(\cC)}(p_{h(\cC)}(L,M),M)$.

Finally, we prove (\ref{altp'}).  Let $x\in M$, and suppose $x \in p_{h(\cC)}(L,M)$.  Let $g: M \ra C$ be as in (\ref{altp'}).  Let $h: M \ra B$ be a $\cC$-preenvelope.  Then by definition, $h(x) \in p(h(L),B)$.  But by the pre-enveloping property, there is some $\tilde g: B \ra C$ in $\cM$ with $\tilde g \circ h = g$.  Hence, by functoriality of $p$, we have $g(x) = \tilde g(h(x))  \in p(\tilde g(h(L)), C) = p(g(L),C)$.  The converse is trivial.
\end{proof}

\section{Applications to integral closure, tight closure, and other closures defined using preenvelopes}
\label{sec:integralclosure}

In this section, we apply our results on residual and hereditary versions to integral closure and tight closure. We discover that EHU-integral closure is not the hereditary version of liftable integral closure and show that we can define a version of tight closure using a preenveloping class.

This version of integral closure arose in the context of \emph{Rees algebras of modules}:   \begin{defn}[{\cite[pp. 702-703]{EHU-Ralg}}]
Let $M$ be an $R$-module.  Then the \emph{Rees algebra} of $M$ is defined to be $\cR(M) := \Sym(M) / \bigcap_{g: M \ra F} \{\ker \Sym(g)\}$, where the intersection is taken over all maps $g: M \ra F$ where $F$ is free.

If $g: M \ra F$ is a map with $F$ free, then we set $\cR(g) := \im \Sym(g)$.

A map $g: M \ra F$, with $F$ free and $M,F$ finitely generated, is \emph{versal} if for any map $h: M \ra G$ to a free module $G$, there is some $R$-linear map $\phi: F \ra G$ such that $\phi \circ h=g$.
\end{defn}

\begin{lemma}[{\cite[Lemma 1.5]{SS-intrinsicRees}}]\label{lem:versal}
If $g: M \ra F$ is versal, then $\cR(M) \cong \cR(g)$.
\end{lemma}

\begin{thm}[{\cite[from Theorem 2.2]{EHU-Ralg}}]\label{thm:EHUic}
Let $R$ be a Noetherian ring, let $U \subseteq L \subseteq M$ be finite $R$-modules and let $f: M \ra F$ be a versal map with $F$ free.  The following are equivalent: \begin{enumerate}
    \item $L$ is integral over $U$ in $M$ (i.e. $L \subseteq \EHUic UM$).
    \item For every minimal prime $\p$ of $R$, $L'$ is integral over $U'$ in $M'$ over the ring $R/\p$, where $'$ denotes images in $F/\p F$.
    \item For every homomorphism $M \ra G$ to a free module and every ring map $R \ra S$ with $S$ a domain, the $S$-module $L'$ is integral over $U'$ in $M'$, where $'$ denotes tensoring with $S$ and taking images in $S \otimes_R G$.
\end{enumerate}

\end{thm}

\begin{rem}[{\cite[top of p. 702]{EHU-Ralg}}]
Let $M$ be a finitely generated $R$-module, $F$ a finitely generated free module, and $\alpha: M \ra F$ an $R$-linear map.  Consider $F$ to be the degree 1 component of the graded $R$-algebra $A=\Sym(F)$ in the usual way.  Then $\cR(\alpha)$ is the (graded) $R$-subalgebra of $A$ generated by $\alpha(M) \subseteq F = A_1$.  One can also think of this as the image of the induced map $\Sym(\alpha): \Sym(M) \ra \Sym(F)$.
\end{rem}

\begin{lemma}\label{lem:Rees-surj}
Let $M$ be a finite $R$-module, let $F,G$ be finite free $R$-modules, and let $\alpha: M \ra F$ and $\mu: F \ra G$ be $R$-linear.  Then $\mu_*(\cR(\alpha)) = \cR(\mu \circ \alpha)$, where $\mu_* := \Sym(\mu)$.
\end{lemma}

\begin{proof}
Every positive-degree element of $\cR(\alpha)$ is a finite $R$-linear combination of terms of the form $\gamma=\prod_{j=1}^t \alpha(m_j)$, $t\in \N$, $m_j \in M$. For any such $\gamma$, we have $\mu_*(\gamma) = \prod_{j=1}^t \mu(\alpha(m_j)) = \prod_{j=1}^t (\mu \circ \alpha)(m_j) \in \cR(\mu \circ \alpha)$.

Conversely, every positive-degree element of $\cR(\mu \circ \alpha)$ is an $R$-linear combination of terms of the form $\beta = \prod_{j=1}^t (\mu \circ \alpha)(m_j)$.  But then $\beta = \mu_*(\prod_{j=1}^t \alpha(m_j)) \in \mu_*(\cR(\alpha))$.
\end{proof}

\begin{prop}\label{pr:EHUishereditary}
The Eisenbud-Huneke-Ulrich (EHU) version of integral closure is both functorial and hereditary as a map on finitely generated modules.  That is, let $R$ be a Noetherian ring, and $L \subseteq N \subseteq M$ finitely generated $R$-modules.  Then $N$ is integral over $L$ (as given in Definition \ref{def:EHU}) if and only if it is integral over $L$ in $M$.
\end{prop}

\begin{proof}
To see functoriality, let $L \subseteq M$ and $M'$ be finitely generated modules and $g: M \ra M'$ an $R$-linear map.  Let $F$ be a finitely generated free module and $\phi: M' \ra F$ an $R$-linear map. Suppose $x \in \EHUic L{M}$.  Then since $\phi\circ g$ is a map from $M$ to a free module, we have $\phi(g(x)) \in \phi(g(L))^-_F$.  Since $\phi$ was arbitrary, it follows that $g(x) \in \EHUic {g(L)} {M'}$.

In particular, if $L \subseteq N \subseteq M$ and $N$ is integral over $L$ in the EHU sense (that is $N = \EHUic LN$), then $N \subseteq \EHUic LM$ as well by functoriality applied to the inclusion map $N \hookrightarrow M$. 

For the reverse direction, note first that by applying Theorem~\ref{thm:EHUic}(2), we may assume $R$ is an integral domain.  
Next, let $\alpha: M \ra G$ and $\beta: N \ra F$ be versal maps to finitely generated free modules.  Such maps exist by \cite[Proposition 1.3]{EHU-Ralg}. Then there is an induced map $\mu: F \ra G$ with $\mu \circ \beta = \alpha |_N$.  Let $i: L \into N$ and $j: N \into M$ be the inclusion maps.  We are given that $N$ is integral over $L$ in $M$ in the EHU sense.  By definition, we then have that the ring $R[N'_G]=\cR(\alpha \circ j)$ is integral over the subring $R[L'_G] = \cR(\alpha \circ j \circ i)$, where for a submodule $K$ of $M$, $K'_G$ is the image of $K$ in $\Sym(G)$. 

Next, we claim that $\mu \circ \beta=\alpha \circ j: N \ra G$ induces an injection on the torsionless quotient $N/\ker(h_N)$ of $N$.  To see this, let $x\in \ker (\alpha \circ j)$.  Then $j(x) \in \ker \alpha = T(M)$, the torsion submodule of $M$.  Since $R$ is a domain, ``torsionless'' and ``torsion-free'' coincide, and we have $T(M) \cap N = T(N)$.  Hence, $x\in T(N)$. 

 Thus, by \cite[Proposition 1.8]{EHU-Ralg}, the kernel of the natural map $\phi: \cR(N) = \cR(\beta) \rightarrow \cR(\mu \circ \beta) = \cR(\alpha \circ j) = R[N'_G]$ is nilpotent.

Let $a\in N$.  Since $a\in \EHUic LM$, we have $\alpha(j(a)) \in (\alpha(j(L)))^-_G$.   That is, $\mu(\beta(a)) \in (\mu(\beta(L)))^-_G$.  Thus, as an element of $\Sym(G)$, $\mu(\beta(a))$ is integral over the subalgebra $R[L'_G] = \cR(\alpha \circ j \circ i) = \cR(\mu\circ \beta \circ i)$. That is, there exist $n \in \N$ and $h_t \in \cR(\mu \circ \beta \circ i)$ with \[
\mu(\beta(a))^n + \sum_{t=1}^n h_t \mu(\beta(a))^{n-t} = 0.
\]
But $\mu_*=\Sym(\mu)$ restricts to a surjective map $\cR(\beta \circ i) \onto \cR(\mu \circ \beta \circ i)$ by Lemma~\ref{lem:Rees-surj}.  Hence, there exist $g_t \in \cR(\beta \circ i)$ with $h_t= \mu_*(g_t)$.  Thus, we have \[
\mu_*\left(\beta(a)^n + \sum_{t=1}^n g_t \beta(a)^{n-t} \right) = \mu(\beta(a))^n + \sum_{t=1}^n \mu_*(g_t) \mu(\beta(a))^{n-t} = 0.
\]
Since $\beta(a) \in \beta(N) \subseteq \cR(\beta)$ and each $g_t \in \cR(\beta \circ i) \subseteq \cR(\beta)$, we can replace ``$\mu_*$'' in the above display with ``$\phi$''.
But since the kernel of $\phi$ is nilpotent, there is some positive integer $m$ with $(\beta(a)^n + \sum_{t=1}^n g_t \beta(a)^{n-t})^m = 0$.  The latter is then an equation of integrality of $\beta(a)$ over $\cR(\beta \circ i) = R[L'_F]$.  Thus, $\beta(a) \in \beta(L)^-_F$.  Then by the implication ``$2 \Rightarrow 1$'' of Theorem~\ref{thm:EHUic}, since $\beta$ is versal, we have $a \in \EHUic LN$.  Thus, $N=\EHUic LN$.
\end{proof}

The next result shows that EHU-integral closure must agree with its hereditary version, even though they appear to be defined differently.

\begin{cor}
\label{cor:ehuhereditaryversion}
Let $R$ be Noetherian. Extend EHU-integral closure to submodules of all $R$-modules by taking its finitistic version as in Definition \ref{def:finitistic}. The hereditary version of this closure is equal to EHU-integral closure on \fg\ $R$-modules.

In fact, this holds for any hereditary closure operation defined on \fg\ $R$-modules.
\end{cor}

\begin{proof}
First note that EHU-integral closure is hereditary on pairs of \fg\ $R$-modules by Proposition \ref{pr:EHUishereditary}. By Proposition \ref{pr:finitisticextensionhereditary}, the finitistic extension of EHU-integral closure is a hereditary closure operation on pairs of $R$-modules, hence equal to its hereditary version on all $R$-modules. Since Lemma \ref{lem:finitisticequalsorigonfg} implies that the finitistic extension of EHU-integral closure agrees with EHU-integral closure on \fg\ $R$-modules, the result is proved.

The proof for arbitrary hereditary closure operations follows in the same way.
\end{proof}

In contrast, when $R$ is not a domain, Rees integral closure is not always hereditary, even on \fg\ modules. When $R$ is a domain, Rees integral closure agrees with EHU-integral closure, so we do not study Rees integral closure separately in this paper, beyond the example below.

\begin{example}
We give an example to show that if $R$ is not a domain, Rees integral closure may not be hereditary. Let $R=k[x]/(x^2)$ and set $M=Rx$. By 
\cite[Example 1.6]{nmeUlr-lint}, $0_M^{\rm -Rs}=0$. Since $R$ is Artinian Gorenstein, it is injective as a module over itself. If Rees integral closure were hereditary, we should get $0_R^{\rm -Rs} \cap M=0$, but $0_R^{\rm -Rs}=(x)=M$. So Rees integral closure is not hereditary in this case.
\end{example}

The next example demonstrates that the hereditary version of liftable integral closure can differ from EHU-integral closure. However, we will see below that these two closures agree on sufficiently nice modules.

\begin{example}
\label{ex:hereditarydisagreeintegral}
Let $R=k[\![x]\!]$, where $k$ is any field.  Let $M=k$ be the residue field of $R$, and set $L=0 \subseteq M$.  Then since $\depth R>0$, we have $\Hom_R(k,R) = 0$, so that for any finitely generated projective (hence free) $R$-module $P\cong R^t$, we have $\Hom_R(M,P) \cong \Hom_R(k,R)^t = 0$.  That is, the only map $M \ra P$ is the zero map, so it factors through the map $M \ra 0$.  Hence the map $M \ra 0$ is versal, so $\EHUic 0 M = M$.


However, the hereditary version of the liftable integral closure of $0$ in $M$ is equal to 0. To see this, let $E$ be an injective module containing $M$ and let $N$ be a finitely generated module with $M \subseteq N \subseteq E$.  Let $G$ be a finite free module and $\pi: G \onto N$ be a surjection.  Then $\ker \pi$ is an integrally closed $R$-submodule of $G$, since any submodule of a finitely generated free module over a discrete valuation ring must be integrally closed \cite[Corollary 16.3.3]{integralclosure}.  Hence $\lic 0 N = \pi( (\ker \pi)^-_G) = \pi(\ker \pi) = 0$.  Since $\li$ is finitistic \cite[Lemma 2.3]{nmeUlr-lint}, it follows that $\lic 0 E = 0$, so that $0^{\li_{\rm {h}}}_M = \lic 0 E \cap M = 0$.

\end{example}

\begin{prop}
\label{pr:torsionlessagree}
Let $M$ be a torsionless \fg\ $R$-module.  Then the hereditary version of liftable integral closure agrees with EHU-integral closure on submodules of $M$.
\end{prop}

\begin{proof}
 Since liftable integral closure is generally smaller than EHU-integral closure \cite[Proposition 2.4 (6) and Corollary 1.5]{nmeUlr-lint}, we only need to prove one containment. Let $u \in \EHUic L M$. Since $M$ is contained in a \fg\ free module $F$, $u \in \EHUic LF=L_F^{\li}$. Let $E$ be an injective module containing $F$. Since liftable integral closure is functorial, $u \in L_E^{\li}$. Since $E$ is injective, this is equal to $L_E^{\po{\li}{\h}}$. So $u \in L_E^{\po{\li}{\h}} \cap M=L_M^{\po{\li}{\h}}$, as desired.
\end{proof}

In the next result, we give a condition on a closure operation $\cl$ that is sufficient to guarantee that the closure operation $\cl_{h(\cQ)}$ as in Proposition \ref{pr:versalop}, is hereditary when the pre-enveloping class $\cQ$ consists of \fg\ projective modules.

\begin{prop}
\label{pr:fgaspreenveloping}
Let $R$ be a Noetherian integral domain.  Let $\cl$ be a functorial closure operation on the set of pairs $\cP := \{(M,P) \mid P $ is a finitely generated projective $R$-module and $M$ is a submodule of $P\}$. Suppose that whenever $\alpha: M \hookrightarrow P$ and $\beta: M \hookrightarrow Q$ are injective maps where $P,Q$ are finitely generated projectives, that for any submodule $L$ of $M$, we have $\alpha^{-1}(\alpha(L)^\cl_P) = \beta^{-1}(\beta(L)^\cl_Q)$.  Then the closure operation $\cl_{h(\cQ)}$ on $\cP' := \{(L,M) \mid M \text{ is finitely generated}\}$, as in Proposition~\ref{pr:versalop}, is hereditary. Consequently, it is equal to its hereditary version as in Proposition \ref{pr:definehereditaryversion}.
\end{prop}

\begin{proof}
Let $L \subseteq N \subseteq M$ be finitely generated $R$-modules. First, by functoriality (as in Proposition~\ref{pr:versalop}) applied to the inclusion map $i: N \hookrightarrow M$, we have $L^{\cl_{h(\cQ)}}_N \subseteq L^{\cl_{h(\cQ)}}_M \cap N$, so we need only prove the reverse containment.

For this, assume first that $M$ is torsionless.  Then $M$ embeds in a finite free module, whence $N$ does as well, so that $N$ is also torsionless \cite[Exercise 1.4.20(b)]{BH}.  Let $\alpha: N \hookrightarrow F$ and $\beta: M \hookrightarrow G$ be pre-envelopes by the class of finitely generated projectives.  Then we have: \begin{align*}
L^{\cl_{h(\cQ)}}_M \cap N &= i^{-1}(i(L)^{\cl_{h(\cQ)}}_M) = i^{-1}(\beta^{-1}(\beta(i(L))^\cl_G)) = (\beta \circ i)^{-1}((\beta \circ i)(L)^\cl_G)\\
&= \alpha^{-1}(\alpha(L)^\cl_F) = L^{\cl_{h(\cQ)}}_N,
\end{align*}
where the first equality on the second line is by assumption (since both $\beta \circ i$ and $\alpha$ are injective maps to finitely generated projective modules) and the final equality is because $\alpha$ is a pre-envelope.

Now we drop the assumption that $M$ is torsionless.  For this, recall   \cite[second paragraph of Introduction]{AusBr} that since $R$ is an integral domain, for any finite $R$-module $C$ the kernel of the natural biduality map $h_C: C \ra \Hom_R(\Hom_R(C,R),R)$ is precisely the torsion submodule $\tau(C)$ of $C$. Note also that $\tau(N) = i^{-1}(\tau(M))$. Hence, there is an induced injection $i_0: N/\tau(N) \ra M/\tau(M)$, so that $i_0 \circ \pi = q \circ i$, where $\pi: N \onto N/\tau(N)$ and $q: M \onto M/\tau(M)$ are the canonical projections.  Now, let $\mu: N \ra F$ be a \fg\ free pre-envelope.  Then since $\mu$ factors through $h_N$, it also factors through $\pi$, which is to say there is some map $\tilde \mu: N/\tau(N) \ra F$ such that $\tilde \mu \circ \pi = \mu$.  Moreover, the versality of $\mu$ implies that $\tilde \mu$ is injective.  To see this, let $q: P \onto \Hom_R(N,R)$ be a surjection from a finite free module.  Then $q^* := \Hom_R(q,R)$ is injective by left-exactness of $\Hom$, and $q^* \circ h_N$ factors through $\mu$, since $\mu$ is versal.  Thus, for any $x\in \ker \mu$, we have $x \in \ker (q^* \circ h_N) = \ker h_N = \tau(N)$.

Similarly, let $\gamma: M \ra G$ be a \fg\ free pre-envelope.  Then there is an injective map $\tilde{\gamma}: M/\tau(M) \ra G$ such that $\tilde{\gamma} \circ q = \gamma$.  
We now have the following commutative diagram: \[
\xymatrix{
N \ar@{^{(}->}[r]^i \ar@{->>}[d]^\pi \ar@/_3pc/[dd]_\mu &M\ar@{->>}[d]^q \ar@/^3pc/[dd]^\gamma\\
N/\tau(N) \ar@{^{(}->}[r]^{i_0} \ar@{^{(}->}[d]^{\tilde \mu} &M/\tau(M) \ar@{_{(}->}[d]^{\tilde{\gamma}} \\
F & G
}
\]
Now let $x \in L^{\cl_{h(\cQ)}}_M \cap N$.  That is, $x\in N$ and $i(x) \in i(L)^{\cl_{h(\cQ)}}_M$.  Then by construction of $\gamma$ and by definition of $\cl_{h(\cQ)}$, we have: \[
(\tilde{\gamma}\circ i_0)(\pi(x)) = \tilde{\gamma}(q(i(x))) = \gamma(i(x)) \in (\gamma(i(L)))^\cl_G = ((\tilde{\gamma} \circ i_0 \circ \pi)(L))^\cl_G.
\]
Since $\tilde{\gamma} \circ i_0$ and $\tilde \mu$ are both injective, we have by the first part of the proof (where we were assuming $M$ to be torsionless) that \[
\pi(x) \in \pi(L)^\cl_G \cap (N/\tau(N)) = \pi(L)^\cl_F \cap (N/\tau(N)) = {\tilde \mu}^{-1}(\tilde \mu(\pi(L))^\cl_F),
\]
so that \[
x \in \pi^{-1}(\tilde \mu^{-1}(\tilde \mu(\pi(L))^\cl_F)) = \mu^{-1}(\mu(L)^\cl_F) = L^{\cl_{h(\cQ)}}_N. \qedhere
\]
\end{proof}

Now we show that this condition holds for tight closure.

\begin{lemma}
Let $R$ be a prime characteristic reduced Noetherian ring.  Let $L \subseteq M$ be finitely generated $R$-modules, and let $\alpha: M \hookrightarrow F$, $\beta: M \hookrightarrow G$ be two different injective maps into finitely generated projective modules.  Then $L^*_F \cap M = L^*_G \cap M$.
\end{lemma}

\begin{proof}
Consider the pushout $H := F \oplus_M G$ of the maps $\alpha$ and $\beta$.  Label the induced maps $\alpha': F \ra H$ and $\beta': G \ra H$.  Note that both $\alpha'$ and $\beta'$ are injective.  Hence we have inclusions $L \subseteq F \subseteq H$ and $L \subseteq G \subseteq H$. 
Then by \cite[Proposition 8.18(b)]{HHmain}, we have $L^*_F = L^*_H \cap F$ and $L^*_G = L^*_H \cap G$. Thus: \[
L^*_F \cap M = L^*_H \cap F \cap M = L^*_H \cap M = L^*_H \cap G \cap M = L^*_G \cap M \qedhere
\]
\end{proof}

\begin{cor}
Let $R$ be a prime characteristic Noetherian domain.  Let $*$ denote the tight closure operation, restricted to submodules of finitely generated projective modules.  Then the operation $*_{h(\cQ)}$, as in Proposition~\ref{pr:fgaspreenveloping}, is a hereditary closure operation on finitely generated modules that agrees with $*$ when the ambient module is projective (e.g. free). In particular, this means that it agrees with its hereditary version as in Proposition \ref{pr:definehereditaryversion}.
\end{cor}

This gives us one way to get a hereditary version of tight closure. Absolute tight closure (see Definition \ref{def:abstightclosure}) is another one:

\begin{lemma}
Absolute tight closure equals $*_{\rm h}$ on \fg\ modules over any Noetherian ring $R$ of prime characteristic $p>0$.  In particular, absolute tight closure is hereditary in this context.
\end{lemma}

\begin{proof}
Let $N \subseteq M$ be finitely generated $R$-modules.
By \cite[8.21]{HHmain}, $N^{*abs}_M = N^{*fg}_E \cap M$ when $E$ is any injective module containing $M$.  By Proposition~\ref{pr:hereditaryorderpreserving}, $(*fg)_h$ is hereditary.  But $N^{(*fg)_h}_M = N^{*fg}_E \cap M$.
%
\end{proof}

As we had for integral closure, we now have 2 methods for creating a hereditary version of tight closure, one involving the pre-enveloping class of injective modules, and one involving the pre-enveloping class $\cQ$ of \fg\ projective modules. As in that case, the two methods differ on some examples, but agree for sufficiently nice modules.

\begin{example}
We repeat Example \ref{ex:hereditarydisagreeintegral}. The argument that $0_M^{*_{h(\cQ)}}=M$ is the same. However, since $R$ is regular, for any module $U \supseteq M$, $0_U^{*fg}=0$. So $0_M^{*abs}=0$.
\end{example}

\begin{prop}
    Let $R$ be a prime characteristic Noetherian domain. If $M$ is a torsionless \fg\ $R$-module and $L \subseteq M$, then $L_M^{*_{h(\cQ)}}=L_M^{*abs}$, where $\cQ$ is the pre-enveloping class of \fg\ projective modules.
\end{prop}

\begin{proof}
    The proof is identical to the proof of Proposition \ref{pr:torsionlessagree}.
\end{proof}

\section*{Acknowledgments}

We thank Abdullah Alshayie for finding some minor errors in this and other papers of ours.  We are also grateful for extensive comments made by the anonymous referee which improved the paper.

\providecommand{\bysame}{\leavevmode\hbox to3em{\hrulefill}\thinspace}
\providecommand{\MR}{\relax\ifhmode\unskip\space\fi MR }
\providecommand{\MRhref}[2]{%
  \href{http://www.ams.org/mathscinet-getitem?mr=#1}{#2}
}
\providecommand{\href}[2]{#2}

\end{document}